\documentclass[12pt,a4paper]{amsart}

\usepackage{amssymb}
\usepackage[all]{xy}

\newcommand{\CC}{{\mathbb{C}}}
\newcommand{\FF}{{\mathbb{F}}}
\newcommand{\GG}{{\mathbb{G}}}
\newcommand{\NN}{{\mathbb{N}}}
\newcommand{\QQ}{{\mathbb{Q}}}
\newcommand{\ZZ}{{\mathbb{Z}}}

\newcommand{\bk} {\mathbf k}
\newcommand{\bm} {\mathbf m}
\newcommand{\bw} {\mathbf w}

\newcommand{\bG} {\mathbf G}
\newcommand{\bL} {\mathbf L}
\newcommand{\bT} {\mathbf T}
\newcommand{\bU} {\mathbf U}
\newcommand{\bX} {\mathbf X}

\newcommand{\cA} {\mathcal A}
\newcommand{\cC} {\mathcal C}
\newcommand{\cE} {\mathcal E}
\newcommand{\cF} {\mathcal F}
\newcommand{\cG} {\mathcal G}
\newcommand{\cL} {\mathcal L}
\newcommand{\cN} {\mathcal N}

\newcommand{\cW} {\mathcal W}

\newcommand{\fF} {\mathfrak F}
\newcommand{\fS} {\mathfrak S}
\newcommand{\fc} {\mathfrak c}

\newcommand{\coT}{T}
\newcommand{\coX}{X}
\newcommand{\coY}{Y}

\newcommand{\Aut}{{{\operatorname{Aut}}}}
\newcommand{\CR}{{{\operatorname{cr}}}}
\newcommand{\diag}{{{\operatorname{diag}}}}
\newcommand{\Gal}{{{\operatorname{Gal}}}}
\newcommand{\Hom}{{{\operatorname{Hom}}}}
\newcommand{\Id}{{{\operatorname{Id}}}}
\newcommand{\inc}{{{\operatorname{inc}}}}
\newcommand{\Ind}{{{\operatorname{Ind}}}}
\newcommand{\Inj}{{{\operatorname{Inj}}}}
\newcommand{\Inn}{{{\operatorname{Inn}}}}
\newcommand{\IBr}{{{\operatorname{IBr}}}}
\newcommand{\Irr}{{{\operatorname{Irr}}}}
\newcommand{\map}{{{\operatorname{Map}}}}
\newcommand{\Mor}{{{\operatorname{Mor}}}}
\newcommand{\Out}{{{\operatorname{Out}}}}
\newcommand{\refl}{{{\operatorname{ref}}}}
\newcommand{\rk}{{{\operatorname{rk}}}}
\newcommand{\Stab}{{{\operatorname{Stab}}}}
\newcommand{\Tr}{{{\operatorname{Tr}}}}
\newcommand{\Uch}{{{\operatorname{Uch}}}}
\newcommand{\GL}{\operatorname{GL}}
\newcommand{\Sp}{\operatorname{Sp}}
\newcommand{\SL}{\operatorname{SL}}

\newcommand{\OO}{\operatorname{O}}

\newcommand{\tw}[1]{{}^{#1}\!}
\newcommand{\twt}{{\tw\tau}}
\newcommand{\tga}{\tilde\gamma}
\newcommand{\tq}{\tilde q}
\newcommand{\brT}{\breve T}

\let\eps=\varepsilon
\let\ga=\gamma
\let\vhi=\varphi
\let\la=\lambda
\let\lra=\longrightarrow

\newtheorem{thm}{Theorem}[section]
\newtheorem{lem}[thm]{Lemma}
\newtheorem{cor}[thm]{Corollary}
\newtheorem{prop}[thm]{Proposition}
\newtheorem{conj}[thm]{Conjecture}

\newtheorem{thmA}{Theorem}
\newtheorem{conjA}{Conjecture}

\theoremstyle{definition}
\newtheorem{rem}[thm]{Remark}
\newtheorem{defn}[thm]{Definition}
\newtheorem{exmp}[thm]{Example}

\begin{document}

%%%%%%%%%%%%%%%%%%%%%%%%%%%%%%%%%%%%%%%%%%%%%%%%%%%%%%%%%%%%
\title{Weight conjectures for $\ell$-compact groups and spetses}
%%%%%%%%%%%%%%%%%%%%%%%%%%%%%%%%%%%%%%%%%%%%%%%%%%%%%%%%%%%%

\author{Radha Kessar}
\address{Department of Mathematics,  University of Manchester, M13 9PL,
  United Kingdom}
\email{radha.kessar@manchester.ac.uk}

\author{Gunter Malle}
\address{FB Mathematik, TU Kaiserslautern, Postfach 3049,
  67653 Kaisers\-lautern, Germany.}
\email{malle@mathematik.uni-kl.de}

\author{Jason Semeraro}
\address{Department of
  Mathematics, Loughborough University,  LE11 3TT, United Kingdom}
\email{j.p.semeraro@lboro.ac.uk}

\thanks{
The authors would like to thank the Banff International Research Station for the
invitation to attend the workshop "Groups and Geometry" in August 2019 where
this work was initiated, College Court (Leicester University) for hosting the
workshop ``Topological Methods in Group Representation Theory'' in
November 2019 and the Isaac Newton Institute for Mathematical Sciences for
support and hospitality during the programme ``Groups, Representations and
Applications: New Perspectives'' during the Spring 2020 semester, supported by
EPSRC grant EP/R014604/1, when part of the work on this paper was undertaken. 
The first author gratefully acknowledges support from EPSRC grant EP/T004592/1.
The second author gratefully acknowledges support by the SFB TRR 195. The third
author acknowledges support from the Heilbronn Institute}



\date{\today}

\begin{abstract}
Fundamental conjectures in modular representation theory of finite groups,
more precisely, Alperin's weight conjecture and Robinson's ordinary weight
conjecture, can be expressed in terms of fusion systems. We use fusion systems
to connect the modular representation theory of finite groups of Lie type to the
theory of $\ell$-compact groups. Under some mild conditions we prove that the
fusion systems associated to homotopy fixed points of $\ell$-compact groups
satisfy an equation which for finite groups of Lie type is equivalent to
Alperin's weight conjecture.
\par
For finite reductive groups, Robinson's Ordinary weight conjecture is closely
related to Lusztig's Jordan decomposition of characters and the corresponding
results for Brauer $\ell$-blocks. Motivated by this, we define the principal
block of a spets attached to a spetsial $\ZZ_\ell$-reflection group, using the
fusion system related to it via $\ell$-compact groups, and formulate an analogue
of Robinson's conjecture for this block. We prove this formulation for an
infinite family of cases as well as for some groups of exceptional type. 

Our results not only provide further strong evidence for the validity of
the weight conjectures, but also point toward some yet unknown structural
explanation for them purely in the framework of fusion systems.
\end{abstract}

\maketitle

%%%%%%%%%%%%%%%%%%%%%%%%%%%%%%%%%%%%%%%%%%%%%%%%%%%%%%%%%%%%%%%%%%%%%%%%%
\section{Introduction and statement of results}
In this paper we connect fusion systems originating from homotopy fixed points
on $\ell$-compact groups with the representation theory of so-called spetses
by means of a natural generalisation of Alperin's weight
conjecture from modular representation theory of finite groups.
\par
Let $\ell$ be a prime. An \emph{$\ell$-compact group} is a topological object
that may loosely be regarded as a homotopical analogue of a compact Lie group,
built from an $\ell$-adic reflection group as Weyl group. It has been shown to
give rise to fusion systems on certain finite $\ell$-groups. In representation
theory, a \emph{spets} is a yet rather mysterious analogue of a finite reductive
group with a complex reflection group as Weyl group, for which some shadow of
Deligne--Lusztig character theory can be formulated. Apparent similarities
between these two classes of objects have led a number of authors to expect
that there should be a more formal connection between them. With this goal in
mind,
the third author \cite{S19} observed various numerical consistencies between
the exotic $2$-fusion systems associated to a $2$-compact group with Weyl group
the exceptional complex reflection group $G_{24}$ and an ad hoc defined set of
irreducible characters associated to the spets with the same Weyl group, in the
spirit of a generalisation of Alperin's weight conjecture. This famous
local-global conjecture from 1986 is an attempt to relate the character theory
of a finite group to the fusion system on its Sylow subgroups.
\par
In the present paper, we provide a formal theoretical context to these
observations. Thus the goals of the present paper are to:
\begin{itemize}
\item associate various global invariants to spetses (such as the principal $\ell$-block);
\item formulate analogues for spetses and $\ell$-compact groups of various
  local-global counting conjectures for groups;
\item prove these conjectures (in some cases);
\item highlight/explain techniques from algebraic topology relevant to the
  study of group representations.
\end{itemize}
\medskip

In order to motivate and explain our constructions, conjectures and results, we
begin by discussing global and local data associated to finite groups of Lie
type.

A \emph{finite group of Lie type} is the group of fixed points $\bG^F$ of a
connected reductive linear algebraic group $\bG$ under a Steinberg endomorphism
$F$ with respect to an $\FF_q$-structure for some prime power $q$ coprime
to~$\ell$. 
The ordinary irreducible characters of $\bG^F$ are constructed from the
$\ell$-adic cohomology groups of so-called Deligne--Lusztig varieties on which
$\bG^F$ acts. Lusztig has given a combinatorial parametrisation of these
characters by \textit{Lusztig series} $\cE(G,s)$ indexed by (classes of)
semisimple elements $s$ in the Langlands dual group, which is
purely in terms of the Weyl group $W$ of $\bG$. Furthermore, for any prime
$\ell$ different from the defining characteristic of $\bG$, the distribution
of these characters into Brauer $\ell$-blocks has also been shown to be
controlled by $W$. This is described by \emph{$e$-Harish-Chandra theory}, where
$e$ denotes the order of $q$ modulo~$\ell$.
For example, when $\ell$ is very good for $\bG$ then the unipotent
characters of $\bG^F$ lying in the principal $\ell$-block $B_0$ are those in
the \emph{principal} $e$-Harish-Chandra series \cite{BMM93}. Moreover,
the latter is in bijection with the irreducible characters of the corresponding
relative Weyl group denoted here $W_e$, a complex reflection group provided
by Lehrer--Springer theory.

The local data we consider for $\bG^F$ is most succinctly described in terms of
fusion systems. Recall that a \emph{fusion system $\cF$ on a finite $\ell$-group
$S$} is a category with objects the subgroups of $S$ and morphisms certain
injective group homomorphisms satisfying some weak axioms. It is called
\emph{saturated} if it satisfies two additional ``Sylow axioms'', derived from
the motivating example of the fusion system $\cF_\ell(G)$ induced by a finite
group $G$ on a Sylow $\ell$-subgroup. A subgroup $P\le S$ is
\emph{$\cF$-centric} if $C_S(Q)=Z(Q)$ for all $\cF$-conjugates $Q$ of $P$ and
\emph{$\cF$-radical} if $\Out_\cF(P)$ has no non-trivial normal
$\ell$-subgroup. The class $\cF^\CR$ of $\cF$-centric, $\cF$-radical subgroups
plays an important role in the local-global conjectures we consider. If $k$ is
an algebraically closed field of characteristic~$\ell$, $G$ is a finite group, 
$B_0$ is the principal $\ell$-block of $G$ and $\IBr(B_0)$ the set of
irreducible $\ell$-Brauer characters of $G$ lying in~$B_0$, \emph{Alperin's
weight conjecture} (AW conjecture) \cite{A87} claims the equality
$|\IBr(B_0)|=\bw(\cF_\ell(G))$ where, for a saturated fusion system $\cF$,
$$\bw(\cF):=\sum_{P\in\cF^\CR/\cF} z(k\Out_\cF(P)),$$
and the sum runs over $\cF$-conjugacy class representatives of $\cF$-centric,
$\cF$-radical subgroups. Here, for a finite group $H$, $z(kH)$ denotes the
number of isomorphism classes of projective simple $kH$ modules.

In the case $\bG^F$ is a finite group of Lie type with Weyl group $W$ then
assuming $\ell$ is very good for $\bG$, and $q$ of order $e$ modulo~$\ell$,
then $|\IBr(B_0)|$ equals  $|\Irr(W_e)|$ (see Proposition~\ref {prop:linkawc}). 
Hence, the above discussion implies that the AW conjecture for the principal
$\ell$-block is equivalent to the assertion
\begin{equation}   \label{e:awc1}
  \bw(\cF_\ell(\bG^F))=|\Irr(W_e)|\,.
\end{equation}
We wish to argue that the left hand side of this equality is also ``generic''
in the Weyl group $W$. For this, we require some algebraic topology,
specifically the theory of $\ell$-compact groups.

To any $\ell$-compact group $X$ is associated a reflection group $W$ on a
$\ZZ_\ell$-lattice $L$. If $X$ is connected, then any prime power $q$ coprime to
$\ell$ determines a self-equivalence $\psi^q$ of $X$, unique up to homotopy,
called an unstable Adams operator. If $X$ is moreover simply connected and
$\ell$ is odd, then Broto--M\o ller have shown in \cite{BM07} that for any
self-equivalence $\tau$ of $X$ whose image in the outer automorphism group of
$X$ is of $\ell'$-order, the space of homotopy fixed points under $\tau \psi^q$
is the classifying space of an $\ell$-local finite group, and in particular the
pair $(q,\tau)$ determines a saturated fusion system $\cF(\twt X(q))$ on a
finite $\ell$-group (see Theorem~\ref{thm:bm}). In Theorem~\ref{thm:sylow} we
describe the underlying $\ell$-group of this fusion system. Moreover
$(X,\tau,q)$ determines a certain relative Weyl group $W_e$ %:=W_{\phi \zeta}$
where $e$ is the
order of $q$ modulo $\ell$ described in Theorem~\ref{thm:homfixedpts}. If $W$ is
rational and $\bG^F$ is a group of Lie type associated to $W$ as above, results
of Friedlander--Mislin \cite{FM84} and Quillen \cite{Qui72} imply that for
suitable $X$ and $\tau$, $\cF_\ell(\bG^F)=\cF(\twt X(q))$ (see
Remark~\ref{r:fried}), and in this case the equality~(\ref{e:awc1}) is
equivalent to
$$\bw(\cF(\twt X(q)))=|\Irr(W_e)|.$$
In view of this, it becomes tempting to ask whether the same equality is true
without the assumption that $W$ is rational, so that $\cF(\twt X(q))$ is a
potentially \emph{exotic} fusion system. Our first local-global result is the
following generalisation of the AW conjecture (see
Theorem~\ref{thm:nr weights}); here, the
new notion of \emph{very good prime} for a $\ZZ_\ell$-reflection group is
defined in terms of its maximal rank subgroups, see Definition~\ref{def:bad}.

\begin{thmA}[AW conjecture for $\ell$-compact groups]   \label{thmb}
 Let $\ell>2$, $X$ be a simply connected $\ell$-compact group with Weyl group
 $(W,L)$, $q$ a prime power prime to $\ell$ and $\tau$ an automorphism of $X$
 whose image in the outer automorphism group of $X$ has finite order prime to
 $\ell$. If $\ell$ is very good for $(W,L)$, then
 $$\bw(\cF(\twt X(q))) = |\Irr(W_e)|,$$
 where $e$ is the order of $q$ modulo~$\ell$.
\end{thmA}

From this, we recover the validity of the AW conjecture for the principal $\ell$-blocks of
simply connected groups of Lie type, in particular the previously unknown
cases of types $E_6,E_7$ and $E_8$, for all very good primes $\ell>2$. For the
infinite series, we prove Theorem~\ref{thmb} by developing an equivariant
version of the Alperin--Fong proof of the AW conjecture for general linear
groups \cite{AF90}. We also have results for cases when $\ell$ is not very good,
viz.~for the Aguad\'e groups (see Theorem~\ref{thm:Aguade weights}).

Now let us return to the above situation where $\bG$ is a connected reductive
group. If $\ell$ is very good for $\bG$ then the ordinary irreducible
characters $\Irr(B_0)$ in the principal block are a union of $e$-Harish-Chandra
series associated to $\ell$-elements of $\bG^F$. We seek to describe this union
purely in terms of the Weyl group $W$ of $\bG$.

To this end, we require a notion of $\ell$-subgroups and their centralisers in
$\ell$-compact groups. Fortunately, owing to historical interest in homotopy
decompositions of classifying spaces, such a theory is available to us. In
particular, centralisers of $\ell$-subgroups of $\ell$-compact groups are again
$\ell$-compact groups. In Theorem~\ref{thm:con cent} we prove the following,
extending \cite[Thm.~1.9]{AGMV} which deals with the case $|Q|=\ell$:
 
\begin{thmA}   \label{thmc}
 Suppose that $X$ is a connected $\ell$-compact group with torsion-free
 fundamental group, $Q$ is a finite cyclic $\ell$-group and $f:BQ\rightarrow X$
 is a morphism. Then the centraliser $C_X(Q,f)$ is connected.
\end{thmA}

We provide a conceptual argument when the Weyl group has order prime to $\ell$;
the general case requires a case-by-case approach as in the case $|Q|=\ell$.
Jesper Grodal has communicated another proof to us relying on properties
of Lannes $T$-functor. 

Now let $q$ be a prime power as above and $\cF=\cF(\twt X(q))$ be the fusion
system associated to $X$, $q$ and $\tau$ as above and let $S$ be the underlying
$\ell$-group of $\cF$. Each $s\in S$ gives rise to a centraliser
space $C_X(s)$ defined in terms of a fixed embedding of the classifying space
$BS$ of $S$, with Weyl group denoted $W(s)$. If $W$ is spetsial, and $\GG$
denotes the $\ZZ_\ell$-spets associated to $W$ and the automorphism of~$W$
induced by $\tau$ (see Section~6), then we invoke Theorem~\ref{thmc} to construct a
$\ZZ_\ell$-spets $C_\GG(s)$. Now attached to any $\ZZ_\ell$-spets is a
collection of \emph{unipotent characters} whose degrees are polynomials over
$\QQ_\ell$. By analogy with Lusztig's theory, we define a set of characters
$\cE(\GG(q),s)$ in bijection with the unipotent characters
$\Uch(C_\GG(s))$. Motivated by the situation for finite reductive groups, if
$\ell$ is very good for $(W,L)$ we define the \emph{principal $\ell$-block}
$$\Irr(B_0):=\coprod_{s\in S/\cF} \cE(\GG(q),s)_1$$
of $\GG$, where $\cE(\GG(q),s)_1$ denotes the subset of $\cE(\GG(q),s)$ in
bijection with the principal $e$-Harish-Chandra series of $\Uch(C_\GG(s))$, and
the union runs over $\cF$-conjugacy class representatives of elements in $S$.
We also provide a degree formula for such characters generalising Lusztig's
Jordan decomposition formula for characters of rational spetses.    
In further analogy with finite group theory, we call $\cF$ the fusion system
of $\GG(q)$ and denote it by $\cF(\GG(q))$.  

If $G$ is a finite group and $\chi\in\Irr(G)$, the $\ell$-adic valuation
$\nu_\ell(|G|/\chi(1))$ is called the \emph{defect} of $\chi$. Let $\Irr^d(G)$
be the subset of irreducible characters of defect $d$. For the principal
$\ell$-block $B_0$ of $G$, Robinson's \emph{ordinary weight conjecture} (OW
conjecture) (see \cite{Rob96}, and also \cite[5.49]{AKO11}) claims the equality
$|\Irr^d(B_0)|:= |\Irr^d(G) \cap \Irr(B_0)|=\bm(\cF_\ell(G),d)$ where, for a
saturated fusion system $\cF$,
$$\bm(\cF,d):=\sum_{P\in\cF^{\CR}} \bw_P(\cF,d),$$
the sum runs over representatives of classes of $\cF$-centric, $\cF$-radical
subgroups and $\bw_P(\cF,d)$ is a certain alternating count of projective
simple modules associated to stabilisers of elements in $\Irr^d(P)$. We propose
the following analogue of the OW conjecture for spetses:

\begin{conjA}[OW conjecture for spetses]   \label{c:conjd}
 Let $\GG$ be a simply connected $\ZZ_\ell$-spets for which $\ell>2$ is very
 good, let $q$ be a power of a prime different from~$\ell$, $B_0$ be the
 principal block of $\GG(q)$ and $\cF(\GG(q))$ be the associated fusion
 system. Then
 $$|\Irr^d(B_0)|=\bm(\cF(\GG(q)),d)\qquad\text{for all $d\ge0$},$$
 where $\Irr^d(B_0)$ is the subset of characters in $\Irr(B_0)$ of
 $\ell$-defect $d$.
\end{conjA}

We show that the defects of characters in $\Irr(B_0)$ are locally controlled
(see Proposition~\ref{prop:heights}), so $|\Irr^d(B_0)|$ can be re-expressed in
terms of characters of the groups $W(s)$ defined above. Thus,
Conjecture~\ref{c:conjd} can be further generalised to a statement about
arbitrary simply connected $\ell$-compact groups (see
Conjecture~\ref{conj:nr oweights}). We give evidence for this by verifying it
when $|W|$ is prime to $\ell$ and for the smallest non-trivial infinite family
of reflection groups. This also supports the general positivity conjecture
(\cite[Conj.~2.5]{KLLS19}) that for any saturated fusion system~$\cF$,
$\bm(\cF,d)\ge 0$ for all $d\ge 0$. Our Conjecture \ref{c:conjd} also implies 
the ordinary weight conjecture for the principal $\ell$-block of $\bG^F$ for all
semisimple algebraic groups $\bG$ with Frobenius map $F$ whenever $\ell>2$ is
very good for $\bG$ and $F$ induces an $\ell'$-automorphism of $W$.
\medskip

We close the introduction with some remarks concerning the case $\ell=2$, the
restriction on the order of $\tau$ and the simply connected hypothesis in
Theorem~\ref{thmc} and Conjecture~\ref{c:conjd}. The
connected $2$-compact groups have been classified by Andersen--Grodal
\cite{AG09}. They show that the exotic $2$-compact group $\operatorname{DI}(4)$
associated to $G_{24}$, constructed by Dwyer--Wilkerson, is the only simple
$2$-compact group not arising as the $2$-completion of a compact connected Lie
group. In \cite{LS17}, Lynd and the third author show the conclusion of
Theorem~\ref{thmb} holds in this case. The paper \cite{S19} performs the
computations necessary to verify the conclusion of Conjecture~\ref{c:conjd}
holds for a suitably adapted definition of the principal block. Note that the
prime~$2$ is bad for $G_{24}$ by
Proposition~\ref{prop:bad}. The restriction on the order of $\tau$ as well as
the simply connected hypothesis come from our being in the setting of
\cite[Thm~A]{BM07} (see Theorem~\ref{thm:bm}). In many situations the
assumptions on \cite[Thm~A]{BM07} are either automatically satisfied (see
Remark~\ref{rem:order}) or are not necessary (see Remark~\ref{r:fried}). It
would be desirable to know to what extent these restrictions may be relaxed. 
 
\medskip
\noindent \textbf{Structure of the paper:}
In Section \ref{sec:refl} we collect background material and prove some new
results on reflection groups which are required for later sections. Among other
things, this section includes a fundamental property of stabilisers
(Proposition~\ref{prop:stab}) needed in the proof of Theorem~\ref{thmc}.
Section~\ref{sec:zltofusion} recalls the description of the fusion system
$\cF(\twt X(q))$ constructed in \cite{BM07}. We also derive, in
Section~\ref{subsec:sylows}, a description of its underlying $\ell$-group.
In Section~\ref{sec:awc}, we restate Theorem~\ref{thmb} as
Theorem~\ref{thm:nr weights}, outline its proof and explain how the AW
conjecture for
principal blocks of finite groups of Lie type for very good primes follows from
it (Corollary~\ref{cor:AWC Lie}). In Section~\ref{sec:centr}
we investigate homotopy fixed-point centralisers of $\ell$-compact groups and
prove Theorem~\ref{thmc} as part of Theorem~\ref{thm:con cent}. These results
allow us to define in Section~\ref{sec:owc} the centralisers of $\ell$-elements
and also the principal block of a $\ZZ_\ell$-spets. Sections~\ref{sec:GGr}
and~\ref{sec:except} complete the proof of Theorem~\ref{thmb} in the cases when
$X$ is of generalised Grassmannian and exceptional type respectively.
\medskip

\noindent{\bf Acknowledgement:}
The authors would like to thank Bob Oliver and Albert Ruiz for providing us
with references for Section 7, Carles Broto and Jesper M\o ller for clarifying
some points in \cite{BM07} and Bob Oliver and Jesper Grodal for helpful
discussions.

\setcounter{tocdepth}{1}
\tableofcontents

%%%%%%%%%%%%%%%%%%%%%%%%%%%%%%%%%%%%%%%%%%%%%%%%%%%%%%%%%%%%%%%%%%%%%%%%%
\section{$\ZZ_\ell$-Reflection groups}   \label{sec:refl}
In this section we recall some properties of finite reflection groups, set our
notation, show a crucial result on stabilisers and define the notion of good
and very good primes for $\ZZ_\ell$-reflection groups.

%%%%%%%%%%%%%%%%%%%%%%%%%%%%%%%%%%%%

Let $R$ be a principal ideal domain of characteristic zero with field of
fractions $K$. An \emph{$R$-reflection group} is a pair $(W,L)$ where $L$ is a
finitely generated free $R$-module and $W \le \Aut_R(L)$ is a \emph{finite}
group generated by reflections, i.e., non-trivial elements fixing point-wise an
$R$-submodule of corank $1$. We say that $(W,L)$ is \emph{irreducible} if the
$KW$-module $K \otimes_R L$ is irreducible. If $(W,L)$ is an $R$-reflection
group then so is $(W,L'):=(W,R'\otimes_R L)$ for any principal ideal domain
$R'$ containing $R$. In this case we say $(W,L')$ is obtained from $(W,L)$
by \emph{extension of scalars}. Two $R$-reflection groups $(W,L)$ and $(W',L')$
are isomorphic if there exists an $ R$-linear isomorphism $\vhi: L \to L'$ such
that $\vhi W \vhi^{-1} = W'$. 

%%%%%%%%%%%%%%%%%%%%%%%%%%%%%%%%%%%%
\subsection{Classification of reflection groups}
Let $\ell$ be a prime. A $\ZZ_\ell$-reflection group is
\emph{exotic} if it is irreducible with character field strictly containing
$\QQ$. Every $\ZZ_\ell$-reflection group $(W,L)$ decomposes as a direct product
\[(W, L) = (W_1 \times W_2, L_1 \oplus L_2), \]
where $(W_1,L_1)$ is an extension of scalars of a $\ZZ$-reflection group and
$(W_2,L_2)$ is a direct product of exotic $\ZZ_\ell $-reflection groups which
are uniquely determined up to permutation of factors (see, e.g.,
\cite[Thm.~11.1]{AGMV}).

Recall that every $\ZZ$-reflection group $(W,L)$ is isomorphic to
$(W_\bG,L_\bG)$ for some (not uniquely determined) compact connected Lie group
$\bG$ with Weyl group $W_\bG$ and cocharacter lattice $L_\bG$. By the Shephard--Todd
classification theorem (see e.g. \cite[Tab.~1]{GM06}) the irreducible complex
reflection groups fall into an infinite series of monomial groups $G(e,r,n)$
(with $n \geq 1$, $ e\geq 2 $, $r|e)$, the symmetric groups $\fS_n$ and $34$
exceptional groups (the rank of $L$ is at most~$8$ in the exceptional cases).

The irreducible $\ZZ_\ell$-reflection groups are also known, see e.g.,
\cite{AGMV}; for this paper we
find it convenient to subdivide them into the following six categories:
\begin{itemize}
\item the $\ZZ_\ell$-reflection groups of order prime to $\ell$ (also called
 \emph{Clark--Ewing groups}); 
\item the symmetric groups $\fS_n$ for $n\ge\ell$;
\item the imprimitive groups $G(e,r,n)$ with $n\ge\ell$, $e\ge2$ and
 $r|e|(\ell-1)$ (with associated $\ell$-compact groups called \emph{generalised
 Grassmannians} in \cite{BM07});
\item the five rational Weyl groups of exceptional types $G_2,F_4,E_6,E_7$ and
 $E_8$ for $\ell\in\{2,3\},\{2,3\},\{2,3,5\},\{2,3,5,7\},\{2,3,5,7\}$
 respectively;
\item the four exceptional reflection groups $G_{12},G_{29},G_{31}$ and
 $G_{34}$ at $\ell=3, 5, 5$ and~7 respectively, studied by Aguad\'e \cite{A89};
 and
\item the exceptional reflection group $G_{24}$ at $\ell=2$.
\end{itemize}

The Clark--Ewing groups comprise the symmetric groups $\fS_n$ for $n<\ell$, the
imprimitive groups $G(e,r,n)$ with $n<\ell$, $e\ge2$, $r|e|(\ell-1)$,
or $5\le e=r|(\ell+1)$ when $n=2$ (the dihedral groups), as well as those
exceptional $\ZZ_\ell$-reflection groups of order prime to $\ell$.

%%%%%%%%%%%%%%%%%%%%%%%%%%%%%%%%%%%%
\subsection{Relative Weyl groups}  \label{subsec:sl}
Let $(W,L)$ be a $\ZZ_\ell$-reflection group,
$V:=\overline\QQ_\ell\otimes_{\ZZ_\ell}L$ and let $\phi\in N_{\GL(V)}(W)$.
Recall that there exists a fundamental system of homogeneous invariants
$f_1,\ldots,f_r$, $r=\rk(L)$, of $W$ in the symmetric algebra on $V$ which are
eigenvectors for $\phi$, with eigenvalues $\eps_i$ say. The uniquely determined
multiset $\{(d_i,\eps_i)\mid 1\le i\le r\}$, where $d_i:=\deg f_i$, are the
\emph{generalised degrees} of $W\phi$. The \emph{order polynomial} of the
coset $W\phi$ is defined as %(see \cite[Def.~1.44]{BMM14})
$$\OO_x(W\phi):=
  \prod_i(x^{d_i}-\eps_i)\in\overline\QQ_\ell[x].$$
Note that $\OO_x(W\phi)\in\ZZ_\ell[x]$ if $\phi\in N_{\GL(L)}(W)$.

\begin{thm}[Lehrer--Springer]   \label{thm:LS}
 Let $g\in W\phi$ have fixed space $V(g,1):=\ker_V(g-1)$ of maximal possible
 dimension among elements in $W\phi$. Then:
 \begin{itemize}
  \item[\rm(a)] $W_\phi:=N_W(V(g,1))/C_W(V(g,1))$ acts faithfully as a
   reflection group on $V(g,1)$, where $C_W(V(g,1))$ is the point-wise
   stabiliser of $V(g,1)$ in $W$.
  \item[\rm(b)] The group $W_\phi$ is uniquely determined up to the
   conjugation action of $W$ on the set of its subquotients.
  \item[\rm(c)] We have $\OO_x(W_\phi)=\prod_{i\,:\,\eps_i=1}(x^{d_i}-1)$,
   where $\{(d_i,\eps_i)\}$ are the generalised degrees of $W\phi$ on
   $\overline\QQ_\ell\otimes_{\ZZ_\ell}L$.
  \item[\rm(d)] Assume $\phi\in N_{\GL(L)}(W)$. Then $W_\phi$ is a
   $\ZZ_\ell$-reflection group on the fixed space $L_\phi=\ker_L(g-1)$ of $g$
   in $L$.
 \end{itemize}
\end{thm}

\begin{proof}
Parts~(a) and~(c) are \cite[Thm.~1.1]{LSp99}. By \cite[Lemma~3.1]{LSp99}
the maximal eigenspaces $V(g,1)$ for $g$ running over $W\phi$ are all
$W$-conjugate, whence~(b). Part~(d) is clear.
\end{proof}

We call $W_\phi$ the \emph{relative Weyl group of a Sylow $1$-torus} of $W\phi$
by analogy with the case of finite reductive groups and spetses. If $V(g,1)$
contains an eigenvector of $g$ not lying in any of the reflecting hyperplanes of
$W$ then $g$ is called $1$-\emph{regular}. In this case, Springer showed that
$W_\phi\cong C_W(g)$.

\begin{exmp}   \label{exmp:LS}
 We describe the relative Weyl groups in imprimitive groups when $\phi$ is a
 scalar. Let $W=G(m,r,n)$ with $m|(\ell-1)$ and let
 $\phi=\zeta\in\ZZ_\ell^\times$
 be a primitive $e$th root of unity. Then
 $$W_\zeta=\begin{cases}
   G(me',r,\frac{n}{e'})& \mbox{if $re\mid mn$}, \\
   G(me',1,\frac{n}{e'}-1)& \mbox{if $re\nmid mn$ but $e \mid mn$}, \\
   G(me',1,\lfloor\frac{n}{e'}\rfloor)& \mbox{otherwise},\\
 \end{cases}$$
 where $e'=e/\gcd(e,m)$ (see \cite[(2.6) and~(5.4)]{Ma95} and
 \cite[A.10]{BM07}).
\end{exmp}

%%%%%%%%%%%%%%%%%%%%%%%%%%%%%%%%%%%%
\subsection{Stabilisers}
The following general fact will be needed in Section~\ref{sec:centr} to prove
a Steinberg-type result on centralisers in $\ell$-compact groups; it might be
of independent interest. Recall that \emph{parabolic subgroups} of a $\bar {\mathbb Q}_{\ell}$-reflection group $W$ are by
definition the point-wise stabilisers of subspaces. By a theorem of Steinberg
\cite[1.20]{StEnd} these are reflection subgroups.

\begin{prop}   \label{prop:stab}
 Let $(W,L)$ be a $\ZZ_\ell$-reflection group,
 $\brT= \ZZ/\ell^\infty\otimes_{\ZZ_\ell} L$ the corresponding discrete
 torus.
 Let $v\in\brT$ and set $H=C_W(v)$. Then $H/H_\refl$ is of $\ell$-power order,
 where $H_\refl$ denotes the normal subgroup generated by the reflections in
 $H$. Moreover, if $W$ is exotic then $H=H_\refl$ is generated by reflections,
 and if $W$ is an $\ell'$-group then $H=H_\refl$ is a parabolic subgroup.
\end{prop}

\begin{proof}
We provide a conceptual argument for the first assertion; the proof of the
second statement will be completed by a case-by-case argument.

First assume that $|H|$ is prime to $\ell$. Let $a \in\NN$. By the Maschke
argument the canonical surjection $L\to L/\ell^a L$ restricts to a surjection
$L^H \to (L/\ell^a L)^H $. Indeed, let $v\in (L/\ell^a L)^H$ and let $y\in L$ be
any lift of $v$. Then $\frac{1}{|H|}\Tr_1^H (y) \in L^H$ has image
$\frac{1}{|H|}\Tr_1^H(v) =v$ in $L/\ell^a L$.

Now let $v\in \brT$ be an element of order $\ell^a$. By the above, there exists
$x\in L$ such that $v\in x +\ell^a L$ and $C_W(v) \leq C_W(x)$. Since
$C_W(x)\leq C_W(v)$ we have that $C_W(x) = C_W(v)$.
Now the stabiliser of any $x\in L$ is the same as the stabiliser of its image
in $\QQ_\ell\otimes_{\ZZ_\ell}L$, which is a reflection representation over a
field of characteristic~0. But there, by the result of Steinberg
\cite[1.20]{StEnd} stabilisers of vectors are reflection subgroups and in fact
parabolic subgroups. \par
In the general case, we may apply the above to a collection of subgroups of
$H$ of order coprime to $\ell$ generating $O^\ell(H)$ to see that
$H_\refl\ge O^\ell(H)$, giving the first claim.   \par

Now assume that $W$ is exotic. By the first part we may assume that $\ell$
divides $|W|$. Here we show the claim by a case-by-case argument. If $v$ has
order $\ell$, then the result was shown in \cite[Thm.~1.9]{AGMV}. 

Let us first discuss the four Aguad\'e groups. Here, Sylow $\ell$-subgroups are
cyclic of order~$\ell$. By the first part we only need to consider stabilisers
$C_W(v)$ containing an element of order~$\ell$. A direct computation shows that
the fixed space in $\Omega_1(\brT)$ of such an element is 1-dimensional and
moreover the centraliser of this subspace is a symmetric group $H=\fS_\ell$ (see
also the last part of the proof of \cite[Thm.~7.1]{AGMV}). Thus, $H$ acts on
$\Omega_1(\brT)$ as the Weyl group of the simply connected group $\SL_\ell$.
Here, the claim follows by \cite[2.14 and 2.16]{St75}. Similarly, for the exotic
$\ZZ_2$-reflection group $G_{24}$, which acts as $\GL_3(2)$ on
$\bar L:=\FF_2\otimes_{\ZZ_2}L$, the stabilisers of elements in $\bar L$ are
Weyl groups of simply connected type $B_3$ and we conclude as before.

It remains to consider the case $W=G(e,r,n)$, with $e\ge2$ and $r|e|(\ell-1)$.
Here we may argue directly following \cite[Lemma 7.11]{Mo02} (which describes
centralisers of elementary abelian subgroups of $\brT$). We identify $\brT$
with $\ZZ/\ell^\infty\otimes_{\ZZ_\ell}L = (\ZZ/\ell^\infty)^n$
(written additively). Identify $W$ with $A(e,r,n)\rtimes\fS_n$, where $A(e,r,n)$
is the subgroup of diagonal $n\times n$-matrices over $\ZZ_\ell$ such that all
diagonal entries are $e$-th roots of unity and the determinant is an $e/r$-th
root of unity \cite{Ma95}.
Let $t =(\la_1,\ldots,\la_n)\in\brT$. We declare indices $i$ and $j$,
$1\leq i,j \leq n$, to be equivalent if $\lambda_i=\zeta\lambda_j$ for some
$e$-th root of unity $\zeta\in\ZZ_\ell$. Denote by ${\bf n}(t)$ the partition
of $\{1,\ldots, n\}$ into equivalence classes. Suppose that there are $u_0$
indices $i$ such that $\la_i =0$ and $s$ further equivalences classes containing
$u_1,\ldots,u_s$ elements respectively.

Now for $a=(\zeta_1,\ldots,\zeta_n)\in A(e,r,n)$ and $\sigma \in\fS_n$, we have
\[ \,^{(a,\sigma)}t = (\zeta_1\la_{\sigma(1)},\ldots,\zeta_n\la_{\sigma(n)}).\]
So, $(a,\sigma)\in C_W(t)$ if and only if $\la_i = \zeta_i\la_{\sigma(i)}$
for all $i$, and then $\sigma$ preserves the partition ${\bf n}(t)$.
Now, the group of $e$-th roots of unity in $\ZZ_\ell^\times$ acts fixed point
freely on $\ZZ/\ell^\infty \setminus \{0\}$. Hence if
$(a,\sigma)\in C_W(t)$ is such
that $\sigma = \tau\sigma_1\cdots\sigma_s$, where $\tau\in\fS_{u_0}$ and
$\sigma_j\in \fS_{u_j}$, $1\leq j \leq s$, then for all $i$ in the $j$-th
equivalence class, $\zeta_i$ is uniquely determined by
$\la_i = \zeta_i \la_{\sigma_j(i) }$. In particular, all entries of $a$ other
than those in the $0$th equivalence class are determined by
$\sigma_1\cdots\sigma_s$. Conversely, if $a$ and
$\sigma = \tau\sigma_1\cdots\sigma_s$ are such that for all $1\leq j \leq s$,
and all $i$ in the $j$th class $\la_i = \zeta_i\la_{\sigma_j(i)}$,
then $\prod_{i'\sim i}\la_{i'}=1$ and hence $(a,\sigma)\in C_W(t) $. Thus the
map
$$ G(e,r,u_0) \times \prod_{1\leq j \leq s} \fS_{u_j}\to C_W(t), \quad
  (a_0,\tau)\times\prod_{1\leq j\leq s}\sigma_j\mapsto(a,\tau\sigma_1\cdots \sigma_s),$$
where the entries of the components of $a$ in the $0$th class are the entries
of $a_0$ and the entries of the components corresponding to the $j$th class are
determined by the equations $\la_i = \zeta_i\la_{\sigma_j(i)}$, is an
isomorphism. Now the image of a transposition $\sigma_j\in\fS_{u_j}$ under the
above map is a reflection. Since the image of $G(e,r,u_0)$ is also generated by
reflections, $C_W(t)$ is indeed generated by reflections. The proof is complete.
\end{proof}

%%%%%%%%%%%%%%%%%%%%%%%%%%%%%%%%%%%%
\subsection{Good, bad and very good primes}
We extend the notion of bad primes to $\ZZ_\ell$-reflection groups.
For Weyl groups these are defined in terms of torsion of the root system modulo
closed subsystems, or of coefficients of the highest root (see e.g.
\cite[\S B.5]{MT}). Since there is no general concept of root systems for
$\ZZ_\ell$-reflection groups (yet), we propose a different characterisation
which turns out to agree with the usual one for Weyl groups.
\par
 We say that a reflection
subgroup of   a $\bar {\mathbb Q}_{\ell}$-reflection group $W$ has \emph{maximal rank} if its fixed space agrees with that
of $W$. For reflection groups in positive characteristic, there can exist
subspace stabilisers which are not parabolic;
for example, a maximal rank subgroup could have invariant vectors. It is this
phenomenon that underlies our definition of bad primes.

\begin{defn}   \label{def:bad}
 Let $(W,L)$ be a $\ZZ_\ell$-reflection group. We say that $\ell$ is \emph{good}
 for $W$ if all maximal rank reflection subgroups of $W$ (on
 $L\otimes_{\ZZ_\ell}\QQ_\ell$) have the same number of
 trivial composition factors on $\FF_\ell\otimes_{\ZZ_\ell}L$. Otherwise we
 call $\ell$ a \emph{bad prime} for $W$.
 We say that $\ell$ is \emph{very good} for $(W,L)$ if it is good for $W$, and
 $L_W$ and $(L^*)_W$ are both torsion-free.
\end{defn}

Here, we write $L_W=L/\langle(w-1)L\mid w\in W\rangle$ for the coinvariants of
$(W,L)$. Let us point out that all exotic $\ZZ_\ell$-reflection groups have
$L_W=0$, that is, they are \emph{simply connected}, see \cite[Prop.~1.6]{N99}.
So for exotic reflection groups, all good primes are very good.

Observe that the notion of being very good may depend on the chosen
$\ZZ_\ell$-lattice for~$W$, while being good only depends on $W$. We have the
following properties:

\begin{lem}   \label{lem:good}
 Assume $(W,L)$ is a $\ZZ_\ell$-reflection group.
 \begin{enumerate}
  \item[\rm(a)] Then $\ell$ is good for $W$ if and only if it is good for all
   irreducible factors of~$W$.
  \item[\rm(b)] If $(W,L)=(W_1\times W_2,L_1\oplus L_2)$ is a direct
   decomposition, then $\ell$ is very good for $(W,L)$ if and only if it is so
   for $(W_1,L_1)$ and $(W_2,L_2)$.
  \item[\rm(c)] If $\ell$ is good for $W$, then it is so for all of its
   reflection subgroups.
  \item[\rm(d)] If $W$ is a Weyl group, the notion of bad primes agrees with
   the classical one.
 \end{enumerate}
\end{lem}

\begin{proof}
Parts~(c) and~(d) will be a consequence of Proposition~\ref{prop:bad} in
conjunction with~(a). The other claims are obvious.
\end{proof}

\begin{prop}   \label{prop:bad}
 Let $(W,L)$ be an irreducible $\ZZ_\ell$-reflection group.
 \begin{enumerate}
 \item[\rm(a)] The bad primes $\ell>2$ are as given in Table~\ref{tab:bad},
  and~$\ell=2$ is bad unless $W\cong\fS_n$.
 \item[\rm(b)] The only groups with good primes that are not very good are
  $W=\fS_n$ with $\ell|n$.
 \end{enumerate}
\end{prop}

\begin{table}[htb]
 \caption{Bad primes $\ell>2$}   \label{tab:bad}
$\begin{array}{c|cccccccccc}
  & G(6,6,2)& G_{12}& G_{28}& G_{29}& G_{31}& G_{34}& G_{35}& G_{36}& G_{37}\\
 \hline
  \ell& 3& 3& 3& 5& 5& 7& 3& 3& 3,5\\
\end{array}$
\end{table}

Note that $G(6,6,2),G_{28},G_{35},G_{36},G_{37}$ are the Weyl groups of types
$G_2,F_4,E_6,E_7$ and $E_8$, respectively, and $G_{12},G_{29},G_{31}$ and
$G_{34}$ are the Aguad\'e groups.

\begin{proof}
If $W$ is a Clark--Ewing group, that is, $W$ has order prime to $\ell$, then
$\bar L:=\FF_\ell\otimes_{\ZZ_\ell}L$ is semisimple and moreover, any subgroup
of $W$ has a trivial composition factor on $\bar L$ if and only if it has one
on $L$. This shows that $\ell$ is very good for $W$.

For Weyl groups, if $\ell$ appears in Table~\ref{tab:bad} then there is a
maximal rank subgroup occurring as the Weyl group of the centraliser of an
isolated $\ell$-element (see e.g.~\cite[\S\S4,5]{Bo05}) of a corresponding
linear algebraic group in characteristic different from $\ell$, so $\ell$ is
bad according to our definition. On the other hand, the description of maximal
rank subgroups via Borel--de Siebenthal \cite[13.2]{MT} shows that only the
listed cases occur. For example, for $W=\fS_n$ there is no proper maximal rank
subgroup at all and so all primes are good.

For $G(e,r,n)$ with $e\ge2$ dividing $\ell-1$, the description in the proof of
Proposition~\ref{prop:stab} shows that there are no proper maximal rank
stabilisers. Among the exotic exceptional $\ZZ_\ell$-reflection groups, only
the four Aguad\'e groups and $G_{24}$ have order divisible by $\ell$. As
already observed in the proof of Proposition~\ref{prop:stab}, for the former
groups there exists an element $0\ne v\in\bar L$ with
stabiliser a proper maximal rank subgroup, viz.\ a symmetric group $\fS_\ell$.
The group $G_{24}$ has a maximal rank reflection subgroup $A_1^3$. Clearly that
has trivial composition factors in characteristic~2, so~2 is bad for $G_{24}$.

Now assume that $\ell$ is good for $(W,L)$. If the reflection character of $W$
on $L$ does not lie in the same $\ell$-block as the trivial character, then the
condition for $\ell$ being very good is obviously satisfied. This is the case,
for example, if $W$ has a non-trivial normal $\ell'$-subgroup. This
deals with the groups $G(e,r,n)$ with $e\ge2$, as well as with the exceptional
Weyl groups. Thus only $W=\fS_n$ remains to be considered. Here, the reflection
character lies in the principal $\ell$-block only if $\ell|n$.
\end{proof}

\begin{rem}   \label{rem:BCM}
 For so-called \emph{well-generated} complex reflection groups $W$,
 Brou\'e--Corran--Michel \cite[\S8]{BCM} define the notion of \emph{bad prime
 ideals} and a \emph{bad ideal} when $W$ is moreover spetsial
 (see Section~\ref{subsec:spets}). These are
 certain ideals in the ring of integers of the field of definition of $W$. In
 case $W$ is also a $\ZZ_\ell$-reflection group, it can be checked from the
 tables in \cite{BCM} that $\ell$ is bad in our sense if and only if it divides
 the norm of their ``bad ideal''.
\end{rem}

\begin{rem}   \label{rem:very good Lie}
 Let $\ell>2$ and $(W,L_0)$ be the Weyl group of a connected reductive group
 $\bG$ in characteristic different from $\ell$, and set
 $L:=\ZZ_\ell\otimes_\ZZ L_0$. Then $L_W=0$ if and only if $\bG$ is semisimple
 with fundamental group of $\ell'$-order, and $(L^*)_W=0$ if and only if
 $\bG$ is semisimple with $|Z(\bG)/Z^\circ(\bG)|$ of $\ell'$-order.
 Furthermore, $L_W$ is torsion-free if and only if $\bG$ has derived subgroup
 with fundamental group an $\ell'$-group.   \par
 Thus, for example, if $\bG$ is simple of classical type $B,C$ or $D$, then all
 primes $\ell>2$ are very good for the corresponding $(W,L)$. 
\end{rem}

The following will be relevant for the descent to relative Weyl groups:

\begin{prop}\label{prop:gooddescent}
 Let $(W,L)$ be a $\ZZ_\ell$-reflection group with $L_W=0$ and
 $\phi\in N_{\GL(L)}(W)$. If $\ell>2$ is very good for $(W,L)$, then it is so
 for $(W_\phi,L_\phi)$.
\end{prop}

\begin{proof}
Assume $W=W_1\times W_2$ is a direct decomposition into $\phi$-stable reflection
subgroups. Then $L_i:=\langle (w-1)L\mid w\in W_i\rangle$ is $\phi$-stable for
$i=1,2$, and since $L_W=0$ we have $L=L_1\oplus L_2$. Furthermore,
$L_\phi=(L_1)_{\phi|_{L_1}}\oplus (L_2)_{\phi|_{L_2}}$ and $(L_i)_{W_i}=0$. Thus
we may assume that $(W,L)=(W'\times\cdots\times W',L'\oplus\cdots\oplus L')$
where $(W',L')$ is irreducible and $\phi$ permutes the $m$ factors transitively.
Set $\phi':=\phi^m|_{L'}$, then
$$L_\phi
  =\{(x,\phi(x),\ldots,\phi^{m-1}(x))\mid x\in L'_{\phi'}\}
  \cong L'_{\phi'}\,.$$
Thus by Lemma~\ref{lem:good} we are reduced to the case of irreducible $(W,L)$. 
\par
By inspection from Proposition~\ref{prop:bad}, if $\ell$ is good for $(W,L)$
then it is so for all of its relative Weyl groups. Any $1\ne w\in Z(W)$ acts by
a root of unity $1\ne\xi\in\ZZ_\ell^\times$ on $L$ and thus stabilises $L_\phi$.
Since $\xi-1\in\ZZ_\ell^\times$, we obtain
$(L_\phi)_{\langle w\rangle}=0=(L_\phi^*)_{\langle w\rangle}$ and so $\ell$ is
very good for $(W_\phi,L_\phi)$.
By \cite[Tab.~1 and~6]{GM06} the only $\ZZ_\ell$-reflection groups with trivial
centre are $G(e,r,n)$ with $e\ge3$ and $\gcd(e,ne/r)=1$, $\fS_n$ and $W(D_n)$
with $n$ odd. \par
In the first case, proper relative Weyl groups again have type $G(f,s,m)$ with
$f>2$ dividing $\ell-1$ (see Example~\ref{exmp:LS}), so $\ell$ is very good for
them. In the other two cases, $\phi$ acts on $L$ like an element $w\in W$ times
a scalar $\xi$, and so $w\in N_W(L_\phi)$ acts as the scalar $\xi^{-1}$ on
$L_\phi$. Hence we are done unless $\xi=1$. Finally, for $\xi=1$ we have
$W_\phi=W$.
\end{proof}

The conclusion fails to hold for $\ell=2$ since $2$ is very good for the
$\ZZ_2$-reflection group attached to $\SL_n$ with $n$ odd, but its relative
Weyl group for $\phi=-\Id$ is of type $B_{(n-1)/2}$, for which~2 is bad.
\smallskip

We do not know whether $\ell>2$ being very good for $(W,L)$ descends to any
relative Weyl group $(W_\phi,L_\phi)$ if $L_W\ne0$. According to
\cite[Prop.~C.4]{GL20} combined with \cite[Thm.~12.2]{AGMV}, if $\phi$ is of
$\ell'$-order then this is at least true for the property $L_W$ being
torsion-free.

%%%%%%%%%%%%%%%%%%%%%%%%%%%%%%%%%%%%%%%%%%%%%%%%%%%%%%%%%%%%%%%%%%%%%%%%%
\section{From $\ZZ_\ell$-reflection groups to fusion systems via $\ell$-compact groups}   \label{sec:zltofusion}
We explain how to obtain a fusion system starting from a $\ZZ_\ell$-reflection
group and a prime power $q$ via homotopy fixed points on the associated
$\ell$-compact group, and we develop some properties of the fusion system
thus obtained.

%%%%%%%%%%%%%%%%%%%%%%%%%%%%%%%%%%%%
\subsection{$\ell$-compact groups}   \label{subsec:l-comp}
We follow the exposition of \cite{AGMV}, \cite{DW94} and \cite{BM07}. Let
$\ell$ be a prime. Recall that the Bousfield--Kan $\ell$-completion, denoted
$\coX\mapsto\coX_\ell^\wedge$, is a functor from the category of spaces (that
is, simplicial sets) to itself, along with a natural transformation
$\Id\to(-)_\ell^\wedge$. A space $X$ is said to be \emph{$\ell$-complete}
if $\coX\to\coX_\ell^\wedge$ is a homotopy equivalence. 

An \emph{$\ell$-compact group} is a pointed, connected and $\ell$-complete
space $\coX$ such that the cohomology $H^*(\Omega\coX; \FF_\ell)$ is finite
dimensional over $\FF_\ell$. Here $\Omega\coX$ refers to the space of based
loops in $\coX$. A \emph{morphism} between $\ell$-compact groups is a map of
pointed spaces. Two morphisms $\vhi,\vhi':X\to Y$ between $\ell$-compact groups
are \emph{conjugate} if $\vhi$ and $\vhi'$ are freely homotopic and $\vhi$ is
an \emph{isomorphism} (or \emph{equivalence}) if it is a homotopy equivalence.
Further, $\vhi:\coX\to\coY$ is a \emph{monomorphism} if the homotopy fibre
of $\vhi$ has finite $\FF_\ell$-homology. We note that a morphism between
$\ell$-compact groups may also be thought of as a based homotopy class of based
maps, but we will use the former point of view.

We choose this setup which we find more intuitive than the conventional one in
which an $\ell$-compact group $\coX$ is presented as a triple $(\bX,B\bX,e)$
where $B\bX = \coX$, $\bX$ is a space homotopy equivalent to $\Omega\coX$ and
$e:\bX\to\Omega\coX$ is a chosen homotopy equivalence. 

An $\ell$-compact group $X$ is \emph{connected} if $\Omega\coX$ is connected.
The reader should beware of the slight notational ambiguity of this terminology
as the space $X$ itself is always connected. We also note that if $Y$ is a
connected $\ell$-compact group, then two morphisms $\vhi,\vhi':\coX\to\coY$ are
conjugate if and only if there is a based homotopy between $\vhi$ and $\vhi'$.

If $\bG$ is a compact Lie group whose component group is a finite $\ell$-group,
then $(B\bG)_\ell^\wedge$ is an $\ell$-compact group, where for any topological
group $\bG$, $B\bG$ refers to the classifying space of $\bG$. The
$\ell$-completed classifying space $\coT=(B\bU(1)^r)^\wedge_\ell$ of a
compact torus $\bU(1)^r$ is called an \emph{$\ell$-compact torus of rank $r$}.
If $\coX$ is an $\ell$-compact group, then a \emph{maximal torus} of $\coX$ is a
monomorphism $i: \coT \to \coX$ from an $\ell$-compact torus $T$ to $\coX $ such
that the homotopy fibre of $i$ has non-zero Euler characteristic. Here the Euler
characteristic of a space is taken to be the alternating sum of the dimensions
of its $\FF_\ell$-homology groups.

By the fundamental results of Dwyer and Wilkerson \cite[Prop.~8.11,
Thm.~8.13]{DW94}, every $\ell$-compact group has a maximal torus $i:\coT\to\coX$
which is unique up to conjugation in the sense that if $i':\coT'\to\coX$ is any
maximal torus, then there exists an isomorphism $\vhi:\coT'\to\coT$ such that
$i'$ and $i\circ\vhi$ are conjugate.

Fix an $\ell$-compact group $X$ with maximal torus $i:\coT\to X$ and set
$L_\coX:=\pi_2(\coT)$, a free $\ZZ_\ell$-module of the same rank as $T$. Let
$\cW_\coX(\coT)$ be the topological monoid of self-maps of $T$ over $X$ (defined
after replacing $i$ by an equivalent fibration) \cite[Def.~9.6]{DW94} and set
$W_\coX = \pi_0 (\cW_\coX(\coT))$. By \cite[Prop.~9.5, Thm.~9.7,
Thm.~9.17]{DW94}, $W_\coX$ is a finite group and if $X$ is connected, then
$W_\coX$ is identified with the set of conjugacy classes of self-equivalences
$\vhi$ of $T$ such that $i$ and $i\circ\vhi$ are conjugate. This
sets up (for connected $X$) a faithful action of $W_\coX$ on $L_\coX$ as a
finite $\ZZ_\ell$-reflection group. The pair $(W_\coX,L_\coX)$ is the
\emph{Weyl group} of $\coX$. Occasionally and when there is no risk of
ambiguity, we will just use the first component $W_\coX$ to denote the Weyl
group of $\coX$. Note that $(W_\coX, L_\coX)$ is independent of the choice of
$\coT$ up to isomorphism. 
 
Suppose that $X$ is connected. Given any homotopy self-equivalence
$\alpha:\coX\to\coX$, there exists a self-equivalence $\alpha_T:\coT\to\coT$
such that $\alpha \circ i$ is conjugate to $i\circ\alpha_T$. Further, the
conjugacy class of $\alpha_T$ is uniquely determined up to the action of $W_X$
(see \cite[Thm.~1.2]{Mo96}). This induces a group homomorphism
$\Phi_X:\Out(\coX)\to N_{\GL(L_X)}(W_X)/W_X$ where $\Out(\coX)$ is the
\emph{outer automorphism group} of $\coX$, that is, the group of conjugacy
(i.e., homotopy) classes of homotopy self-equivalences of $\coX$. For odd
$\ell$, we have the following classification result.

\begin{thm} \cite[Thm.~1.1]{AGMV}   \label{thm:class}
 Suppose that $\ell$ is an odd prime. The assignment 
 $$\coX \leadsto (W_X,L_X)$$
 induces a bijective correspondence between isomorphism classes of connected
 $\ell$-compact groups and isomorphism classes of $\ZZ_\ell$-reflection groups.
 Here, the map
 $$\Phi_X: \Out(\coX) \to N_{\GL(L_X)}(W_X)/W_X $$
 is an isomorphism.
\end{thm}

%%%%%%%%%%%%%%%%%%%%%%%%%%%%%%%%%%%%
\subsection{Fusion systems from homotopy fixed points}   \label{subsec:def fus}
Let $X$ be a connected $\ell$-compact group with Weyl group $(W,L)$.
For $q\in\ZZ_\ell^\times$ we denote by $\psi^q$ a self-equivalence
of $X$ whose class in $\Out(X)$ is the image under $\Phi_X^{-1}$ of the coset
$\lambda_q W\in N_{\GL(L)}(W)/W$ where $\lambda_q:=q\Id \in\GL(L)$;
$\psi^q$ is called an \emph{unstable Adams operator} corresponding to $q$. For
$\tau$ an equivalence  of $\coX$  whose class in $\Out(X)$ has
finite order, let $\twt\coX(q)$ denote the space of homotopy fixed points
$X^{h \tau\psi^q}$. Thus $\twt\coX(q)$ is the space of all paths
$\gamma:[0,1]\to X$ such that $\gamma(1)=\tau\psi^q(\gamma(0))$ (with the
compact-open topology). Note that $\twt X(q)$ may also be described by the
homotopy pullback diagram:
$$\xymatrix{\twt X(q) \ar[d]^\iota\ar[r]^\iota& X \ar[d]^\Delta\\ 
 X \ar[r]^{(1,\tau \psi^q)} & X \times X}\eqno{(*)}$$
where $\Delta:\coX\to\coX\times\coX$ is the diagonal embedding, and that up to
homotopy equivalence $\twt\coX(q)$ only depends on the free homotopy class
of $\tau\psi^q$ and hence only on the class of $\tau\psi^q$ in $\Out(X)$ (see
\cite[Rem.~2.3, Lemma~B.1]{BMO}). 

Suppose that $X$ is a space and $S$ is a finite $\ell$-group for which there is
a continuous map $f:BS\rightarrow X$. By the functoriality of the classifying
space construction to each injective group homomorphism $\vhi \in \Inj(P,Q)$
is associated a map of spaces $B\vhi: BP \rightarrow BQ$, and we may form a
category $\cF_{S,f}(X)$ whose objects are the subgroups of $S$ with morphisms
given by
$$\Hom_{\cF_{S,f}(X)}(P,Q)
  :=\{\vhi \in \Inj(P,Q) \mid f|_{BP} \simeq f|_{BQ} \circ B\vhi\}$$
for each $P,Q \le S$ \cite[Def.~III.3.3]{AKO11}. To such an $f$ is also
associated a certain ``centric linking category'' $\cL^c_{S,f}(X)$ (see
\cite[Def.~III.3.3]{AKO11}). An $\ell$-complete space $X$ is said to be a
\emph{classifying space} of a saturated fusion system if there exists a finite
$\ell$-group $S$ and a morphism $f: BS\to X$ such that $\cF_{S,f}(X)$ is
saturated and $(S,\cF_{S,f}(X),\cL^c_{S,f}(X))$ is an $\ell$-local finite
group (see \cite[Def.~I.4.4]{AKO11}) with classifying space $X$, i.e.,
$X \simeq |\cL^c_{S,f} (X)|_\ell^\wedge $. By \cite[Thm.~7.4]{BLO03}
(see also \cite[Thm.~III.4.25]{AKO11}), the homotopy type of $X$ determines
the triple $(S,\cF_{S,f}(X),\cL^c_{S,f}(X))$ up to isomorphism, and in
particular an $\ell$-complete space can be the classifying space of at most one
(isomorphism class of) saturated fusion system. 
 
An $\ell$-compact group $X$ is called \emph{simply connected} if $\Omega X$ is 
simply connected. For $\ell$ odd, this is equivalent to the underlying
$\ZZ_\ell$-reflection group $(W,L)$ having $L_W=0$ (see \cite[Thm.~1.7]{AGMV}).
The following is shown in \cite[Thm.~A]{BM07}:

\begin{thm}   \label{thm:bm}
 Let $\ell$ be an odd prime and $\coX$ be a simply connected $\ell$-compact
 group. Let $q$ be a power of a prime~$p\neq\ell$ and $\tau$ be a
 self-equivalence of $\coX$ whose class in $\Out(X)$ has finite order prime
 to~$\ell$. Then $\twt\coX(q)$ is a classifying space of a saturated fusion
 system.
\end{thm}

\begin{rem}   \label{r:fried}
(a) Let $q$ be a prime power, prime to $\ell$, $\bG$ a connected reductive
algebraic group over $\overline\FF_q$ with Weyl group $W$ and cocharacter
lattice $L_0$, and $F:\bG\to\bG$ a Frobenius morphism with respect to an
$\FF_q$-structure corresponding to $\phi \in N_{\GL(L_0)}(W)$. Let $\coX$ be a
connected $\ell$-compact group with Weyl group $(W,\ZZ_\ell\otimes_\ZZ L_0)$
and let $\tau$ be a homotopy self-equivalence of $X$ whose class in
$\Out(\coX)$ corresponds to $\Phi_X^{-1}(W\phi)$.
Then by a theorem of Friedlander $(B\bG^F)^\wedge_\ell\simeq \twt X(q)$
(see \cite[Thm.~3.1]{BMO}). We note that here there is no assumption on $\ell$
nor on simple connectivity of $X$.\par
(b) By the existence and uniqueness of centric linking systems
established by Chermak \cite{Ch13} given any saturated fusion system $\cF$ on a
finite $\ell$-group there is a unique (up to homotopy equivalence) space $X$
such that $X$ is a classifying space of $\cF$. Thus, in the above theorem we
may also speak of $\twt\coX(q)$ being \emph{the} classifying space of a
saturated fusion system.
\par
(c) Theorem A of \cite{BM07} is stated under the \emph{a priori} stronger
assumption that $\tau$ is of finite $\ell'$-order but as shown in
\cite[Thm.~B]{BM07} the two conditions are in fact equivalent.
\end{rem}

In the situation of Theorem~\ref{thm:bm} we denote by $\coX(q)$ the space
$\twt\coX(q)$ with $\tau$ equal to the identity map, that is, the homotopy
fixed points in $\coX$ under $\psi^q$. We then have the following comparison
and recognition theorem. The action of the group $\Gamma $ on $\coX$ below is in
the sense of homotopical actions (see \cite[Section~6]{BM07},
\cite[Section~C.1]{GL20}).
 
\begin{thm}   \label{thm:homfixedpts}
 Let $\ell$ be an odd prime, $\coX$ a connected $\ell$-compact group with Weyl
 group $(W,L)$ and $\tau$ a self-equivalence of $X$ whose class in $\Out(X)$ has
 finite order
 prime to $\ell$, with $W\phi\in N_{\GL(L)}(W)/W$ associated to $\tau$ via
 Theorem~\ref{thm:class}. Let $q$ be a power of a prime~$p\neq\ell$ and write
 $q=\zeta q_0$ where $\zeta\in\ZZ_\ell^\times$ is a primitive $e$-th root of
 unity and $q_0\equiv 1 \pmod\ell$. Let $\Gamma$ be the subgroup of $\Out(X)$
 generated by the class of $\tau \psi^q$.
 \begin{itemize}
  \item[(a)] There is a canonical action of $\Gamma$ on $\coX$ such that the
   homotopy fixed point space $\coX^{h\Gamma}$ is a connected $\ell$-compact
   group with Weyl group $(W_{\phi\zeta},L_{\phi\zeta})$, and
   \[ \twt \coX(q) \simeq \coX^{h\Gamma}(q_0). \]
   If $\coX$ is simply connected, then so is $\coX^{h\Gamma}$.
  \item [(b)] If $q'\in\ZZ_\ell^\times\setminus\{1\}$ is such that
   $q\equiv q'\pmod\ell$ and $(q^e-1)_\ell = (q'^e-1)_\ell$, then
   $$\coX(q)\simeq\coX(q').$$
 \end{itemize}
\end{thm}

\begin{proof}
The first and last assertion of~(a) is \cite[Thm.~B]{BM07}, the homotopy
equivalence in~(a) for simply connected $X$ is given in the proof of
\cite[Thm.~A]{BM07} and~(b) for simply connected $X$ is \cite[Thm.~E(2)]{BM07}.
The general case of~(a)
is given in \cite[Thm.~1.2]{GL20}, the Weyl group is identified in
\cite[Thm.~C.3]{GL20} using Lehrer--Springer theory (Theorem~\ref{thm:LS}),
and~(b) is in \cite[Cor.~C.9]{GL20}.
\end{proof}

\begin{rem}   \label{rem:order}
 In Theorems~\ref{thm:bm} and~\ref{thm:homfixedpts} the assumption on the order
 of the class of $\tau$ is unnecessary if the underlying $\ZZ_\ell$-reflection
 group $(W,L)$ of $X$ is irreducible. Indeed, by \cite[Prop.~3.13]{BMM99}, the
 group $N_{\GL(L)}(W)/W$ has order prime to $\ell>2$ except when $W$ is of type
 $D_4$ and $\ell=3$. But for the latter situation, the restriction on $\tau$ is
 not necessary since Friedlander's result (see Remark~\ref{r:fried}(a)) covers
 this case.   \par 
 Thus, if $W$ is reducible, outer automorphisms of order divisible by $\ell$
 can only arise from wreath product situations. It would seem interesting to
 extend the above theorems to cover this case as well, in the spirit of
 Remark~\ref{rem:wreath}.
\end{rem}

%%%%%%%%%%%%%%%%%%%%%%%%%%%%%%%%%%%%
\subsection{The Sylow $\ell$-subgroups}   \label{subsec:sylows}
We identify the $\ell$-subgroup $S$ underlying the saturated fusion system
determined by the homotopy fixed points via Theorem~\ref{thm:bm}.

\begin{thm}   \label{thm:sylow}
 Let $\ell$ be an odd prime, and $\coX$ a simply connected $\ell$-compact group
 with Weyl group $(W,L)$ and a self-equivalence $\tau$ whose class in $\Out(X)$
 is of finite order prime to $\ell$, with $W\phi$ the corresponding element of
 $N_{\GL(L)}(W)/W$. Let $q$ be a prime power and $\zeta\in\ZZ_\ell^\times$ the
 primitive $e$-th root of unity such that $q\equiv\zeta\pmod\ell$. Then the
 saturated fusion system $\cF$ associated to $\twt\coX(q)$ via
 Theorem~\ref{thm:bm} satisfies:
 \begin{enumerate}
  \item[\rm(a)] The $\ell$-group $S$ underlying $\cF$ is a semidirect product
   $S=(\ZZ/\ell^a\ZZ)^r\rtimes (W_{\phi\zeta})_\ell$, where
   $\ell^a=(q^e-1)_\ell$,
   $r=\rk(L_{\phi\zeta})$ and $(W_{\phi\zeta})_\ell$ denotes a Sylow
   $\ell$-subgroup of~$W_{\phi\zeta}$.
  \item[\rm(b)] We have the Steinberg-type order formula
   $$|S| =\prod_{i\,:\,\zeta^{d_i}\eps_i=1}(q^{ed_i}-1)_\ell
         =\big(\prod_{i}(q^{d_i}-\eps_i^{-1})\big)_\ell = \OO_q(W\phi^{-1})_\ell $$
   where $(d_1,\eps_1),\ldots,(d_r,\eps_r)$ are the generalised degrees of
   $W\phi$ on $\overline\QQ_\ell\otimes_{\ZZ_\ell}L$.
 \end{enumerate}
\end{thm}

\begin{proof}
Write $q=\zeta q_0$, so $q_0\equiv1\pmod\ell$. By
Theorem~\ref{thm:homfixedpts}(a) we have
$\twt\coX(q)\simeq \coX^{h\Gamma}(q_0)$ where $\Gamma$ is the subgroup
of $\Out(X)$ generated by the class of $\tau\psi^q$. Now $q_0^e=q^e$, so
$$\nu_\ell(q^e-1)=\nu_\ell(q_0^e-1)=\nu_\ell(q_0-1),$$
where $\nu_\ell$ denotes the $\ell$-adic valuation on $\QQ_\ell$. Hence
$\coX^{h\Gamma}(q_0)\simeq \coX^{h\Gamma}(q^e)$ have isomorphic
fusion systems by Theorem~\ref{thm:homfixedpts}(b). So we may instead consider
the fusion system determined by $\coX^{h\Gamma}(q^e)$. Here, as
$q^e\equiv1\pmod\ell$ and as by Theorem~\ref{thm:homfixedpts}(a),
$W_{\phi\zeta}$ is the Weyl group of $\coX^{h\Gamma}$ the assertion in~(a)
follows from \cite[Prop.~7.6]{BM07}
\par
By Theorem~\ref{thm:LS}(b), $W_{\phi\zeta}$ is a reflection group with
multiset of degrees given by $D:=\{d_i\mid \zeta^{d_i}\eps_i=1\}$, so a
well-known formula shows $|W_{\phi\zeta}|=\prod_{d\in D} d$.
On the other hand, as $e$ is the order of $q$ modulo~$\ell$ it is easily seen
that $\nu_\ell\big((q^{de}-1)/(q^e-1)\big)=\nu_\ell(d)$ for any $d\ge1$, whence
$$\nu_\ell\big(\prod_{d\in D}(q^{de}-1)/(q^e-1)\big)
  =\nu_\ell(|W_{\phi\zeta}|).$$
The formula stated in~(b) now follows from~(a) and from the definition of the
order
polynomial (see Section~\ref{subsec:sl}).
\end{proof}

%%%%%%%%%%%%%%%%%%%%%%%%%%%%%%%%%%%%%%%%%%%%%%%%%%%%%%%%%%%%%%%%%%%%%%%%%
\section{Alperin's weight conjecture for homotopy fixed point fusion systems}   \label{sec:awc}
Let $k$ be an algebraically closed field of characteristic $\ell$ and let $\cF$
be a saturated fusion system on a finite $\ell$-group. Recall the function
$\bw(\cF)$ defined in the introduction. If $B_0$ is the principal block of a
finite group with fusion system $\cF$ then $\bw(\cF)$ counts the number of
weights associated to $B_0$. Alperin's weight conjecture (AW conjecture) \cite{A87}
predicts that $\bw(\cF) = l(B_0)$ where $l(B_0)$ is the number of isomorphism
classes of simple $B_0$-modules. When the underlying group is a finite group of
Lie type, then $e$-Harish-Chandra theory combined with the theory of basic sets
and block theory gives a highly non-trivial formula for $l(B_0)$ in terms of
relative Weyl groups, which leads to:

\begin{prop}   \label{prop:linkawc}
 Let $\bG$ be a connected reductive algebraic group defined over $\FF_q$ with
 corresponding Frobenius endomorphism $F:\bG\to\bG$ acting as $\phi$ on the
 Weyl group $W$ of~$\bG$. If $\ell$ is a very good prime for $\bG$, and $q$ has
 order~$e$ modulo~$\ell$, then
 $$\text{$B_0(\bG^F)$ satisfies AW conjecture}\quad\Longleftrightarrow\quad
   \bw(\cF_\ell(\bG^F)) = |\Irr (W_{\phi\zeta})|,$$
 where $\zeta$ is a primitive $e$th root of unity.
\end{prop}

Since $W$ is rational, the isomorphism type of $(W_{\phi\zeta},L_{\phi\zeta})$
is independent of the choice of~$\zeta$.

\begin{proof}
Set $B_0=B_0(\bG^F)$. It suffices to show that $l(B_0)=|\Irr(W_{\phi\zeta})|$.
Now
$l(B_0)=|\Irr(B_0)\cap\cE(\bG^F,1)|$ is the number of unipotent characters of
$\bG^F$ in $B_0$, by \cite[Thm.~A]{Ge93} and Remark~\ref{rem:very good Lie}.
By \cite[Thm.~5.24]{BMM93}, $\Irr(B_0)\cap\cE(\bG^F,1)$ is the principal
$e$-Harish-Chandra series of~$\bG^F$, which in turn by \cite[Thm.~3.2]{BMM93}
is in bijection with $\Irr(W_{\phi\zeta})$.
\end{proof} 

In light of the above, our first main result, Theorem~\ref{thmb} from the
introduction and which we restate here can be seen as the statement that a
version of the AW conjecture holds for $\ell$-compact groups. 

\begin{thm}   \label{thm:nr weights}
 Let $\ell>2$, $X$ a simply connected $\ell$-compact group with Weyl group
 $(W,L)$, $\tau$ an automorphism of $X$ whose class in $\Out(X)$ has finite
 order prime to $\ell$. Let $q$ be a power of a prime different from~$\ell$ and
 let $\cF(\twt X(q))$ be the saturated fusion system associated to $\twt X(q)$.
 If $\ell$ is very good for $(W,L)$, then
 $$\bw(\cF(\twt X(q))) = |\Irr(W_{\phi\zeta})|,$$
 where $W\phi=\Phi_X(\tau)\in N_{\GL(L)}(W)/W$ and $\zeta\in\ZZ_\ell^\times$ is
 the root of unity with $q\equiv\zeta\pmod\ell$.
\end{thm}

\begin{proof}
First note that by Proposition~\ref{prop:gooddescent}, $\ell$ is very good for
$(W_{\phi\zeta},L_{\phi\zeta})$. Hence, by Theorem~\ref{thm:homfixedpts} we may
assume that $q\equiv1\pmod\ell$ and $\tau=\Id$, so $W_{\phi\zeta}=W$.
Now any simply connected $\ZZ_\ell$-reflection group is the direct
product of irreducible simply connected $\ZZ_\ell$-reflection groups by
\cite[Thm.~11.1]{AGMV}. Analogously, if $\cF$ is the product of saturated
fusion systems $\cF_1,\cF_2$, then $\cF^\CR=\cF_1^\CR\times \cF_2^\CR$ by
\cite[Lemma~3.1]{AOV12}, and thus $\bw(\cF)=\bw(\cF_1)\bw(\cF_2)$. Hence, it
suffices to consider the irreducible cases. By Proposition~\ref{prop:bad}, this
means that either $W$ is a Clark--Ewing group, a group $G(e,r,n)$ with $e\ge2$
and $r|e|(\ell-1)$, a Weyl group of type $E_6,E_7 $ or $E_8$ with $\ell=5$,
$\ell\in\{5,7\}$, $\ell=7$ respectively, or $\fS_n$ with $\ell{\not|}n$.
In the latter case, the assertion holds by the main result of \cite{Fe19}
combined with Proposition~\ref{prop:linkawc}.

For Clark--Ewing groups $W$ the order is prime to $\ell$. Then by
Theorem~\ref{thm:sylow}(b) the associated fusion system $\cF(\twt X(q))$ is on
an abelian $\ell$-group $S$. By, for example, \cite[Thm.~III.5.10]{AKO11},
$\cF(\twt X(q)) = \cF_S(N)$, where $N=S\rtimes W$. Thus, $S$ is the only centric
radical subgroup, $W$ is its outer automiser and
$|\Irr(W)|=z(kW)=\bw(\cF(\twt X(q)))$ since $|W|$ is coprime to~$\ell$. 

The other cases will be considered in Sections~\ref{subsec:GGr weights} and
\ref{subsec:awc exc}.
\end{proof}

Note that by Remark~\ref{r:fried}(a) the assumption on $\tau$ having order prime
to $\ell$ is not necessary if $(W,L)$ arises from a $\ZZ$-reflection group.
Thus, combined with Proposition~\ref{prop:linkawc}, Theorem~\ref{thm:nr weights}
yields the following.

\begin{cor}   \label{cor:AWC Lie}
 Alperin's weight conjecture holds for the principal $\ell$-blocks of finite
 reductive groups $\bG^F$, whenever $\bG$ is semisimple of simply connected
 type and $\ell>2$ is a very good prime for $\bG^F$.
\end{cor}

It is known that the conclusion of Theorem~\ref{thm:nr weights} does not
extend to all primes $\ell$ even for rational Weyl groups, since both the
description of blocks by $e$-Harish-Chandra series as well as unipotent
characters being basic sets can fail to hold. Examples for this can be found
in \cite[\S1.2]{GH91} and \cite{En00}.
It is thus not surprising that for bad primes we obtain a different result
(this will be shown in Section~\ref{subsec:awc exc}):

\begin{thm}   \label{thm:Aguade weights}
 If $W$ is one of the four Aguad\'e exceptional $\ZZ_\ell$-reflection groups,
 and $q\equiv 1\pmod\ell$ then
 $$\bw(\cF(X(q))) = |\Irr(W)|+1.$$
\end{thm}

\begin{rem}
 (a) It would be interesting to compute $\bw(\cF)$ for the fusion systems $\cF$
 of the principal $\ell$-blocks of $F_4,E_6,E_7,E_8$ for bad primes $\ell$. To
 our knowledge, the weights have not been determined in these cases. The
 validity of the AW conjecture for $G_2(q)$ at the bad prime $\ell=3$ has been
 shown by An \cite[(3B)]{An94}; here again the number of weights is
 $|\Irr(W)|+1$.
 \par
 (b) Note that in all of the above results the number of weights is independent
 of the power of $\ell$ dividing $q^e-1$.
\end{rem}

%%%%%%%%%%%%%%%%%%%%%%%%%%%%%%%%%%%%%%%%%%%%%%%%%%%%%%%%%%%%%%%%%%%%%%%%%
\section{Centralisers}   \label{sec:centr}
An important ingredient for our generalisation of the ordinary weight conjecture
will be the existence of centralisers, which we now discuss.

%%%%%%%%%%%%%%%%%%%%%%%%%%%%%%%%%%%%
\subsection{Centralisers in $\ell$-compact groups}
For $Q$ a discrete group, we consider $BQ$ as a pointed space with a chosen
base point. Extending the terminology from Section~\ref{subsec:l-comp} by a
morphism from $BQ$ to an $\ell$-compact group $X$, we will mean a continuous,
pointed map $f: BQ\to X$. Two morphisms $f,f': BQ \to X$ are
called \emph{conjugate} if $f$ and $f'$ are freely homotopic.

For spaces $Y,Z$, $\map(Y,Z)$ denotes the function space of maps from $Y$ to
$Z$ and the component containing a particular map or homotopy class $f$ is
denoted $\map(Y,Z)_f$. For $Q$ a discrete group, and a morphism $f$ from $BQ$
to an $\ell$-compact group $X$, we set $C_X(Q,f):=\map(BQ,X)_f$. If $R\leq Q$,
and $g:BR\to X$ is a continuous map we say that $f$ is an extension of $g$ if
$g\simeq f\circ B\iota$, where $\iota:R\to Q$ is the inclusion map, and we
denote by $f:BR\to X$ also the composition $f\circ B\iota$. 

From now on, and for the rest of this section, let $X$ be a connected
$\ell$-compact
group with maximal torus $i:T \to X$ and Weyl group $(W,L)$. Then $W$ identifies
with the set of conjugacy classes of self-equivalences $\vhi$ of $T$ such
that $i\circ\vhi$ and $i$ are conjugate (see Section~\ref{subsec:l-comp})
and for $w\in W$, we denote by $w: T\to T$ any representative of the conjugacy
class of self-equivalences indexed by $w$. 

We say that a morphism $f: BQ\to X$ \emph{factors through $T$} if there exists a
morphism $j:BQ \to T$ such that $i\circ j$ is conjugate to $f$, and in this case
we call $j$ a \emph{factorisation of $f$}. The following basic structural result
says that any $\ell$-element of a connected $\ell$-compact group is conjugate
to an element in the maximal torus and any two elements of the maximal torus
are conjugate in $X$ if and only if they are conjugate by an element of the
Weyl group (the Weyl group controls fusion of $\ell$-elements).
 
\begin{lem}   \label{l:toral}
 Let $Q$ be a finite cyclic $\ell$-group. Then any morphism $f:BQ\to X$ factors
 through $T$. If $j,j'$ are two factorisations of $f$, then $j'$ is conjugate
 to $w\circ j$ for some $w\in W$.
\end{lem} 

\begin{proof}
For the first assertion, see for instance \cite[Thm.~2.6]{MN}. The second
assertion is \cite[Prop.~4.1]{MN}. 
\end{proof} 

The $\ell$-compact torus $T$ has a discrete approximation $B\brT\to T$ as in
\cite[Sec.~6]{DW94} (see also \cite[Sec.~6]{DW95}) with
$\brT = \ZZ/\ell^\infty\otimes_{\ZZ_\ell } L \cong(\ZZ/\ell^\infty)^r$
where $r$ is the rank of $T$.
If $Q$ is a finite $\ell$-group or a discrete torus, then to every morphism
$j: BQ \to BT$ is associated a group homomorphism $\breve{j}:Q\to\brT$,
uniquely determined by $j$ such that $ B\breve{j}:BQ\to B\brT$ is a lift of
$j$, i.e., such that composition of $ B\breve{j}$ with $B\brT\to T$ is conjugate
to $j$ \cite[Prop.~3.2]{DW95} (here note that if $Q$ is a finite $\ell$-group,
$Q$ is a discrete approximation of $BQ$). We call $\breve{j}$ the \emph{discrete
approximation to~$j$}. For a subset $A \subset \brT$ we denote by $C_W(A)$ the
subgroup of elements $w\in W$ which fix $A$ point-wise.

By \cite[Prop.~5.1]{DW94} for any morphism $f:BQ\to X$, the centraliser
$C_X(Q,f)$ is an $\ell$-compact group. If $f$ factors through $T$, then the Weyl
group of $C_X(Q,f)$ may be described as follows. Let $j:BQ\to T$ be a
factorisation of $f$ and $\breve{j}$ its discrete approximation.
Then $i':T\to C_X(Q, f)$ is a maximal torus of $C_X(Q,f)$ for some morphism
$i'$ whose composition with the evaluation map $C_X(Q,f)\to X$ at the base
point of $X$ is conjugate to~$i$. Since $f$ is conjugate to $i\circ j$, we have
$C_X(Q,f)= C_X(Q,i\circ j)$. By \cite[Prop.~4.4]{DW95} and
\cite[Thm.~7.6]{DW95}, the Weyl group of $C_X(Q)$ with respect to $i'$
identifies with $(C_W(\breve{j}(Q),L)$.

We now prove Theorem~\ref{thmc}. The case $|Q|= \ell$ was shown in
\cite[Thm.~1.9]{AGMV}.

\begin{thm}   \label{thm:con cent}
 Suppose that $X$ is a connected $\ell$-compact group with torsion-free
 fundamental group and Weyl group $(W,L)$. Then for any morphism $f:BQ \to X$
 with $Q$ a finite cyclic $\ell$-group, the $\ell$-compact group $C_X(Q,f)$ is
 connected.
\end{thm} 

\begin{proof}
By \cite[Thm.~11.1]{AGMV} and \cite[Thm.~1.1]{AG09} there is a direct
decomposition $(W,L)=(W_\bG\times W_Y,(\ZZ_\ell\otimes_\ZZ L_\bG)\oplus L_Y)$,
where $(W_\bG,L_\bG)$ is the Weyl group of a compact Lie group whose fundamental
group has $\ell'$-torsion and $W_Y$ is a direct product of exotic
$\ZZ_\ell$-reflection groups. Since this induces a corresponding direct
decomposition of the pair $(W,\brT)$ and thus of stabilisers, we may assume
that $(W,L)$ either comes from a compact Lie group whose fundamental group has
$\ell'$-torsion, or is an exotic $\ZZ_\ell$-reflection group. In the former
case,
%since $(\ZZ_\ell\otimes_\ZZ L)_W= \ZZ_\ell\otimes_\ZZ L_W$,
%the torsion part of $L_W$ is an $\ell'$-group. As
%\[\brT= \ZZ/\ell^\infty\otimes_{\ZZ_\ell} (\ZZ_\ell\otimes_\ZZ L)
%  =\ZZ/\ell^\infty\otimes_\ZZ L,\]
the assertion follows from Steinberg's result \cite[Cor.~2.16]{St75}.
\par
Now let $(W,L)$ be exotic and let $(C_W(\breve{j} (s)),L)$ be the Weyl group of
$C_X( Q,f)$ as described above the theorem and denote by $W_1$ the Weyl group
of the connected component of $C_X(Q,f)$. By \cite[Thm.~7.6 and Rem.~7.7]{DW95},
$W_1\leq C_W(\breve{j}(s))$ is the subgroup generated by those
$w\in C_W(\breve{j}(s))$ whose image in $W$ is a reflection. Here note that if
$\ell=2$, then by the classification of $2$-compact groups in \cite{AG09},
$W=G_{24}$ and therefore for any reflection in $W$, of which there is just one
class, the subgroups $\sigma(s)$ and $F(s)$ from \cite[Def.~7.3]{DW95} agree.
Now by \cite[Thm.~4.7]{DW95}, $C_X(Q,f)$ is connected if and only if
$W_1= C_W(\breve{j}(s))$. Hence $C_X(Q,f)$ is connected if and only if
$C_W(\breve{j}(s))$ is a reflection subgroup of $W$. Since $W$ is exotic, the
latter holds by Proposition~\ref{prop:stab}.
\end{proof}

Let $\alpha$ be a self-equivalence of $X$ and let $\alpha_T:T\to T$ be such
that $i\circ\alpha_T$ is conjugate to $\alpha\circ i$ (see the discussion before
Theorem~\ref{thm:class}); such an $\alpha_T$ is called a lift of $\alpha$.
The following is a rephrasing of \cite[Lemma~7.3]{BM07}.

\begin{lem}   \label{l:identcent} 
 Let $\alpha$ be a self-equivalence of $\coX$, $\alpha_T: T\to T$ a lift of
 $\alpha$ and $\breve{\alpha_T}:\brT \to \brT$ its discrete approximation.
 Let $Q$ be a finite cyclic $\ell$-group, $f: BQ \to X$ a morphism, $j$ a
 factorisation of $f$ and $\breve{j} $ the discrete approximation of $j$. Let 
 $i': T\to C_X(Q,f)$ be a maximal torus of $C_X(Q,f)$ where $i'$ is a morphism
 whose composition with the evaluation map $C_X(Q,f)\to X$ at the base point of
 $X$ is conjugate to~$i$.
 \begin{enumerate}
  \item[\rm(a)] $f$ is conjugate to $\alpha \circ f$ if and only if there exists
   $w\in W$ such that $\alpha_T\circ j$ is conjugate to $ w\circ j$
   (equivalently $\breve{\alpha_T} \circ\breve{ j} = w \circ \breve{j}$). If it
   exists then $w$ is unique up to elements of $C_W(\breve{j}(Q))$. 
  \item[\rm(b)] For any $w \in W$ such that
   $\alpha_T\circ j$ is conjugate to $ w\circ j$, the induced homotopy
   equivalence $ \alpha_\#$ of $C_X(Q,f)$ has a lift to $T$ of the form
   $w^{-1} \alpha_T$, where $\alpha_\# $ denotes composition with $\alpha$.
 \end{enumerate}
\end{lem} 

\begin{proof}
The hypotheses imply that $i\circ \alpha_T\circ j$ is conjugate to~$i\circ j$,
i.e., that $ \alpha_T \circ j$ is a factorisation of $f$. Hence by
Lemma~\ref{l:toral}, $\alpha_T\circ j$ is conjugate to $ w\circ j$ for some
$w \in W$. Moving to discrete approximations yields
$\breve{\alpha_T} \circ\breve{ j} = w \circ \breve{j}$. The uniqueness up to
$C_W(\breve j(Q))$ is clear. For the final assertion see \cite[Lemma~7.2]{BM07}.
\end{proof}

\begin{rem}   \label{r:welldef}
Suppose that $\alpha\circ f = f$ in the above. If $j':BQ \to T$ is another
factorisation of $f$, and $w'\in W$ with
$\breve{\alpha_T}\circ\breve j'=w'\circ\breve{j'}$, then by uniqueness of
factorisations $\breve{j'} = x\circ\breve j$ for some $x\in W$. Setting
$C= wC_W(\breve{j}(Q))$ and $C' = w'C_W(\breve{j'}(Q))$ we have that
$C' = x^{-1} C\ \tw{\breve{\alpha_T}}x$, where by $\tw{\breve{\alpha_T}}x$ we
mean $\breve{\alpha_T }\circ x \circ \breve{\alpha_T }^{-1}\in\Aut(\brT)$.
A similar remark holds about the choice of restriction $\alpha_T$.
\end{rem} 

%%%%%%%%%%%%%%%%%%%%%%%%%%%%%%%%%%%%
\subsection{Homotopy fixed point centralisers}   \label{subsec:centrII}
We continue with the notation of this section. So $X$ is a connected
$\ell$-compact group with torsion-free fundamental group, maximal torus
$i: T\to X$ and Weyl group $(W,L)$. Let $\tau$ be a self equivalence whose order
in $\Out(X)$ is finite. Let $S$ be a finite $\ell$-group and, for some
$q\in\ZZ_\ell^\times$, let $f:BS\to\twt\coX(q)$ be a morphism of loop spaces.
For $s\in S$, define the \emph{centralisers}
$$C_X(s):= C_X(\langle s\rangle,\iota\circ f)\quad\text{and}\quad
 C_{\twt X(q)}(s):= C_{\twt X(q)}(\langle s\rangle,f), $$
where $\iota$ is an in ~$(*)$ in Section~\ref{subsec:def fus}. Set 
$$W(s) := C_W(\breve{j}(s))$$
where $j: B\langle s\rangle\to T$ is a factorisation of (the restriction to
$\langle s\rangle$ of) $\iota\circ f$ and $\breve{j}:\langle s\rangle\to\brT$
the discrete approximation of $j$. 
Recall from the previous section that $(W(s),L)$ is the Weyl group of the
connected $\ell$-compact group $C_X(s)$.

\begin{prop}   \label{prop:centr}
 With $X$, $S$ and $f:BS\to\twt\coX(q)$ as above, for any $s\in S$ we have:
 \begin{itemize}
  \item[(a)] With $\iota_\#$ denoting composition with $\iota$
  and $\tau\psi^q_\#$ denoting composition with
   $\tau\psi^q$, $C_{\twt X(q)}(s) \simeq C_X(s)^{h\tau\psi^q_\#}$.
  \item[(b)] If $\twt\coX(q)$ is the classifying space of a saturated fusion
   system $\cF$ on $S$ via $f$ then for any fully $\cF$-centralised $s\in S$,
   $C_{\twt X(q)}(s)$ is the classifying space of $C_\cF(s)$.
 \end{itemize}
\end{prop} 

Note that by Theorem~\ref{thm:bm} the assumption in (b) is satisfied if $X$ is
in addition simply connected and $\tau$ has $\ell'$ order.

\begin{proof}
In (a), it suffices to prove that $$\xymatrixcolsep{5pc}\xymatrix{C_{\twt X(q)}(s)\ar[d]^{\iota_\#} \ar[r]^{\iota_\#}& \quad C_X(s)\quad\ar[d]^\Delta\\ 
   C_X(s)\ar[r]^{(1,\tau\psi^q_\#)}& C_X(s)\times C_X(s)}$$
is a homotopy pullback diagram (see \cite[Rem.~2.3]{BMO}). Now $C_X(s)$ is
connected by
Theorem~\ref{thm:con cent}. Set $\beta=\iota\circ f$ and $\alpha = \tau\psi^q$.
By definition, $\Delta\circ\iota\simeq (1,\alpha)$. Composing with $f$ on the
right and with projection onto the second component on the left, we obtain
$\beta\simeq\alpha\circ\beta$. Now (a) follows by applying
\cite[Lemma~7.2]{BM07} with $f$ in place of $g$, and~(b) follows from
\cite[Thm.~6.3 and Prop.~2.5(c)]{BLO03}.
\end{proof}

Now assume $\ell>2$ and let $\phi\in N_{\GL(L)}(W)$ be the element (uniquely
determined up to multiplication by $W$) such that $\phi$ corresponds to the
conjugacy class of $\tau$ in $\Out(X)$, i.e., $W\phi$ is the image under
$\Phi_{X}$ of the homotopy class of $\tau\psi^q_\# $. Now let $s \in S$ and let
$\tau\psi^q_\#$ be as in Proposition~\ref{prop:centr}. 
Since $C_X(s)$ is connected by Theorem~\ref{thm:con cent} there exists
$\phi_s\in N_{\GL(L)}(W(s))$ (uniquely determined up to multiplication by
$W(s)$) such that $\phi_s\lambda_q$ corresponds to the conjugacy class of
$\tau\psi^q_\#$ via Theorem~\ref{thm:class}, where $\lambda_q=q\Id \in \GL(L)$,
i.e., $W(s)\phi_s\lambda_q$ is the image under $\Phi_{C_X(s)}$ of the
homotopy class of $\tau\psi^q_\# $. 

Lemma~\ref{l:identcent} may be used to identify $W(s) \phi_s$. Let $\tau_T$ be
a lift of $\tau$ to $ T$ and let $\psi^q_T $ be a lift of $\psi^q$ such that
$\breve{\psi^q_T}$ is the map $x\mapsto x^q$, $x\in\brT$. 
As explained in the proof of Proposition~\ref{prop:centr}, $\iota\circ f$ is
conjugate to $\tau\psi^q\circ\iota\circ f$. Hence, by
Lemma~\ref{l:identcent} applied with $\alpha=\tau\psi^q$, $Q=\langle s\rangle$
and $\iota\circ f$ replacing $f$, there exists $w \in W$ such that
$\tau_T\psi^q _T\circ j$ is conjugate to $ w\circ j$ and for any such $w$,
$w^{-1} \tau_T\psi ^q_T $ is a lift of $\tau \psi ^q_{\#}$ to $T$. Moving over
to discrete approximations, it follows that 
$w^{-1}\breve{\tau_T}\breve{\psi^q_T}$ centralises $\breve{j}(s)$ and
$w^{-1} \breve{\tau_T}$ normalises $W(s)$. Thus we may take for $\phi_s$ the
element $w^{-1}\breve{\tau_T}$.
\medskip

We conjecture the following order formula for centralisers (compare with
Theorem~\ref{thm:sylow}(b)).

\begin{conj}   \label{c:spetsandcomp}
 In the situation of Proposition~\ref{prop:centr} assume that $\ell>2$, $X$ is
 simply connected and $\tau$ has order prime to $\ell$. Then for any fully
 $\cF(\twt X(q))$-centralised $s\in S$,
 $$ \nu_\ell(|C_S(s)|) = \nu_\ell(\OO_q(W(s)\phi_s^{-1}))$$
\end{conj}

In Proposition~\ref{p:spetsgroups} below we show that the conclusion of the
conjecture
holds whenever the data $(W,L,\tau,q)$ corresponds to a finite group of Lie
type and in Lemma~\ref{l:identifyws2} we show that it holds for generalised
Grassmannians $X(e,r,n)(q)$ when $\tau=1 $ and $q\equiv 1\pmod\ell$.

%%%%%%%%%%%%%%%%%%%%%%%%%%%%%%%%%%%%
\subsection{Finite groups of Lie type.}\label{s:lie}
We record here the translation of the constructions above for finite groups of
Lie type. The discussion below is an enhancement of Remark~\ref{r:fried} and is
a combination of results of Quillen, Friedlander, and Mislin.

Let $p\ne\ell$ be a prime, $\bG$ a connected reductive group over
$\overline\FF_p$ with maximal torus~$\bT$, Weyl group $W$ and cocharacter
lattice $L_0$. Let $X= B\tilde \bG^\wedge_\ell$ where $\tilde\bG$ is a
connected reductive group over $\CC$ with the same root datum as $\bG$. Then
$\coX$ is a connected $\ell$-compact group with Weyl group
$(W,\ZZ_\ell\otimes_\ZZ L_0)$ and underlying maximal torus $i: T\to\coX$.
Let $B\brT\to T$ be a discrete approximation with
$\brT = \ZZ/\ell^\infty \otimes L_0$.
We identify the subgroup of $\bT$ consisting of all $\ell$-power elements with
$\brT$ via the restriction of the $W$-equivariant isomorphism
$\bT\cong\overline\FF_p^\times\otimes_\ZZ L_0$. Let $F_0:\bG\to\bG$ be a
standard Frobenius morphism such that $\bT$ is $F_0$-stable and $F_0(x) =x^p$
for $x\in\bT$. For each $d\geq 0$ there exist maps 
$$\iota_{d} : (B\bG^{F_0^d})^\wedge_\ell\to X$$
such that:
\par
(a) The restriction of $B \brT \to T $ to $B(\brT \cap \bG^{F_0^d})$ is a
 factorisation of the restriction of $\iota_d$ to $B (\brT \cap \bG^{F_0^d})$
 where we identify $B(\brT\cap\bG^{F_0^d})$ to its $\ell$-completion. 
\par
(b) For all $d,d'$ such that $d\mid d'$, we have a homotopy commutative diagram
 $$\xymatrix{ (B\bG^{F_0^d} )^\wedge_\ell \ar[d]_{\inc}\ar[r]^{\ \ \ \iota_d}& X \ar[d]^{1}\\ 
(B\bG^{F_0^{d'}} )^\wedge_\ell \ar[r]^{\ \ \ \iota_{d'}} & X }$$
where the left map is inclusion and the right hand side map is the identity
(see \cite[Corollary~1.3]{FM84}). 
\par
Let $q$ be a power of $p$ and $F:\bG\to\bG$ a Frobenius morphism with respect
to an $\FF_q$-structure such that $T$ is $F$-stable. Then there exists a
self-equivalence $\tau$ of $X$ (induced from a finite order automorphism of
$\tilde\bG$), a map
$$\iota:(B\bG^F)^\wedge_\ell\to X$$
and a positive integer $r$ with $F^r=F_0^r $ such that 
$$\xymatrix{ (B\bG^{F} )^\wedge_\ell \ar[d]_{\inc}\ar[r]^{\ \ \ \iota}& X \ar[d]^{1}\\ 
(B\bG^{F_0^{r}} )^\wedge_\ell \ar[r]^{\ \ \ \iota_{r}} & X }$$
is homotopy commutative and
$$\xymatrix{ (B\bG^F)^\wedge_\ell \ar[d]^{\iota}\ar[r]^{\iota}& X \ar[d]^\Delta\\
 X \ar[r]^{(1,\tau \psi^q)} & X \times X}$$
is a homotopy pullback. In particular, $(B\bG^F)^\wedge_\ell\simeq X^{h\tau\psi^q}$ and 
consequently $\cF(\twt X(q))=\cF_S(\bG^F)$ for a Sylow $\ell$-subgroup $S$ of
$\bG^F$ (see \cite[Thm.~3.1]{BMO}). Moreover, setting $L=\ZZ_\ell\otimes_\ZZ L_0$ and $\phi :=1\otimes F \in \GL(L)$ we have that $W\phi=\Phi_X (\tau)$,
where $\Phi_X$ is as in Theorem~\ref{thm:class}.
\par 

Let $s\in \bG^F$ be an $\ell$-element. We explain how to identify the coset
$W(s)\phi_s $ of Conjecture~\ref{c:spetsandcomp}. Since $s=F(s)$, $s$ is
contained in an $F$-stable maximal torus $\bT_w :=g\bT g^{-1}$ of $\bG$, with
$g^{-1}F(g)= w\in N_\bG(\bT)$ (see e.g. \cite[\S25.1]{MT}). We may assume that
$g\in G^{F^d}$ for some multiple $d$ of $r$. So, $g^{-1}sg\in\brT\cap \bG^{F^d}$. Let
$Bc_{g^{-1}}:\langle s\rangle\to B \brT\cap\bG^{F^d}$ be induced from
$c_{g^{-1}}:\langle s\rangle\to\brT\cap\bG^{F^d}$. By (a) and (b) above
we obtain the homotopy commutative diagram:
$$\xymatrix{ B\langle s \rangle \ar[d]_{\Id}\ar[r]^{\inc} &(B\bG^F)^\wedge_\ell \ar[d]_{\inc}\ar[r]^{\ \ \ \iota}& X \ar[d]^{\Id}\\ 
B\langle s\rangle \ar[rd]_{Bc_{g^{-1}}} \ar[r]^{\inc}&(B\bG^{F_0^d} )^\wedge_\ell \ar[r]^{\ \ \ \iota_{F_0^d}} & X \\
& B(\brT \cap \bG^{F_0^d} )\ar[u] _{\inc}&\\}$$
Thus $c_{g^{-1}}:\langle s \rangle\to\brT$ is the discrete approximation of a
factorisation of $\iota \circ\inc: B\langle s\rangle\to X$. It follows that 
$W(s) = C_W(t)$, where $t = g^{-1}sg\in\brT$. Moreover, the equation $F(s)= s$
translates to
$$ w\breve{\tau}_T (t^q) = w F (t) = t. $$
By the recipe for $\phi_s$ given at the end of the previous subsection,
we have $$\phi_s = w \phi. $$

\begin{prop}   \label{p:spetsgroups}
 With the notation above, suppose that $\ell$ is odd and $s\in S$ is fully
 $\cF(\tw\tau X(q))$-centralised. Then
 $$\nu_\ell(|C_S(s)|) = \nu_\ell(\OO_q(W(s)\phi_s^{-1})).$$
\end{prop} 

\begin{proof} 
We have $F \circ c_g = c_g\circ wF$ as maps from $\bT$ to $\bT_w$. So by
transport of structure $gC_W(t)g^{-1}$ is the Weyl group of $C_\bG(s)$ with
respect to the reference torus $\bT_w$, and $F$ acts on it as $w\phi$. 
By the Steinberg order formula for finite groups of Lie type,
$|C_\bG(s)^F|_\ell = (O_q( W(s)\phi_s^{-1}))_\ell$. Since $\cF(\tw\tau X(q))
 =\cF_S({\bG^F})$, the fully centralised assumption on $s$ means that $C_S(s)$
is a Sylow $\ell$-subgroup of $C_\bG(s)^F$. The result follows. 
\end{proof}

The next result will be used in section~\ref{sec:GGr}.

\begin{lem}   \label{l:factor}
 With the above notation, let $s\in\bG^{F_0^d}$ be an $\ell$-element. If
 $j: B\langle s\rangle\to T$ is a factorisation of
 $\iota_d\circ\inc: B \langle s\rangle\to X$ then $\breve{j}(s)$ and $s$
 are $\bG$-conjugate.
\end{lem} 

\begin{proof}
As above, there exist $d\geq 0$ and $g\in\bG^{F^d}$ such that
$gsg^{-1}\in\brT$, so $Bc_{g}:B\langle s\rangle\to T$ is a factorisation of
$\iota_d\circ\inc:B\langle s \rangle\to X$ with discrete approximation $c_g$.
If $j: B \langle s\rangle\to T$ is another factorisation, then by
Lemma~\ref{l:toral}, $j$ is conjugate to $w\circ Bc_g$ for some $w\in W$. By
uniqueness of discrete approximations, $\breve{j} = w\circ c_g$ and the result
follows.
\end{proof}

%%%%%%%%%%%%%%%%%%%%%%%%%%%%%%%%%%%%%%%%%%%%%%%%%%%%%%%%%%%%%%%%%%%%%%%%%
\section{Weight conjectures for Spetses}   \label{sec:owc}
So far, we have stayed on the ``local'' side of fusion systems. Now as for
finite groups we want to compare this to some ``global'' side, for which we
need to introduce suitable global objects. This will be the spets, and as a
special case, a finite reductive group.

%%%%%%%%%%%%%%%%%%%%%%%%%%%%%%%%%%%%
\subsection{Spetses and their unipotent degrees}   \label{subsec:spets}
The term ``spetses'' denotes a hypothetical mathematical object resembling a
connected reductive linear algebraic group but whose Weyl group is a
non-rational, finite complex reflection group. Attached to a spets is a
collection of data that behave like combinatorial data arising in the
representation theory of finite reductive groups, more precisely, analogues of
unipotent characters, subdivided into Harish-Chandra series, with Fourier
matrices attached to them, and the like. Such data have been constructed for
the so-called \emph{spetsial} complex reflection groups: a proper subset of all
complex reflection groups defined via certain integrality or rationality
properties of their associated Hecke algebras, see \cite[\S3]{MaICM},
encompassing and much larger than the collection of all rational and even all
real reflection groups.

The spetsial complex reflection groups are the direct products of
irreducible complex reflection groups of the following types:
\begin{itemize}
\item a real irreducible reflection group;
\item $G(e,1,n)$ or $G(e,e,n)$ for some $n\ge 1$, $e\ge3$; or
\item $G_n$, with $n\in \{4, 6, 8, 14, 24, 25, 26, 27, 29, 32, 33, 34\}$
\end{itemize}
(see \cite[\S8]{Ma00}). We say that a $\ZZ_\ell$-reflection group $(W,L)$ is
\emph{spetsial} if its extension to $\CC$ is spetsial. 

\begin{defn}   \label{d:spetsorder}
 A \emph{$\ZZ_\ell$-spets} $\GG=(W\phi,L)$ is a spetsial $\ZZ_\ell$-reflection
 group $(W,L)$ together with an element $\phi\in N_{\GL(L)}(W)$. The \emph{order
 (polynomial)} of $\GG$ is
 $$|\GG|:=x^N\big(\prod_{i=1}^r\eps_i\big)^{-2}\OO_x(W\phi)
   =x^N\prod_{i=1}^r\eps_i^{-2}(x^{d_i}-\eps_i)\in\ZZ_\ell[x],$$
 where $\{(d_i,\eps_i)\mid 1\le i\le r\}$ are the generalised degrees of
 $W\phi$ on $\overline\QQ_\ell\otimes_{\ZZ_\ell}L$ as in
 Section~\ref{subsec:sl} and $N$ is the number of reflections in $W$ (see
 \cite[Def.~1.44]{BMM14}).
\end{defn}

This is a subclass of spetses constructed from arbitrary complex spetsial
reflection groups (see \cite{BMM99}), augmented by also fixing a
$\ZZ_\ell$-lattice for~$W$. We will need the following observation:

\begin{prop}   \label{prop:centr spets}
 Let $(W,L)$ be a spetsial $\ZZ_\ell$-reflection group with $L_W$ torsion-free.
 Then for any $s\in\breve T=\ZZ/\ell^\infty\otimes_{\ZZ_\ell}L$, the
 $\ZZ_\ell$-reflection group $C_W(s)$ is spetsial.
\end{prop}

\begin{proof}
Arguing as in the proof of Theorem~\ref{thm:con cent} we see that
$(C_W(s),L)$ is a reflection subgroup of $(W,L)$.
For Weyl groups, all reflection subgroups are again Weyl groups. So by
\cite[Thm.~11.1]{AGMV} we may assume that $W$ is exotic. Our argument here is
case-by-case. For Clark--Ewing groups, centralisers are parabolic subgroups,
and parabolic subgroups of spetsial groups are spetsial by \cite[\S8]{Ma00}. 
Next assume that $W=G(e,r,n)$ is in the infinite series, with $3\le e|(\ell-1)$,
and $r\in\{1,e\}$ as $W$ is spetsial. Then Proposition~\ref{prop:stab} shows
the claim. So we are left to consider $W=G_{29}$ and $W=G_{34}$, at $\ell=5,7$
respectively. As mentioned in the proof of Proposition~\ref{prop:stab}, here
the centralisers of non-trivial elements $s$ are either symmetric groups
$\fS_\ell$, hence spetsial, or of order prime to $\ell$ and thus parabolic.
\end{proof}

Attached to a ($\ZZ_\ell$-) spets $\GG=(W\phi,L)$ is a set $\Uch(\GG)$ of
\emph{unipotent characters} \cite{Ma95,BMM14}. Each $\ga\in\Uch(\GG)$ has a
\emph{degree} $\ga(1)$ which is a polynomial with coefficients in the
character field of $W$ on $L$. Thus, in particular if $W$ is a
$\ZZ_\ell$-reflection group then $\ga(1)\in\QQ_\ell[x]$.
One unipotent character that always exists is the \emph{trivial character}
$1_\GG$, with degree $1_\GG(1)=1$.

The unipotent degrees are not necessarily contained in $\ZZ_\ell[x]$. Recall
our Definition~\ref{def:bad} of bad primes. One observes the following
by a case-by-case check (see the formulas in \cite{Ma95} and the lists in
\cite{BMM14}, respectively, and also compare with \cite[Conj.~8.3]{BCM} about
'bad ideals'):

\begin{prop}   \label{prop:unip integral}
 Let $\GG=(W\phi,L)$ be a $\ZZ_\ell$-spets. Then $\ell$ is good for $W$ if and
 only if the degrees of all unipotent characters of $\GG$ have $\ell$-adically
 integral coefficients.
\end{prop}

Let $\GG=(W\phi,L)$ be a $\ZZ_\ell$-spets. Then, for any root of unity
$\zeta\in\ZZ_\ell^\times$  the set of unipotent characters $\Uch(\GG)$ is
naturally partitioned into so-called $\zeta$-Harish-Chandra series, and one
among them, the \emph{principal $\zeta$-Harish-Chandra series
$\cE(\GG,1,\zeta)$} of $\Uch(\GG)$ containing $1_\GG$, is in bijection with
$\Irr(W_{\phi\zeta^{-1}})$ (see \cite[Folg.~3.16 and~6.11]{Ma95} and
\cite[4.31]{BMM14}), where $W_{\phi\zeta^{-1}}$ is the corresponding relative
Weyl group, see Section~\ref{subsec:sl}. 

\begin{exmp}   \label{ex:weyl}
 Let $q$ be a prime power. If $W$ is a Weyl group, then letting $\bG$ be a
 connected reductive algebraic group defined over $\FF_q$ with Weyl group $W$,
 and $F:\bG\to\bG$ the corresponding Frobenius endomorphism, acting as $\phi$
 on $W$ as in Remark~\ref{r:fried}(a), then $|\bG^F| = |\GG|_{x=q}$, where
 $\GG = (W\phi^{-1},L)$ and $\Uch(\GG)$ can be identified with the set of
 unipotent
 characters (in the sense of Lusztig, see e.g. \cite[Def.~2.3.8]{GM20}) of the
 finite reductive group $\bG^F$, such that the degree of $\ga\in\Uch(\GG)$ is
 precisely the polynomial $\ga(1)\in\QQ[x]$ for which $\ga(1)|_{x=q}$ is the
 degree of the corresponding unipotent character of $\bG^F$. Further, if $q$
 has order $e$ modulo~$\ell$, then for any primitive $e$th root of unity $\zeta$
 in $\ZZ_\ell$, the principal $e$-Harish-Chandra series of $\bG^F$ is in
 bijection with the principal $\zeta $-Harish-Chandra series of $\GG$.
\end{exmp}

There is the following important ``local control'' of heights for unipotent
characters of spetses (the Weyl group case of which was used in the proof of
\cite[Thm.~4.7]{BM15}):

\begin{prop}   \label{prop:heights}
 Let $\GG=(W\phi,L)$ be a $\ZZ_\ell$-spets, $q$ a prime power and
 $\zeta\in\ZZ_\ell^\times$ the root of unity with $q\equiv\zeta\pmod\ell$
 (respectively $q\equiv\zeta\pmod4$ when $\ell=2$). Then there is a bijection
 $\Psi:\cE(\GG,1,\zeta)\buildrel{\text{$1$--$1$}}\over\lra
  \Irr( W_{\phi\zeta^{-1}})$ with
 $$\nu_\ell\big(\ga(1)|_{x=q}\big)=\nu_\ell\big(\Psi(\ga)(1\big))\qquad
    \text{for all $\gamma\in\cE(\GG,1,\zeta)$}.$$
 In particular, the distribution of $\ell$-heights in the principal
 $\zeta$-Harish-Chandra series depends solely on the relative Weyl group
 $W_{\phi\zeta^{-1}}$.
\end{prop}

\begin{proof}
By \cite[Folg.~3.16 and~6.11]{Ma95} and \cite[4.31]{BMM14} there is a
bijection $\Psi:\cE(\GG,1,\zeta)\to\Irr( W_{\phi\zeta^{-1}})$ such that
$\ga(1)|_{x=\zeta}=\pm\Psi(\ga)(1)$, for suitable signs depending on $\ga$.
Furthermore, $\ga(1)$ is a constant times a product of cyclotomic polynomials
over $\ZZ_\ell$; by the afore-mentioned specialisation property, none of these
factors is $x-\zeta$.   \par
Writing $\tga(1):=\ga(1)(\zeta^{-1}x)$ and $\tq:=\zeta^{-1}q$ we now claim that
$$\nu_\ell(\ga(1)|_{x=q})=\nu_\ell(\tga(1)|_{x=\tq})=\nu_\ell(\Psi(\ga)(1)),$$
where no factor of $\tga(1)$ equals $x-1$. Since $\tq\equiv1\pmod\ell$, a
polynomial $f\in\ZZ_\ell[x]$ has $f(\tq)\equiv0\pmod\ell$ if and only if
$f(1)\equiv0\pmod\ell$. Now any cyclotomic polynomial over $\QQ$ is a product
of factors $(x^m-1)/(x-1)$ and their inverses, and as $\Phi_m(\tq)$ and
$\Phi_m(1)$ are divisible by $\ell$ at most once for any $m>1$, any
cyclotomic polynomial over $\QQ$ different from $x-1$ has at most one factor
$f$ over $\ZZ_\ell$ with $f(\tq)$ (and $f(1)$) divisible by $\ell$. Since
for any $m\ge1$, $(\tq^m-1)/(\tq-1)$ and $((x^m-1)/(x-1))|_{x=1}$ are divisible
by the same power of $\ell$ (by an application of L'Hospital's rule), our claim
now follows.
\end{proof}

%%%%%%%%%%%%%%%%%%%%%%%%%%%%%%%%%%%%
\subsection{The fusion system and principal block of a $\ZZ_\ell$-spets}
For the rest of this section, let $\ell>2$ be a prime. Let $\GG=(W\phi^{-1},L)$
be a simply connected $\ZZ_\ell$-spets (that is, $(W,L)$ is simply connected)
with $\phi$ of $\ell'$-order.
For $q$ a prime power not divisible by~$\ell$ set $\GG(q):=(W\phi^{-1},L,q)$. 
Let $X$ be a connected $\ell$-compact group with Weyl group $(W,L)$, and let
$\tau$ be a homotopy self-equivalence of $X$ with $\Phi_X(\tau)=W\phi$ (see
Theorem~\ref{thm:class}).
By Theorem~\ref{thm:bm}, $\twt X(q)$ is the classifying space of a saturated
fusion system, $\cF(\twt X(q))$ and we set $\cF(\GG(q)):=\cF(\twt X(q))$. By
Theorem~\ref{thm:sylow}(b) and Definition~\ref{d:spetsorder} if $S$ is the
underlying $\ell$-group of $\cF(\GG(q))$, then $|S|= (|\GG|_{x=q})_\ell$. 

We let $f:B S\to \twt X(q)$ and $\iota:\twt X(q)\to X$ be as defined in
Section~\ref{subsec:centrII}. For $s\in S$ let $j: B\langle s\rangle\to T$ be a
factorisation of (the restriction to $\langle s\rangle$ of) $\iota \circ f$ and
$\breve{j}:\langle s\rangle\to\brT$ the discrete approximation of $j$ and let
$W(s)=C_W(\breve{j}(s))$ and $\phi_s$ be as defined after
Proposition~\ref{prop:centr}; recall that $\phi_s\in W\phi$ normalises $W(s)$. 
By Proposition~\ref{prop:centr spets}, the $\ZZ_\ell$-reflection group $W(s)$
is spetsial. Let $C_\GG(s):=(W(s)\phi_s^{-1},L)$ be the associated
$\ZZ_\ell$-spets, with unipotent characters $\Uch(C_\GG(s))$.
By Lemma~\ref{l:identcent} and Remark~\ref{r:welldef}, the isomorphism class of
$C_\GG(s)$ is independent of the choice of $j$. Moreover, we have:

\begin{lem}
 The order polynomial $|C_\GG(s)|$ divides $|\GG|$ in $\ZZ_\ell[x]$.
\end{lem}

\begin{proof}
Let $\xi\in\overline\QQ_\ell^\times$ be a zero of $|\GG|$. The Sylow
theorems \cite[Thm.~3.4]{BM92} show that the multiplicity of $x-\xi$ as a
factor of $|\GG|$ equals the maximal dimension of a $\xi$-eigenspace of elements
$g\in W\phi^{-1}$, and similarly for $|C_\GG(s)|$. As
$W(s)\phi_s^{-1} \subseteq W\phi ^{-1}$, the claim follows. 
\end{proof}

We let
$$\cE(\GG(q),s)=\{\ga_{s,\la}^q\mid\la\in\Uch(C_\GG(s))\}$$
denote a set in bijection with $\Uch(C_\GG(s))$ and call it the
\emph{characters of $\GG(q)$ in the series~$s$}. The sets $\Uch(C_\GG(s))$ are
in canonical bijection for conjugate elements $s$. The degree of $\ga_{s,\la}^q$
is defined as
$$\ga_{s,\lambda}^q(1):=\big(|\GG:C_\GG(s)|_{x'}\,\la(1)\big)|_{x=q}
  \in\QQ_\ell.$$
Here, $|\GG:C_\GG(s)|_{x'}$ means the prime-to-$x$ part of the polynomial
$|\GG|/|C_\GG(s)|\in\ZZ_\ell[x]$.

This is inspired by (and specialises to) Lusztig's formula for Jordan
decomposition of characters (see e.g. \cite[Thm.~2.6.4]{GM20}) in the case of
rational spetses, i.e., finite reductive groups. Note that our notion of
characters of $\GG(q)$ in the series~$s$ is different from that in finite
reductive groups where $s$ is taken to be an element in the dual group. However
as explained in the proof of Proposition~\ref{prop:linkowc}, our assumptions on
$\ell$ allow for this change.

When $s=1$ we just write $\la^q$ for $\ga_{1,\la}^q$, where now
$\la^q(1)=\la(1)|_{x=q}$, and call $\cE(\GG(q),1)$ the unipotent characters of
$\GG(q)$. The \emph{$\ell$-defect of $\ga_{s,\la}^q\in\cE(\GG(q),s)$} is
defined to be the $\ell$-adic valuation
$$\nu_\ell(|\GG|_{x=q}) - \nu_\ell(\ga_{s,\la}^q(1))
   = \nu_\ell(|C_\GG(s)|_{x=q}) - \nu_\ell(\lambda(1)_{x=q}). $$

Further, for $\zeta\in\ZZ_\ell^\times$ the root of unity with
$q\equiv\zeta\pmod\ell$ we denote by
$\cE(\GG(q),s)_1$ the subset of $\cE(\GG(q),s)$ in bijection with the principal
$\zeta$-Harish-Chandra series of $\Uch(C_\GG(s))$. By
Proposition~\ref{prop:heights}, $\cE(\GG(q),s)_1$ is in bijection with
$\Irr(W(s)_{\phi_s^{-1}\zeta^{-1}})$. The following definition is inspired by
the results of Cabanes--Enguehard \cite[Thm.]{CE94} on unipotent $\ell$-blocks
of finite reductive groups; indeed, if $\GG$ is a rational spets for which
$\ell$ is very good, then what we define are exactly the characters in the
principal $\ell$-block of the corresponding finite group of Lie type (see
Proposition~\ref{prop:linkowc} below):

\begin{defn}   \label{d:irrb0}
 Let $\GG=(W\phi^{-1},L)$ be a simply connected $\ZZ_\ell$-spets such that
 $\ell$ is very good for $\GG$. Then define the \emph{characters in the
 principal block $B_0$ of $\GG(q)$} as
 $$\Irr(B_0):=\coprod_{s\in S/\cF}\cE(\GG(q),s)_1,$$
 where the union runs over a set $S/\cF$ of fully centralised representatives
 $s$ of $\cF(\GG(q))$-conjugacy classes in $S$. For a non-negative integer $d$
 we let
 $$\Irr^d(B_0):=\{\chi\in\Irr(B_0)\mid
                  d=\nu_\ell(|\GG|_{x=q})- \nu_\ell(\chi(1))\},$$
 the set of irreducible characters in $B_0$ of $\ell$-defect $d$.
\end{defn}

\begin{prop}   \label{prop:linkowc}
 Let $\bG$ be a semisimple algebraic group defined over $\FF_q$ with respect to
 a Frobenius endomorphism $F:\bG\to\bG$. Suppose that $F$ has $\ell'$-order
 in its action on the Weyl group and that $\ell>2$ is a very good prime
 for~$\bG$. Let $B_0(G)$ be the principal $\ell$-block of $G:=\bG^F$ and let
 $B_0$ be the principal block of the associated spets $\GG(q)$. Then
 $$ |\Irr^d(B_0 (G) )| = |\Irr^d(B_0)|\qquad\text{for all $d\ge0$}. $$
\end{prop}

\begin{proof}
Let $(\bG^*,F)$ be dual to $(\bG,F)$ and let $S^*$ be a Sylow $\ell$-subgroup
of $G^*:=\bG^{*F}$. Then
$$\Irr(B_0(G))=\coprod_{s}\big(\Irr(B_0(G))\cap\cE(G,s)\big)$$
where the union runs over $G^*$-classes of elements $s\in S^*$ and
$\cE(G,s) \subseteq\Irr(G)$ is the Lusztig series corresponding to the
$G^{*}$-class of $s$. Let $e$ be the order of $q$ modulo~$\ell$. Then,
according to the description of characters in the principal $\ell$-block of $G$
for very good primes in \cite[Thm.]{CE94}, $\Irr(B_0(G))\cap\cE(G,s)$ is in
defect preserving bijection with the principal $e$-Harish-Chandra series,
$\cE(\bL(s)^F,1)_1$ of $\bL(s)^F$ where $\bL(s)$ is an $F$-stable Levi subgroup
of $\bG$ in duality with the Levi subgroup $C_{\bG^*}(s)$ of $\bG^*$.

Now let $S$ be a Sylow $\ell$-subgroup of $G$. By \cite[Prop.~4.2]{GH91}
under our assumption on~$\ell$ there is a bijection between $G$-classes in
$S$ and $G^*$-classes in $S^*$ preserving centralisers through duality.
Hence $\Irr(B_0(G))$ is in defect preserving bijection with
$\coprod_s\cE(C_G(s),1)_1$ where now the union runs over $G$-classes of
$\ell$-elements in $S$ or alternatively (see Section~\ref{s:lie}) over
fully centralised representatives of $\cF(\GG(q))=\cF_S(G)$-classes of $S$.
By the discussion above Proposition~\ref{p:spetsgroups}, $(C_\bG(s),F)$
corresponds to the spets $C_\GG(s)(q)$. Now the claim follows from
Example~\ref{ex:weyl}.
\end{proof}

\begin{rem}   \label{rem:wreath}
Let us comment on the restrictive assumption on the order of $\phi$.
We can associate a saturated fusion system to any spets as follows. For
irreducible $(W,L)$ we proceed as above using Remark~\ref{rem:order}. Now assume
that $(W,L)=(\prod_{i=1}^r W_i,\bigoplus_{i=1}^r L_i)$ with irreducible factors
$(W_i,L_i)$ permuted transitively by~$\phi$. Then the unipotent characters of
$\GG$ are in bijection with those of $\GG_1=(W_1\phi^{-r}|_{L_1},L_1)$, and we
may set $\cF(\GG(q)):=\cF(\GG_1(q^r))$. Finally, if
$(W,L)=(W_1\times W_2,L_1\oplus L_2)$ is a $\phi$-stable direct decomposition,
then by \cite[4.1]{BMM14} the unipotent characters of $\GG$ are the Cartesian
product of those of $\GG_i=(W_i\phi^{-1}|_{L_i},L_i)$, $i=1,2$, and their
degrees multiply. We then set $\cF(\GG(q)):=\cF(\GG_1(q))\times\cF(\GG_2(q))$.
In this way, we have related a fusion system to any $\GG$ and $q$, but this is
not always known to arise from the known homotopy fixed point construction on
$\ell$-compact groups.
\end{rem}

%%%%%%%%%%%%%%%%%%%%%%%%%%%%%%%%%%%%
\subsection{The ordinary weight conjecture for spetses}   \label{s:ordspets}
If $\cF$ is a saturated fusion system on a finite $\ell$-group $S$, we now
recall functions $\bk(\cF)$, $\bm(\cF,d)$ for $d$ a non-negative integer, and
$\bm(\cF)=\sum_{d\ge0}\bm(\cF,d)$ associated to $\cF$ defined in Sections~1
and~2 of \cite{KLLS19}, where $\bk(\cF):=\bk(\cF,\alpha)$ etc.~when the
associated K\"{u}lshammer--Puig family of classes $\alpha$ is zero. These
invariants appear in conjectures concerning the number of characters in the
principal $\ell$-block of a finite group with fusion system $\cF$.

The integer $\bm(\cF,d)$ is an alternating sum count of projective simple
modules associated to stabilisers of certain pairs $(\sigma,\mu)$, where
$\sigma$ is a chain of $\ell$-subgroups in the (outer) $\cF$-automorphism group
of some $\cF$-centric subgroup $Q\le S$ and $\mu$ is an irreducible character
of $Q$ of defect~$d$. In fact, by \cite[Lemma~7.5]{KLLS19} only radical centric
subgroups $Q\in\cF^\CR$ will contribute to the sum.

More precisely, recall from \cite[\S2]{KLLS19} that
\[ \bm (\cF, d) = \sum_{Q\in\cF^\CR/\cF}\bw_Q (\cF, d),\]
where $\bw_Q(\cF,d)$ is the contribution coming from the set $\cN_Q$ of
non-empty normal chains of $\ell$-subgroups of $\Out_\cF(Q)$ of the form
$(1=X_0<X_1 <\cdots<X_m)$.

If $B_0$ is the principal block of a finite group with fusion system $\cF$ then 
Robinson's ordinary weight conjecture (OW conjecture) is the assertion that $\bm(\cF,d)$
counts the number of ordinary irreducible characters of defect $d$ in $B_0$.
This leads us to the following analogue for spetses (Conjecture~\ref{c:conjd}
from the introduction). By Proposition~\ref{prop:linkowc}, if $\GG$ is
associated to a finite group $\bG^F$ as in Example~\ref{ex:weyl}, then the
assertion of the conjecture is equivalent to the assertion that the OW
conjecture holds for the principal block of $\bG^F$.

\begin{conj}[OW conjecture for spetses]   \label{c:spetsowc}
 Let $\GG$ be a simply connected $\ZZ_\ell$-spets such that $\ell >2 $ 
 is very good for $\GG$, and $q$ be a power of a prime different from~$\ell$.
 Let $B_0$ be the principal block of $\GG(q)$. Then
 $$|\Irr^d(B_0)|=\bm(\cF(\GG(q)),d)\qquad\text{for all $d\ge0$}.$$
\end{conj}

From the local control of heights in Proposition~\ref{prop:heights} we obtain
the following expression for the left hand side in Conjecture~\ref{c:spetsowc}:

\begin{prop}   \label{prop:2owc2}
 Let $\GG =(W\phi^{-1},L)$ be a simply connected $\ZZ_\ell$-spets such that
 $\ell>2$ is very good for $\GG$. Let $q$ be a power of a prime different
 from~$\ell$, $\cF=\cF(\GG(q))$ the associated saturated fusion system on $S$
 and $B_0$ the principal block. Then
 $$|\Irr^d(B_0)|
   =\sum_{s\in S/\cF} |\Irr^{d-u_s}(W(s)_{\phi_s^{-1}\zeta^{-1}})|$$
 where $\zeta\in\ZZ_\ell^\times$ is the root of unity with
 $q\equiv\zeta\pmod\ell$, and
 $u_s= \nu_\ell(|C_\GG(s)|_{x=q}) - \nu_\ell( |W(s)_{\phi_s^{-1}\zeta^{-1}}|)$.
\end{prop}

\begin{proof}
Suppose $\chi \in \Irr^d(B_0)$, so $\chi=\gamma_{s,\la}^q$ for some
$\lambda\in \Uch(C_\GG(s))$, and $s\in S$ is fully $\cF$-centralised. 
By definition $d = \nu_\ell(|C_\GG(s)|_{x=q})- \nu_\ell(\la(1)_{x=q})$. Now the
result follows since $\nu_\ell(\la(1)_{x=q}) = \nu_\ell(\Psi(\la)(1))$ by
Proposition~\ref{prop:heights}.
\end{proof}

This leads us to the following reformulation of the OW conjecture which is new
even for principal blocks of finite groups of Lie type and which on the other
hand makes sense even in the non-spetsial case. It shows in particular that the
OW conjecture for finite groups of Lie type is a purely local statement.

\begin{conj}[OW conjecture for $\ell$-compact groups]   \label{conj:nr oweights}
 Let $X$ be a simply connected $\ell$-compact group with Weyl group $(W,L)$ for
 which $\ell >2$ is very good, and $\tau$ an automorphism of $X$ of finite order
 prime to $\ell$. Let $W\phi = \Phi_X (\tau)$, $q$ be a power of a prime
 different from~$\ell$ and $\cF=\cF(\twt X(q))$ the associated fusion system
 on~$S$. Then
 $$\bm(\cF,d)
   = \sum_{s\in S/\cF} |\Irr^{d- v_s}(W(s)_{\phi_s\zeta})|\qquad
   \text{for all $d \ge 0$},$$
 where $(W(s),L)$ is the $\ZZ_\ell$-reflection group underlying $C_X(s)$,
 $\zeta\in\ZZ_\ell^\times$ is the root of unity with $q\equiv\zeta\pmod\ell$,
 and $v_s= \nu_\ell(O_q(W(s) \phi_s^{-1})) - \nu_\ell(|W(s)_{\phi_s\zeta}|)$.
 \end{conj}

Note that since $W_{\phi_s^{-1} \zeta^{-1}} = W_{\phi_s\zeta}$, and since the
order formulas for spetses and reflection cosets agree at $\ell$, by
Proposition~\ref{prop:2owc2}, Conjecture~\ref{conj:nr oweights} is a
generalisation of Conjecture~\ref{c:spetsowc}.
%% Also note that if
%% Conjecture~\ref{c:spetsandcomp} holds, then $v_s = \nu_\ell(|S:C_S(s)|)$.

There is a coarser version of Conjecture~\ref{c:spetsowc} obtained by summing
over all $d$. Under the hypotheses of Conjecture~\ref{c:spetsowc}, this asserts
that
$$ |\Irr(B_0)|=\bm(\cF(\GG(q))).$$
Passing through localisation afforded by Proposition~\ref{prop:2owc2}, we obtain
under the hypothesis of Conjecture~\ref{conj:nr oweights}, the conjectural
equation 
$$\bm(\cF) = \sum_{s\in S/\cF} |\Irr(W(s)_{\phi_s\zeta})|$$
for $\ell$-compact groups which we call ``Summed OW conjecture'' or SOW
conjecture for short.

For $\cF$ a saturated fusion system on a finite group $S$ we set
$$\bk(\cF):=\sum_{s \in S/\cF} \bw(C_\cF(s)).$$
If $B_0$ is the principal block of a finite group with fusion system $\cF$ then,
by \cite[Prop.~4.5]{KLLS19} assuming the AW conjecture holds for all principal
blocks, $\bk(\cF)=\bk(B_0)$ is the number of ordinary characters in $B_0$.
In \cite{KLLS19} it has been conjectured that $\bm(\cF)=\bk(\cF)$ for an
arbitrary saturated fusion system~$\cF$.
This yields the following consequence of Theorem \ref{thm:nr weights}:

\begin{cor} 
 If the AW conjecture holds for all principal $\ell$-blocks of all finite
 groups then the SOW conjecture holds for $X(e,r,n)(q)$ with $r|e|(\ell-1)$ and
 $e\ge2$, for all prime powers $q\equiv 1\pmod\ell$.
\end{cor}

\begin{proof}
Let $X=X(e,r,n)$ and let $\cF$ be the fusion system corresponding to $X(q)$ via
Theorem~\ref{thm:bm}. Then, the required equality is
$$\bm(\cF) = \sum_{s\in S/\cF} |\Irr(W(s)_{\phi_s})|.$$
By our assumption on the validity of the AW conjecture, by \cite[Cor.~1.3]{KLLS19} we have
$\bm(\cF)=\bk(\cF)$. Note that \cite[Cor.~1.3]{KLLS19} assumes the AW conjecture for all
blocks of all finite groups but it can be easily seen from the proof that in the
trivial $2$-cocycle case, it suffices to assume the AW conjecture only for principal blocks.
Let $s\in S$ be fully $\cF$-centralised. By Proposition~\ref{prop:centr} and
Theorem~\ref{thm:centr GGR}, $C_\cF(s)$ is a direct product of fusion systems
$\cF_i$ where $\cF_i$ corresponds to some $X_i(q^{d_i})$ with $X_i$ an
$\ell$-compact group with Weyl group $G(e_i,r_i,n_i)$. Moreover,
$W(s)_{\phi(s)}=\prod_i G(e_i,r_i,n_i)$ by Lemma~\ref{l:identifyws2}. Applying
the weight count for generalised Grassmannians to each component $X_i$ (noting
that $q^{d_i}\equiv 1\pmod\ell$) and passing to direct products gives
$\bw(C_\cF(s)) = |\Irr(W(s)_{\phi_s})|$. Hence,
$$ \bm(\cF) =\bk(\cF)= \sum_{s \in S/\cF} \bw(C_\cF(s)) =
  \sum_{s \in S/\cF} |\Irr(W(s)_{\phi_s})|$$
where the second equality is by definition.
\end{proof}

A very similar argument shows the following, where we don't need to assume
the AW conjecture by using \cite[Rem.~7.8]{KLLS19}:

\begin{prop}   \label{prop:nrw for CE}
 Conjecture~\ref{conj:nr oweights} holds whenever $W$ has order prime to $\ell$
 and $q\equiv1\pmod\ell$.
\end{prop}

When $q \equiv 1 \pmod \ell$, we prove Conjecture~\ref{conj:nr oweights} for
the `smallest cases' in the infinite series not covered by the preceding result,
viz.\ the groups $W=G(e,r,\ell)$, in Section~\ref{sec:GGr ord}.

As discussed in Section~\ref{sec:awc}, the $\ell$-blocks of finite reductive
groups for not very good primes behave somewhat differently, so
Conjecture~\ref{conj:nr oweights} does not extend to arbitrary~$\ell$.
This is exemplified by the following, which will be shown in
Section~\ref{subsec:8.2}; see also Remark~\ref{rem:extrachar}:

\begin{prop}   \label{prop:owc Aguade}
 The conclusion of Conjecture~\ref{conj:nr oweights} holds for the four
 Aguad\'e exotic $\ZZ_\ell$-reflection groups, except for defect $d=2$.
\end{prop}

In the excluded case, when $d=2$ and $(W,\nu_\ell(q-1)) \neq (G_{12},1)$ we
have instead
$$\bm(\cF,d)
  = \sum_{s\in S/\cF} |\Irr^{d-v_s}(W(s)_{\phi_s\zeta})|+\ell.$$
Similarly, if $W=G_{24}$ and $\ell=2$ then \cite[Thm.~2.4]{S19} shows that the
conclusion of Conjecture~\ref{conj:nr oweights} holds except when $d=4$ or
$d=\nu_2(q^2-1)+4$ where it predicts $\bm(\cF,d)=0$. In fact $\bm(\cF,d)$ is
respectively $2$ and $4$ in these cases (see \cite[Table 1]{S19}).

\begin{rem}   \label{rem:extrachar}
 A natural way to account for the $\ell$ additional characters of defect~$2$
 in the principal $\ell$-blocks of Aguad\'{e} spetses $\GG_n(q)$ would be to
 expand the set $\Irr(B_n)$ to include the $\ell$ characters in $\Uch(\GG_n(q))$
 of defect $2$. This is well-justified since for the groups of Lie type
 $\tw2F_4(q^2)$ and $E_8(q)$, in which $G_{12}$ and $G_{31}$ occur as relative
 Weyl groups, the principal $\ell$-blocks for $\ell=3$, $5$ respectively, also
 contain unipotent characters outside the principal series: For $\tw2F_4(q^2)$,
 $q^2\equiv-1\pmod3$, according to \cite{Ma90} the principal $3$-block
 contains~11 unipotent characters, 8 of which lie in the principal 2-series
 (for $q^2$) and are in bijection with $\Irr(G_{12})$. Here by Himstedt
 \cite[Tab.~C5]{Hi11} we have $l(B_0)=9=|\Irr(G_{12})|+1$.   \par
 For $E_8(q)$ with $q\equiv\pm2\pmod5$, so $e=e_5(q)=4$, according to Enguehard
 \cite[Tab.~I]{En00} the principal 5-block contains the unipotent characters in
 the principal 4-series (in bijection with $\Irr(G_{31})$) plus five further
 unipotent characters of positive height.
\end{rem}

%%%%%%%%%%%%%%%%%%%%%%%%%%%%%%%%%%%%%%%%%%%%%%%%%%%%%%%%%%%%%%%%%%%%%%%%%
\section{Generalised Grassmannians}   \label{sec:GGr}
In this section we provide proofs of our main results for the infinite series
of imprimitive irreducible $\ZZ_\ell$-reflection groups $W=G(e,r,n)$,
$r|e|(\ell-1)$, $e\ge2$, where throughout, $\ell$ denotes an odd prime. For
uniformity of notation and because these groups occur naturally in our
discussion, we include the case $e=1$. So, $ G(1,1,n)$ denotes the
(non-irreducible) $\ZZ_\ell$-reflection group $(\fS_n,L)$ where $L =\ZZ_\ell^n$
is the permutation module of $\fS_n$. We denote by $\coX(e,r,n)$ the connected
$\ell$-compact group corresponding to $G(e,r,n)$ via Theorem~\ref{thm:class}. 
For any prime power $q$ prime to $\ell$, $\coX(1,1,n)^{h\psi^q}\simeq (B\GL_n(q))^\wedge_\ell$ 
and if $q'$ is a prime power such that $q'\equiv 1\pmod\ell$ and $q$ has
order~$e$ modulo $\ell$ with $\nu_\ell(q^e-1)=\nu_\ell(q'-1)$, then by a
theorem of Ruiz \cite{Ru07} described in detail in \S~\ref{subsub:descent},
$\cF(X(e,r,n)(q'))$ is a subsystem of the $\ell$-fusion system of $\GL_{en}(q)$.
Thus weight calculations for the system $\cF(X(e,r,n)(q'))$ can be done via
descent from the fusion system of $\GL_{en}(q) $.

%%%%%%%%%%%%%%%%%%%%%%%%%%%%%%%%%%%%
\subsection{Proof of Theorem~\ref{thm:nr weights} for generalised Grassmannians}   \label{subsec:GGr weights}
Here we show the validity of the AW conjecture for the fusion systems attached
to $G(e,r,n)$, generalising the results of Alperin and Fong \cite{AF90} for the
general linear groups.

\subsubsection{The set up.} 
Let $q$ be a prime power prime to $\ell$ and let $e$ be the order of $q$ modulo
$\ell$. Set $k=\FF_q $ and $ K=\FF_{q^e}$. Let $V'$ be an $n$-dimensional
$k$-vector space and set $V = K\otimes_k V'$. Then we may view $V$ both as
an $n$-dimensional $K$-vector space and as an $ne$-dimensional $k$-vector
space. We identify $\GL_K(V)\leq\GL_k(V)$ with $\GL_n(q^e)\leq\GL_{en}(q)$
and $\GL_n(q) = \GL_k(V')$ with a subgroup of $\GL_K(V)$ via the map
which sends $\phi\in\GL_k(V')$ to its unique $K$-linear extension to $V$.

Fix a direct sum decomposition
\[ V' = V_1'\oplus\cdots\oplus V_n ' \]
of $V'$ into one-dimensional $k$-spaces and for each $i$, choose a basis element
$v_i$ of $V_i'$. Set $V_i = K\otimes_k V_i'$ and let $\fS_n\leq\GL_k(V')$ be
the symmetric group on $\{v_1,\ldots,v_n\}$. We identify the $n$-fold direct
product $ K^\times\times\cdots\times K^\times$ with its image in $\GL_K(V)$ via
the map which sends $(\la_1,\ldots,\la_n)$ to the $K$-linear map 
\[\sum_{1\leq j\leq n}x_j\otimes v_j\mapsto
  \sum_{1\leq j\leq n}\lambda_jx_j\otimes v_j\quad\text{for } x_j\in K. \]

Let $\Delta\cong C_{\ell^a}$ be the Sylow $\ell$-subgroup of $K^\times$ and set
\[ D = \Delta\times\cdots\times\Delta\le\GL_K(V), \]
the Sylow $\ell$-subgroup of $K^\times\times\cdots\times K^\times$ ($n$
factors). Let
\[ n = a_0+a_1\ell +\cdots +a_m\ell^m\quad
  \text{with } 0\leq a_i <\ell, \]
be the $\ell$-adic decomposition of $n$. Choose a direct sum decomposition
\[ V' = \bigoplus_{0\leq i\leq m}\bigoplus_{1\leq j\leq a_i} V_{i,j}' \]
such that $\dim_k V'_{i,j} = \ell^i $ and each $V'_{i,j}$ is the sum of some
subset of $\{ V_t \mid 1\leq t\leq n\}$. Let
\[\prod_{0\leq i\leq m}\prod_{1\leq j\leq a_i}\fS_{\ell^i} \ \ \leq\fS_n\]
be the Young subgroup corresponding to the above decomposition and let $E$ be a
Sylow $\ell$-subgroup of $\fS_n$ contained in the above Young subgroup. Then,
\[ E\cong\prod_{0\leq i\leq m}\prod_{1\leq j\leq a_i} C_\ell \ \wr \stackrel{i\text{ times}}\cdots\wr\ C_\ell, \]
$ D\cap E =1$ and  %% in $\GL_K(V)$ and 
\[ S := D E\cong\prod_{0\leq i\leq m} (C_{\ell^a} \ \wr C_\ell \ \wr \stackrel{ i\text{ times}}{\cdots}\wr\ C_\ell)^{a_i} \]
is a Sylow $\ell$-subgroup of $\GL_k(V)$ contained in $\GL_K(V)$. Let $\cF$
denote the fusion system of $\GL_k(V)$ on $S$.
 
Let $\Phi$ be the generator of $\Gal(K/k)\cong C_e$ defined by $\Phi(x) = x^q$,
$x\in K$. Identify $\langle\Phi\rangle^n$ with its image in $\GL_k(V)$ via the
map which sends $(\Phi^{i_1},\ldots,\Phi^{i_n})$ to the $k$-linear map $V\to V$
defined by 
\[\sum_{1\leq j\leq n} x_j\otimes v_j\mapsto\sum_{1\leq j\leq n} x_j^{q^{i_j}} \otimes v_j\quad\text{for } x_j\in K. \]
Then
\[C_{\GL_k(V)}(D) = (K^\times)^n,\ 
  N_{\GL_k(V)} (D) = (K^\times)^n\langle\Phi\rangle^n\fS_n, \ 
  (K^\times)^n\cap\langle\Phi\rangle^n\fS_n= 1, \ 
  \langle\Phi\rangle^n\cap\fS_n =1\]
and $\fS_n$ normalises $\langle\Phi\rangle^n$. Thus the natural map
$N_{\GL_k(V)}(D)\to\Aut(D)$ induces an isomorphism 
\[\Out_\cF(D) = \Aut_\cF(D)\cong\langle\Phi\rangle^n\rtimes\fS_n\cong C_e\wr\fS_n .\]
We will identify $\Aut_\cF(D)$ with $\langle\Phi\rangle^n\rtimes\fS_n$ via
the above. Set $H = \langle\Phi\rangle^n$ the base subgroup of $\Aut_\cF(D)$
and set $H_0 =[H,\fS_n]$. Then
$$H_0=\{(\Phi^{i_1},\ldots,\Phi^{i_n})\mid i_1+\cdots+i_n\equiv 0\pmod e\}.$$
We define $\pi :\Aut_\cF (D)\to\ZZ/e\ZZ $ by 
\begin{equation}   \label{e:pi}
  (\Phi^{i_1},\ldots,\Phi^ {i_n})\sigma\mapsto i_1 + \cdots + i_n\pmod e\quad
  \text{for }(\Phi^{i_1},\ldots,\Phi^{i_n})\in H,\ \sigma\in\fS_n,
\end{equation} 
a surjective group homomorphism with kernel $H_0\fS_n$.

%%%%%%%%%%%%%%%%%%%%%%%%%%%%%%%%%%%%
\subsubsection{On the Alperin--Fong description of centric, radical and weight contributing subgroups}
We recall some of the key results of \cite{AF90} on counting weights in
principal blocks of finite general linear groups. For a saturated fusion system
$\cF$ on a finite $\ell$-group $S$ and a subgroup $R\leq S$, we say that $R$ is
\emph{$\cF$-weight contributing} if $R$ is $\cF$-centric and $\Out_\cF(R)$ has
an irreducible character of zero $\ell$-defect. We note some elementary facts.

\begin{lem}   \label{l:differentrads}
 Let $\cF$ be a saturated fusion system on a finite $\ell$-group $S$ and let
 $R\leq S$.
 \begin{enumerate}
  \item[\rm(a)] If $R$ is weight contributing, then $R$ is $\cF$-centric and
   $\cF$-radical.
  \item[\rm(b)] Suppose that $\cF=\cF_S(G)$ for some finite group $G$ with $S$
   as Sylow $\ell$-subgroup. Then $R$ is $\cF$-centric if and only if $Z(R)$ is
   a Sylow $\ell$-subgroup of $C_G(R)$. If $R$ is both $\cF$-centric and
   $\cF$-radical, then $R$ is $G$-radical, i.e., $O_\ell(G/R) =1$. In
   particular, if $R$ is $\cF$-weight contributing then $R$ is $G$-radical. 
 \end{enumerate}
\end{lem} 

Recall that $\ell\ne2$. Denote by $\cC$ the set of finite
sequences $\fc=(c_1,\ldots, c_t)$ of strictly positive integers including the
empty sequence $()$. For $\fc\in\cC $, write $|\fc|:=c_1 +\cdots + c_t$. For
$\fc\in\cC$ and non-negative integers $m,\ga,\alpha$, let $R_{m,\alpha,\ga,\fc}$
denote a corresponding \emph{basic subgroup} of $\GL_k(V)$ as defined in
\cite[Sec.~4]{AF90}. We will assume that the extra-special component of
$R_{m,\alpha,\ga,\fc}$ is of exponent $\ell$.

\begin{prop}   \label{p:AFradical}
 Let $R$ be a radical $\ell$-subgroup of $\GL_k(V)$. There exist decompositions
 \[ V = V_0\oplus V_1\oplus\cdots\oplus V_s,\qquad
    R =R_0\times R_1\times\cdots\times R_s, \]
 of $k$-vector spaces and of groups such that $R_0$ is the trivial subgroup of
 $\GL_k(V_0)$ and for each $i\geq 1$, $R_i = R_{m_i,\alpha_i,\ga_i,\fc_i}$
 is a basic subgroup of $\GL_k(V_i)$ with
 $\dim_k(V_i) = e\ell^{m_i +\alpha_i +\ga_i + |\fc_i|}$.
 \begin{enumerate}
  \item[\rm(a)] $Z(R)$ is a Sylow $\ell$-subgroup of $C_G(R)$ if and only if
   $V_0 =\{0\}$ and $m_i = 0$ for all $i\geq 2$.
  \item[\rm(b)] The $\GL_k(V)$-conjugacy classes of radical $\ell$-subgroups of
   $\GL_k(V)$ are in bijection with the set of assignments
   \[f:\NN_{+}\times\NN\times\NN\times\cC\to\NN\quad\text{such that}\quad
     \sum_{(m,\alpha,\ga,\fc)}\ell^{m+\alpha+\ga +|\fc|} f(m,\alpha,\ga,\fc)\leq n.\]
   Under this bijection, $f$ corresponds to the class of groups
   $R =\prod_{(m,\alpha,\ga,\fc)} R_{m,\alpha,\ga,\fc}^{f(m,\alpha,\ga,\fc)}$.
  \item[\rm(c)] The $\GL_k(V)$-classes of radical $\ell$-subgroups $R$ such
   that $Z(R)$ is a Sylow $\ell$-subgroup of $C_G(R)$ correspond to functions
   $f$ as in~$(b)$ satisfying $f(m,\alpha,\ga,\fc) =0 $ for $m\geq 2$ and
   \[\sum_{(m,\alpha,\ga,\fc)}\ell^{m+\alpha+\ga +|\fc|}f(m,\alpha,\ga,\fc) = n. \] 
 \end{enumerate} 
\end{prop}
 
\begin{proof}
All statements follow from Section 4 of \cite{AF90}. The first statement and~(b)
are given in \cite[(4A)]{AF90}. 
For (a), suppose first that $ V_0 $ is non-trivial. Since the $k$-dimension of
$V_i$ for $i\geq 1$ is divisible by $e$ and $\dim_k V = e n$, $e$ divides
$\dim_k(V_0)$
and it follows that $\GL_k(V_0)$ has a non-trivial Sylow $\ell $-subgroup, say
$E$. Since clearly $E\leq C_{\GL_k(V)} (R)$ and $ E\cap R =1$, it follows
that $R$ is not a Sylow $\ell $-subgroup of $ RC_{\GL_k(V)}(R)$, equivalently
that $ Z(R)$ is not a Sylow $\ell $-subgroup of $ C_{\GL_k(V)}(R)$. 

Hence we may assume from now on that $V_0=\{0\}$. So, $Z(R)$ is a Sylow
$\ell$-subgroup of $C_{\GL_k(V)}(R)$ if and only if for all $ i\geq 1$,
$Z(R_i)$ is a Sylow $\ell$-subgroup of $C_{\GL_k(V_i)}(R_i)$. Hence we may
assume that $i=1$ and $R= R_{m,\alpha,\ga,\fc}$. 
By construction $R = R_{m,\alpha,\ga}\wr A_\fc$ where
$R_{m,\alpha,\ga}:= R_{m,\alpha,\ga,()}$, $Z(R)$ is isomorphic to the
diagonally embedded copy of $Z(R_{m,\alpha,\ga})$ and $C_{\GL_k(V)} (R)$ is
isomorphic to the diagonally embedded copy of $C_{\GL_k(W)}(R_{m,\alpha,\ga})$
where $W$ is the subspace underlying $R_{m,\alpha,\ga}$.
Hence we may assume that $\fc = ()$. In this case, as explained in
\cite[Sec.~4] {AF90}, $Z(R)\cong K'^\times$ where $K'$ is the extension of $k$
of degree $e\ell^\alpha $ (hence the extension of $K$ of degree $\ell^\alpha$)
and $C_{\GL_k(V)}(R)$ is isomorphic to $\GL_m(K'^\times)$. It follows that
$Z(R)$ is a Sylow subgroup of $C_{\GL_k(V)}(R)$ if and only if $m=1$. This
proves (a).
Finally, (c) is a consequence of the preceding statements.
\end{proof}

From now on, we write $R_{\alpha,\ga,\fc}:=R_{1,\alpha,\ga,\fc}$. We next
identify the outer automorphism groups of $\GL_k(V)$-radical and centric
subgroups (see Section 4 of \cite{AF90}):

\begin{prop}   \label{p:outer}
 Let $R=\prod_{(1,\alpha,\ga,\fc)} R_{\alpha,\ga,\fc}^{f(\alpha,\ga,\fc)}
 \leq S$ be an $\cF$-centric and $\GL_k(V)$-radical subgroup with
 corresponding decomposition
 \[V=\bigoplus_{(1,\alpha,\ga,\fc)} V_{\alpha,\ga,\fc}^{f(\alpha,\ga,\fc)}\]
 from Proposition~\ref{p:AFradical}. We have:
 \begin{enumerate}
  \item[\rm(a)]
   \[\Out_\cF(R)\cong N_{\GL_k(V)}(R)/R C_{\GL_k(V)}(R)\cong\prod_{\alpha, \ga,\fc} N_{\alpha,\ga,\fc} \wr \fS_{f(\alpha,\ga,\fc)} \]
   where
   $N_{\alpha,\ga,\fc}:=N_{\GL_k(V_{\alpha,\ga,\fc})}(R_{\alpha,\ga,\fc})/R_{\alpha,\ga,\fc}C_{\GL_k(V_{\alpha,\ga,\fc})}(R_{\alpha,\ga,\fc})$.
  \item[\rm(b)] Suppose that $R = R_{\alpha,\ga,\fc}$. Then
   \[ N_{\GL_k(V)}(R)/R C_{\GL_k(V)}(R)\cong
     (\Sp_{2\ga}(\ell)\rtimes C_{e \ell^\alpha})
     \times \GL_{c_1}(\ell) \times \ldots \times \GL_{c_t}(\ell),\]
   where we set $\Sp_{2\ga}(\ell):= 1 $ if $\ga=0$.
 \end{enumerate} 
\end{prop}

\begin{lem}   \label{l:alpha}
 Let $R=\prod_{(1,\alpha,\ga,\fc)} R_{\alpha,\ga,\fc}^{f(\alpha,\ga,\fc)}
 \leq S$ be $\cF$-centric. If $R$ is $\cF$-weight contributing, then
 $f(\alpha,\ga,\fc) = 0$ for all $\alpha\geq1$.
\end{lem}

\begin{proof}
Suppose that $R$ is $\cF$-weight contributing and let
$R_i= R_{1,\alpha,\ga,\fc}$ be one of its factors. By
Proposition~\ref{p:outer}, $\Out_\cF (R_i)$ contains a normal subgroup
isomorphic to $\Sp_{2\ga}(\ell)\rtimes C_{e\ell^\alpha}$. Hence
$\Sp_{2\ga}(\ell)\rtimes C_{e\ell^\alpha}$ has an irreducible character say
$\chi$ of $\ell$-defect zero and if $\chi_0$ is an irreducible character of
$\Sp_{2\ga}(\ell)$ lying below $\chi$, then $\chi_0$ is also of $\ell$-defect
zero. Now the Steinberg character of $\Sp_{2\ga}(\ell)$ is the unique character
of $\ell$-defect zero (if $\ga=0 $, then we take the
trivial character as Steinberg character). It follows that $\chi_0$ is
$\Sp_{2\ga}(\ell) \rtimes C_{e\ell^\alpha}$-stable and hence since
$C_{e\ell^\alpha}$ is cyclic, $\chi$ is an extension of $\chi_0$ and in
particular, $\chi$ and $\chi_0$ have the same degree. Thus, $\alpha =0$. 
\end{proof} 

\subsubsection{ Descent to $r>1 $.}  \label{subsub:descent}
Recall that for a saturated fusion system $\cF$ on a finite $\ell$-group $S$,
there is a group $\Gamma_{\ell'}(\cF)$ whose lattice of subgroups is in
one-to-one correspondence with the lattice of saturated fusion subsystems of
$\cF$ of index coprime to $\ell$ (see \cite[Thm.~I.7.7]{AKO11}) and let
$\theta:\Mor(\cF^c)\to\Gamma_{\ell'} (\cF)$ be the canonical map as defined
before \cite[Thm~I.7.7]{AKO11}. We need the following fact:

\begin{lem}   \label{l:weaklyclosed}
 Let $\cF$ be a saturated fusion system on a finite $p$-group $S$ and suppose
 $T \le S$ is abelian and weakly $\cF$-closed. If $P,Q\le T$ and
 $\vhi\in \Hom_\cF(P,Q)$ then there exists $\bar\vhi\in\Aut_\cF(T)$ with
 $\bar\vhi|_P = \vhi$.
\end{lem}

\begin{proof}
Choose $R \le S$ and $\psi \in \Hom_\cF(Q,R)$ with $R$ fully $\cF$-normalised.
Notice that $R \le T$ since $T$ is weakly $\cF$-closed. Since $\cF$ is
saturated, there exist morphisms $\beta_1 \in \Hom_\cF(N_S(P),S)$ and
$\beta_2\in\Hom_\cF(N_S(Q),S)$ extending $\psi \circ \vhi$ and $\psi$
respectively. Since $T$ is abelian $T\le N_S(P),N_S(Q)$, so $\beta_1$ and
$\beta_2$ both restrict to $\cF$-automorphisms of $T$ (where we also use the
fact that $T$ is weakly $\cF$-closed.) Therefore
$\bar{\vhi}:=(\beta_2|_T)^{-1} \circ \beta_1|_T \in \Aut_\cF(T)$ extends
$\vhi$, as needed. 
\end{proof}

\begin{prop}   \label{prop:gamma}
 Suppose that $\ell $ is an odd prime. Then
 $$\Gamma_{\ell'} (\cF)\cong\begin{cases} \ZZ/e\ZZ& \text{if $n\geq\ell$,}\\
    \ZZ/e\ZZ \rtimes \fS_n& \text{if $n<\ell$.}\end{cases}$$
\end{prop}

\begin{proof}
This is Theorem 5.10 of \cite{Ru07} (see also \cite{OR}). 
\end{proof} 

We will use a precise description of $\theta$ given in \cite{OR}. 

\begin{lem}   \label{l:XinT}
 Let $A= \{ g\in S \mid o(g)=\ell,\,\ker(g-\Id_V)=0 \}$. Then $A\subseteq D$.
\end{lem}
 
\begin{proof}
We follow an argument in the proof of \cite[Ex.~6.2]{OV07}. Note that the set
$A$ defined here is stable under $\cF$-conjugation and therefore contained in
the set denoted $X$ in \cite[Lemma~3.3]{OR}. Let
$g = (\la_1,\ldots,\la_n)\sigma\in A$ with $(\la_1,\ldots,\la_n)\in D$
and $\sigma\in E$. By assumption $g^\ell=1$. Hence, $\sigma^\ell=1$ and
$g (\sigma g\sigma^{-1})\cdots (\sigma^{\ell -1} g\sigma^{-(\ell-1)}) = 1$.
The first condition implies that every non-trivial cycle of $\sigma$ is an
$\ell$-cycle and the second condition implies that if $(i_1,\ldots,i_\ell)$
is an $\ell$-cycle in $\sigma $, then the product
$\la_{i_1}\cdots\la_{i_\ell} = 1$. Letting $u$ be
the number of $\langle\sigma\rangle$-orbits of size $\ell$ in $\{1,\ldots,n\}$,
it follows that the characteristic polynomial of $g$ viewed as a $K$-linear
transformation of $V$ is divisible by $(x^\ell-1)^u$.
Thus, if $u\ne0 $, then $1$ is an eigenvalue of $g$ in its action on $V$.
\end{proof}

Let $\pi:\Aut_\cF(D)\to\ZZ/e\ZZ$ be the homomorphism defined in (\ref{e:pi})
and let $Z$ be the unique subgroup of order $\ell$ of $Z(\GL_K(V))$.

\begin{prop}   \label{prop:olru}
 There exists a surjective homomorphism $\Pi:\Gamma_{\ell'}(\cF)\to\ZZ/e\ZZ$
 with the following property: for any pair of $\cF$-centric subgroups
 $P, Q\leq S$ and any morphism $\alpha\in\Hom_\cF(P,Q)$, there exists
 $\beta\in\Aut_\cF(D)$ such that $\beta|_Z =\alpha|_Z$ and for each such
 $\beta$, $\Pi(\theta(\alpha))= \pi(\beta)$. If $n\geq \ell$, then $\Pi$ is an
 isomorphism.
\end{prop}

\begin{proof}
For this we note by Lemma~\ref{l:XinT} that $Z^\cF\leq D$. Further, we
check that $H_0\fS_n$ is the normal closure of $\Aut_{C_\cF(Z)}(D)$ in
$\Aut_\cF(D)$. The first assertion follows by \cite[Prop.~3.3(a)] {OR}.
The second is immediate from Proposition~\ref{prop:gamma}.
\end{proof} 

Let $\cF':=\cF_S(\GL_K(V))$ be the fusion system of $\GL_K(V)$ on $S$.
 
\begin{prop}   \label{p:descent}
 Let $P,Q\leq S$ be $\cF$-centric subgroups. Then $\Pi\circ\theta(\vhi) = 1$
 for all $\vhi\in\Hom_{\cF'}(P,Q)$. 
\end{prop} 
 
\begin{proof}
Let $P, Q\leq S $ be $\cF$-centric and let $\vhi\in\Hom_{\cF'} (P, Q)$. Since
$D$ is also weakly $\cF'$-closed (see \cite[Prop.~5.13(b)]{BM02}),
$\vhi|_{\langle P\cap X\rangle}$ extends to an $\cF'$-automorphism $\bar\vhi$
of $D$ by Lemma~\ref{l:weaklyclosed}. But since $\Aut_{\cF'}(D) =\fS_n$,
$\pi(\bar\vhi) =1$.
\end{proof}

For $r|e$, we denote by $\cF^{(r)}$ the $\ell'$-index fusion subsystem of $\cF$
corresponding to the subgroup $\Pi^{-1}(r\ZZ/ e\ZZ)$ of $\Gamma_{\ell'} (\cF)$
as described in \cite[Thm.~I.7.7]{AKO11}.

\begin{prop}   \label{prop:ruiz2}
 Let $r|e$ and let $q'$ be a prime power such that $q'\equiv 1\pmod\ell$ and 
 $\nu_\ell(q^e-1)=\nu_\ell(q'-1)$. Then $\cF^{(r)}$ is isomorphic to
 $\cF(X(e,r,n)(q'))$. 
\end{prop}

\begin{proof}
By Theorem~\ref{thm:homfixedpts}(b), $\coX(e,r,n)(q')\simeq\coX(e,r,n)(q^e)$
have isomorphic associated fusion systems. Now the result follows from
\cite[Thm.~6.3]{Ru07}. 
\end{proof}

\begin{prop}   \label{p:centradbasic}
 Let $r|e$ and let $P\leq S$.
 \begin{enumerate}
  \item[(a)] $P$ is $\GL_K(V)$-radical if and only if $P$ is $\GL_k(V)$-radical.
   If $Q\leq S$ is $\GL_k(V)$-radical, then $P$ and $Q$ are $\cF$-conjugate if
   and only if $P$ and $Q$ are $\cF'$-conjugate.
  \item[(b)] The $\cF$-classes of $\cF$-centric, $\GL_k(V)$-radical subgroups
   of $S$ coincide with the $\cF'$-classes of $\cF'$-centric, $\GL_k(V)$-radical
   subgroups of $S$.
 \item [(c)] $P$ is $\cF$-centric if and only if $P$ is $\cF^{(r)}$-centric
   and $P$ is $\cF$-centric and $\cF$-radical if and only if $P$ is
   $\cF^{(r)}$-centric and $\cF^{(r)}$-radical. 
  \item [(d)] The $\cF$-classes of $\cF$-centric, $\GL_k(V)$-radical subgroups
   of $S$ coincide with the $\cF^{(r)}$-classes of $\cF^{(r)}$-centric,
   $\GL_k(V)$-radical subgroups of $S$. 
 \end{enumerate}
\end{prop}
 
\begin{proof}
(a) Applying Proposition~\ref{p:AFradical}(b) with $k$ replaced by $K$ yields
the same description of the $\GL_K(V)$-conjugacy classes of radical subgroups
of $\GL_K(V)$ as that of the $\GL_k(V)$-conjugacy classes of radical subgroups
of $\GL_k(V)$. Since every $K$-decomposition of the type of
Proposition~\ref{p:AFradical} is also a $k$-decomposition, every
$\GL_K(V)$-conjugacy class of $\GL_K(V)$-radical subgroups is contained in a
unique $\GL_k(V)$-conjugacy class of $\GL_k(V)$-radical subgroups and every
$\GL_k(V)$-conjugacy class of $\GL_k(V)$-radical subgroups contains a unique
$\GL_K(V)$-conjugacy class of $\GL_K(V)$-radical subgroups. Now (a) follows
since $S\leq \GL_K(V)$.   \par

(b) It suffices to prove that the bijection between radical subgroups given in
(a) preserves the centric and radical property. But this is immediate from
Proposition~\ref{p:AFradical}. 
 
(c) For the first assertion see \cite[Lemma I.7.6]{AKO11}. Now suppose that
$P$ is $\cF$-centric. Since $\Gamma_{\ell'}(\cF)$ is cyclic, by
\cite[Thm.~I.7.7]{AKO11}, $\Aut_{\cF^{(r)}}(P)$ is a normal subgroup of
$\Aut_\cF (P)$ of $\ell'$-index, and consequently $\Out_{\cF^{(r)}}(P)$ is a
normal subgroup of $\Out_\cF (P)$ of $\ell'$-index. This proves the second
assertion. 

(d) By (c), it suffices to prove that any two $\cF$-conjugate $\cF$-centric
$\GL_k(V)$-radical subgroups $P,Q\leq S$ are $\cF^{(r)}$-conjugate. By part
(a), $P$ and $Q$ are $\cF'$-conjugate. Let $\alpha: P \to Q$ be an isomorphism
in $\cF'$. By Lemma~\ref{p:descent}, $\Pi\circ\theta(\alpha) =1$ hence by
\cite [Thm.~I.7.7]{AKO11}, $\alpha \in \cF^{(r)}$, proving (d).
\end{proof}

\begin{lem}  \label{l:weightdescent}
 For any $r|e$ a subgroup $R\leq S$ is $\cF$-weight contributing if and only if
 it is $\cF^{(r)}$-weight contributing.
\end{lem}

\begin{proof}
As described above, if $R$ is $\cF$-weight contributing, then $R$ is both
$\cF$-centric and $\cF$-radical and similarly for $\cF^{(r)}$. By
Proposition~\ref{p:centradbasic}, the set of $\cF^{(r)}$-centric,
$\cF^{(r)}$-radical subgroups of $S$ coincides with the set of $\cF$-centric,
$\cF$-radical subgroups of $S$. Let $ R\leq S $ be $\cF$-centric and
$\cF$-radical. As explained in the proof of Proposition~\ref{p:centradbasic}(c),
$\Out_{\cF^{(r)}}(R)$ is a normal subgroup of $\Out_\cF(R)$ of $\ell'$-index.
Thus an irreducible character $\chi$ of $\Out_\cF(R)$ is of $\ell$-defect zero
if and only if the irreducible characters of $\Out_{\cF^{(r)}}(R)$ covered by
$\chi$ are of $\ell$-defect zero.
\end{proof} 

\subsubsection{Counting weights.} By the above results, in order to count
$\cF$- and $\cF^{(r)}$-weights we need only consider those $\cF$-centric,
$\GL_k(V)$-radical subgroups of $S$ whose basic components are all of type
$R_{1,0,\ga,\fc}$. We will concentrate on such subgroups from now on.

We denote
$$\fF:=\{f: \NN \times \cC \to \NN \mid
  \sum_{ (\ga, \fc)} \ell^{\ga + |\fc |} f(\ga, \fc) = n\}.$$
For $f \in \fF$ choose a direct sum decomposition into $k$-subspaces
\begin{equation}
  V'= \bigoplus_{(\ga,\fc)\in\NN\times\cC}\bigoplus_{1\leq i\leq f(\ga, \fc)}V'_{ \ga,\fc,i}
\end{equation}
such that 
\[\dim_k (V'_{ \ga, \fc, i}) = \ell^{ \ga + |\fc |}\qquad\text{for }
  (\ga, \fc) \in \NN \times \cC,\ \ 1\leq i \leq f(\ga, \fc), \] 
and such that for each $\ga,\fc,i$, there is a subset $Y$ of $\{1,\ldots,n \}$
and an index $(i', j')$ of the direct sum decomposition of $V'$ underlying $S$
such that denoting by $Y'$ the subset of $\{1\leq i \leq n \}$ such that
$V_{i',j'}' = \sum_{y'\in Y'} V'_{y'}$ we have that
$V'_{\ga,\fc,i}= \sum_{y \in Y} V'_y$ and $Y\subseteq Y'$. Such a decomposition
is always possible given the arithmetic constraints on $f$.
Consider the induced decomposition 
\[ V = \bigoplus_{(\ga,\fc)}\bigoplus_{1\leq i\leq f(\ga, \fc)}V_{ \ga,\fc,i} \]
where $V_{ \ga, \fc, i} = K \otimes_k V'_{ \ga, \fc, i}$.
Set 
\[ R_f = \prod_{(\ga,\fc)}\prod_{1\leq i\leq f(\ga,\fc)} R_{ \ga, \fc}^{(i)}\] 
where for each $i$, $R_{\ga,\fc}^{(i)}\leq\GL_K (V_{\ga,\fc,i})\leq \GL_k(V_{\ga,\fc,i})$ is a basic subgroup of type $R_{0,1,\ga,\fc}$.

By Propositions~\ref{p:AFradical} and~\ref{p:centradbasic}, $f$ determines the
groups $R_f$ up to $\cF$-conjugacy and the $\cF^{(r)}$-class of $R_f$ determines
$f$. We may and will assume that the groups $R_f$ are subgroups of $S$.

For $i \geq 0$, let $\fc_i$ be the sequence $(1,\ldots,1)$ of length $i$ and
define $f_S,f_D:\NN\times\cC\to\NN$ by
$$f_S(\ga,\fc)= \begin{cases} a_i& \text{for $(\ga,\fc)=(0,\fc_i)$,}\\
                 0& \text{otherwise,}\end{cases}\qquad
f_D(\ga,\fc) = \begin{cases} n& \text{if $(\ga,\fc)=(0,())$,}\\
                 0& \text{otherwise.}\end{cases}$$
Then $f_S,f_D\in\fF$ and $S = R_{f_S}$ and $D = R_{f_D}$.

\begin{prop}   \label{p:out}
 Let $R=R_f \leq S$ be as above. We have:
 \begin{enumerate}
  \item[\rm(a)]
   \[ \Out_\cF(R) = \prod_{(\ga,\fc)} N_{\ga,\fc}\ \wr \fS_{f(\ga,\fc)},\] 
   where 
   $N_{\ga,\fc} = (\Sp_{2\ga}(\fc)\rtimes C_{\ga,\fc})\times\GL_{c_1}(\ell) \times\cdots\times\GL_{c_t}(\ell)$
   and $C_{\ga, \fc}$ is cyclic of order~$e$.
  \item[\rm(b)] For $(\ga,\fc)$ with $f(\ga,\fc)\ne 0$, let $\alpha_{\ga,\fc} \in\Aut_\cF(R)$ be a lift of a generator of $C_{\ga,\fc}$. Then
   $\Pi\circ\theta(\langle \alpha_{\ga, \fc}\rangle) = \ZZ/e\ZZ$.
  \item[\rm(c)] We have $\Pi\circ\theta(\Aut_{\cF^{(r)}}(R)) = r\ZZ/e\ZZ$, and $\ker(\Pi\circ\theta)$
   contains all $\fS_{f(\ga,\fc)}\Inn(R)$, $(\ga,\fc)\in\NN\times\cC$. 
 \end{enumerate}
\end{prop}

\begin{proof}
The assertion in (a) is proved in Proposition~\ref{p:outer}.
For (b), we may assume that $R= R_f$, where $f$ is the characteristic function
of some $(\ga,\fc)\in\NN\times\cC$ and $n=\ell^{\ga +|\fc|}$. So $R$ is of type
$R_{1,0,\ga,\fc}$.
Let us first consider the case that $\fc =()$. Set $N:=N_{\GL_k(V)}(R)$,
$C=C_{\GL_k(V)}(R)$ and let $N^0:= C_N(Z(R))$. By \cite[Sec.~4]{AF90},
$N^0 = L C R$, where $L\cong\Sp_{2\ga}(\ell)$, $[C,L] =1$ and $L\cap RC =1$.
Also, $N =N^0 \rtimes\langle\sigma\rangle$ with $o(\sigma)=e$.
 
Now consider the conjugation action of $\langle\sigma\rangle$ on $Z(R)$. By
definition of $N^0$, this action is faithful. By definition of $R$, $Z(R)$ is
the Sylow $\ell$-subgroup of $Z(\GL_K (V))\cong K^\times$. Note that $Z$ is a
subgroup of order $\ell$ of $Z(R)$.
Let $Z(R) =\langle g \rangle $ and suppose that $\sigma (g) = g^i$. Since
$\sigma (g)$ has the same eigenvalues as $g$, we see that $\sigma(g)=g^{q^{j}}$
for some $0 \leq j \leq e-1$. Since the action of $\sigma$ on $Z(R)$ is
faithful, by replacing $\sigma$ by a suitable power we may assume
that $\sigma (g) = g^q$. Thus, letting $c_\sigma $ denote the image of $\sigma$
in $\Aut_\cF(R)$, we see that restriction of $c_\sigma $ to $Z(R)$ coincides
with the restriction of $\Phi_{\{1,\ldots,n\}}$ to $Z(R)$, where
$\Phi_{\{1,\ldots,n\}}\in\Aut_\cF(D)$ is the automorphism of $D$ acting as
$\Phi$ on all components. Thus by Proposition~\ref{prop:olru},
\[\Pi\circ\theta (c_{\sigma}) =\pi(\Phi_{\{1,\ldots,n\}})\equiv n\pmod e .\]
Since $n$ is a power of $\ell$, $n$ and $e$ are relatively prime and
surjectivity follows. The case of general $\fc$ follows from the case above by
\cite[Eq.~(4.1)]{AF90}.

By \cite[Thm.~I.7.7]{AKO11},
$\cF^{(r)} =\langle (\Pi\circ\theta)^{-1} (r\ZZ/e\ZZ)\rangle$
and $\Aut_{\cF^{(r)}} (R) =\Aut_\cF(R)\cap(\Pi\circ\theta)^{-1}(r\ZZ/e\ZZ)$.
Hence the first statement of (c) follows from~(b).
By construction of the isomorphism in (a), each $\fS_{\ga,\fc}$ is the image
of some subgroup of $\fS_n\cap N_{\GL_k(V)}(R)$ under the canonical map
$N_{\GL_k(V)}(R)\to\Out_\cF(R)$ and conjugation by an element of $\fS_n$ goes
to the identity under $\Pi\circ\theta$. This proves the second assertion of (c).
\end{proof} 

For $\fc\in\cC$, denote by $\cA(\fc)$ the set of irreducible characters of
$\GL_{c_1}(\ell)\times\cdots\times\GL_{c_t}(\ell)$ of the form
$\chi_1\cdot\ldots\cdot\chi_t$, where each $\chi_j$ is a Steinberg character of
$\GL_{c_j}(\ell)$. Note that $|\cA(\fc)| = (\ell-1)^t$. 

\begin{prop}   \label{p:all}
 \begin{enumerate}
  \item[\rm(a)] The set of $\cF$-weights (up to conjugation) is in bijection
   with the set $\cW $ of assignments 
   \[ w: \NN\times\ZZ/e\ZZ\times\bigcup_{\fc\in\cC}\cA(\fc)\to \{ \ell\text{-cores} \} \]
    such that
   \[ \sum_{\fc}\sum_{(\ga,x,\vhi)\in\NN \times\ZZ/e\ZZ \times\cA(\fc)}\ell^{\ga +|c|} |w(\ga, x, \vhi)|=n. \] 
  \item[\rm(b)] Let $\ZZ/r\ZZ$ act on $\cW$ via
   $y.w(\ga,x,\vhi):=w(\ga,y +x,\vhi)$ for $y\in\ZZ/r\ZZ$, $w\in\cW$ and
   $(\ga,x,\vhi)\in \NN\times \ZZ/e\ZZ \times\cA$. Then the $\cF^{(r)}$-weights
   (up to conjugation) are indexed by the $\ZZ/ r\ZZ$-orbits of $\cW$ with each
   orbit contributing as many weights as the order of the stabiliser of a point
   of the orbit.
 \end{enumerate}
\end{prop} 

\begin{proof}
By Lemma~\ref{l:alpha} and Lemma~\ref{l:weightdescent} any $\cF$- (respectively
$\cF^{(r)}$)-weight contributing subgroup is conjugate to an $R_f$ for some
$f\in\fF$. By Propositions~\ref{p:AFradical} and~\ref{p:centradbasic} distinct
choices of $f$ give rise to distinct $\cF$- (respectively $\cF^{(r)}$)-conjugacy
classes of subgroups of $S$. Thus, we are reduced to the case when $R=R_f$.
Then, (a) is a consequence of Proposition~\ref{p:out}(a) and
Proposition~\ref{p:wreath} applied with $G = M =\Out_\cF (R)$. 

By Proposition~\ref{p:out}(b),(c), $\Out_\cF(R)/\Out_{\cF^{(e)}}(R)\cong
\ZZ/e\ZZ$ and for any $(\ga,\fc)$ with $f(\ga, \fc)\ne 0$, the image of 
$\alpha_{\ga,\fc,}$ is a generator of $\Out_\cF (R)/\Out_{\cF^{(e)}}(R) \cong
\ZZ/e\ZZ$. This means that the action of
$\Irr(\Out_\cF (R)/\Out_{\cF^{(e)}} (R))$ on $\Irr^0(N_{\ga,\fc})$ given by 
\[ \la \cdot \vhi = \la_0 \cdot \vhi \quad\text{for }
  \la\in\Irr(\Out_\cF(R)/ \Out_{\cF^{(e)}}(R)),\ \vhi\in\Irr^0(N_{\ga,\fc}), \]
where $\la_0$ is the $(\ga,\fc,i)$-component of the restriction of $\la$ to
$N_{\ga,\fc}$ for some (and hence all) $1\le i\le f_{\ga,\fc}$, is a
regular action.
Thus the $\ZZ/e\ZZ$-set $\Irr^0(N_{\ga,\fc})$ can be identified with the
$\ZZ/e\ZZ$-set $\ZZ/e\ZZ \times \cA(\fc)$ (acting through the left regular
action on the first component). 

Now (b) follows by Proposition~\ref{p:out} and by Proposition~\ref{p:wreath}
applied with $G = \Out_\cF(R)$ and $M = \Out_{\cF^{(r)}}(R)$.
\end{proof}

We now give an equivariant form of (1A) of \cite{AF90}. 

\begin{lem}   \label{lem:wisol}
 Let $e$, $\ell$, $r|e$ and $n$ be as above. For each $d\geq 0$, let
 \[I_d =\ZZ/e\ZZ \times \{ (d, j) \mid 0\leq j < \ell^ d \}. \]
 View $I_d$ as a $\ZZ/r\ZZ$-set via the left regular action on the first
 component and let $I = \bigcup_d I_d$. Let $\cL$ be the set of
 $e$-tuples of partitions of $n$
 viewed as a $\ZZ/r\ZZ$-set via the left regular action on indices and
 $$\cW:=\Big\{w:I\to\{\ell\text{-cores}\}\mid
   \sum_{i,j,d}\ell^d w(i,d,j) = n\Big\}$$
 viewed as a $\ZZ/r\ZZ $-set via $i'. w (i, d, j) = w(i+i', d, j)$.

 Then $\cW$ and $\cL$ are isomorphic $\ZZ/r\ZZ $-sets.
\end{lem} 

This assertion follows by inspection of the proof of (1A) of \cite{AF90}, which
we recall here for the convenience of the reader:

\begin{proof}
Let $(\la_1,\la_2,\ldots,\la_e)\in\cL$.
%A corresponding element $w \in \cW$ is constructed as follows.
Fix $i$ with $1 \le i \le e$. For each
$d \ge 0$ and $0 \le j < \ell^d$ we define a family of $\ell$-cores
$\kappa_{i,j}^d$ inductively as follows. Set $\la_{i,0}^0:=\la_i$. Then if
$d\ge0$ and $\la_{i,j}^d$ is defined for $0\le j < \ell^d$, we let
$\kappa_{i,j}^d$ be the $\ell$-core of $\la_{i,j}^d$ and
$(\la_{i,\ell j+0}^{d+1},\ldots,\la_{i,\ell j+\ell-1}^{d+1})$ be its
$\ell$-quotient. By construction $\displaystyle\sum_{d \ge 0} \sum_{0 \le j < \ell^d} \ell^d|\kappa_{i,j}^d|=|\la_i|$ and each $\la_i$ is uniquely determined
by the $\kappa_{i,j}^d$.
Now for each $1 \le i \le e$ define $w(i,d,j):=\kappa_{i,j}^d$ and observe that
$w\in\cW$ since $\sum_{i=1}^e |\la_i| =n$ by assumption. Moreover it is clear by
construction that the resulting map $\cL \rightarrow \cW$ is both a bijection
and $\ZZ/r\ZZ$-equivariant.
\end{proof}

\begin{proof}[Completion of the proof of Theorem \ref{thm:nr weights}]
By Theorem~\ref{thm:homfixedpts} and Proposition~\ref {prop:ruiz2}, it suffices
to show that $\bw(\cF^{(r)}) = |\Irr(G(e,r, n)|$. For each $d >0$, let
$(\NN \times\cC)_d$ be the subset of $\NN\times\cC$ consisting of pairs
$(\ga,\fc)$ such that $\ga +|\fc| = d$.
By the proof of (4C) of \cite{AF90}, we have that
\[ \sum_{(\ga,\fc)\in(\NN\times\cC)_d} |\cA (\fc) |= \ell^d. \]
Let
\[ I_d =
   \ZZ/e\ZZ\times\bigcup_{(\ga,\fc) \in (\NN\times\cC)_d}\cA(\fc).\]
Then $\bigcup_{\fc\in\cC}(\NN \times\ZZ/e\ZZ\times\cA(\fc))$ may be
naturally identified with $\bigcup_{d} I_d$. Thus by Lemma~\ref{lem:wisol}
and by Proposition~\ref{p:all}, the $\cF^{(r)}$-weights are indexed by the
$\ZZ/r\ZZ$-orbits of $\cL$ with each orbit contributing as many
weights as the order of the stabiliser of a point of the orbit. Here
$\cL$ is as defined in Lemma~\ref{lem:wisol}. Since the irreducible
characters of $G(e,r,n)$ are also indexed by the $\ZZ/r\ZZ $-orbits of
$\cL$ in the same way, the proof is complete.
\end{proof}

%%%%%%%%%%%%%%%%%%%%%%%%%%%%%%%%%%%%
\subsection{Centralisers}   \label{subsec:GGr centr}
Let $e,r,n$ be positive integers with $r\mid e \mid (\ell-1)$. Let $q$ and $q'$
be prime powers such that such that $q'\equiv 1\pmod\ell$, $q$ has order~$e$
modulo $\ell$ and $\nu_\ell(q^e-1)=\nu_\ell(q'-1)$. We describe centralisers of
$\ell$-elements in $X(e,r,n)(q')$ as defined in Section~\ref{sec:centr} via the
identification of $X(e,1, n)(q')$ with $X(1,1,en)(q)$ from
Theorem~\ref{thm:homfixedpts} and Example~\ref{exmp:LS}.
\par
Let us first record some relevant information for the general linear groups. For
$i\ge 0$ we let $\fF_i\subset\FF_q[x]$ denote the collection of all irreducible
polynomials of degree $e\ell^i$ all of whose roots have $\ell$-power order. The
following is
easy and well-known:

\begin{lem}   \label{lem:centrGL}
 Let $G=\GL_n(q)$. Then, via their characteristic polynomials, $G$-conjugacy
 classes of $\ell$-elements of $G$ are in bijection with polynomials
 $$(x-1)^{a_1}\prod_{i\ge0}\prod_{f\in \fF_i} f^{a_f}\ \in\FF_q[x]$$
 of degree~$n$; thus $a_1,a_f\ge0$ satisfy
 $a_1+e\sum_{i\ge0}\ell^i\sum_{f\in \fF_i} a_f = n$.   \par
 The centraliser of an $\ell$-element $s\in G$ with this characteristic
 polynomial is
 $$C_G(s)\cong
     \GL_{a_1}(q)\times\prod_{i\ge0}\prod_{f\in \fF_i}\GL_{a_f}(q^{e\ell^i}).$$
\end{lem}

Let $\zeta\in\ZZ_\ell^\times$ be the root of unity with $q\equiv\zeta\pmod\ell$.
By Theorem~\ref{thm:homfixedpts} and Example~\ref{exmp:LS}, $X(e,1,n) \simeq
X(1,1,en)^{h\Gamma}$ where $\Gamma$ is the subgroup of $\Out(X(1,1,en))$
generated by the class of $\psi^\zeta$. There is a homotopy commutative diagram
$$\xymatrix{ X(1,1,en) (q) \ar[d]^{\wr}_{g}\ar[r]^{\ \ \ \iota} & X(1,1,en) \\ 
  X(e, 1,n) (q') \ar[r]^{\ \ \ \iota_1} & X(e,1,n)\ar[u]_{\gamma} }$$
where $\iota$ and $\iota_1$ fit into homotopy pullback squares as in~(*) in
Section~\ref{subsec:def fus}.   \par
Let $S\le\GL_{en}(q)$ be a fixed Sylow $\ell$-subgroup. By \cite[Thm.~6.3]{Ru07}
and its proof, for each $r|e$, there is a map $\beta_r: BS\to X(e,r,n)(q')$ and
a homotopy monomorphism $X(e,r,n)(q')\to X(e,1,n)(q')$ whose composition
$BS\to X(e,r,n)(q')\to X(e,1,n)(q')$ factorises as the inclusion
$BS\to B\GL_{en}(q)^\wedge_\ell \simeq X(1,1,en)(q)$ followed by $g$ and such
that $\cF_{S,\beta_r}(X(e,r,n)(q'))$ is the saturated subsystem $\cF^{(r)}$ of
$\cF_{\GL_{en}(q)}(S)$ as in Proposition~\ref{prop:ruiz2}. For each $s \in S$ set
$$C_{X(e,r,n)}(s) := C_{X(e,r,n)}(s,\iota_r\circ\beta_r), \ \ 
  C_{X(e,r,n)(q') } (s) := C_{X(e,r,n) (q') } ( s, \beta_r) $$
where again $\iota_r: X(e,r,n) (q') \to X(e,r,n)$ fits into a homotopy pullback
square as in~(*) in Section~\ref{subsec:def fus}.

\begin{thm}   \label{thm:centr GGR}
 With the notation above, for any $\ell$-element $s\in S$ with characteristic
 polynomial
 $$(x-1)^{a_1}\prod_{i\ge 0}\prod_{f\in \fF_i} f^{a_f} \qquad\text{with }
   a_1+e\sum_{i\ge0}\ell^i\sum_{f\in \fF_i} a_f = en$$
 we have
 $$C_{X(e,r,n)}(s)
  \simeq X(e,r,a_1/e)\times\prod_{i\ge0}\prod_{f\in\fF_i} X(1,1,a_f)^{\ell^i}.$$
 If moreover $s$ is fully centralised in the fusion system $\cF^{(r)}$, then
 $$C_{X(e,r,n)(q')}(s)\simeq X(e,r,a_1/e)(q')\times
   \prod_{i\ge 0}\prod_{f\in \fF_i} X(1,1,a_f)(q'^{\ell^i}).$$
 Here, the products are understood to run over those $f$ with $a_f\ne0$.
\end{thm}

\begin{proof}
For the first assertion, by Theorem~\ref{thm:class} it suffices to show that the
Weyl group of $C_X(s)$ is isomorphic to the Weyl group of the right hand side.
For this, we first identify the discrete tori of $X(e,r,n)$ and $X(e,1,n)$ such
that the Weyl group $W^{(r)}:=G(e,r,n)$ of the former acts as a subgroup of the
Weyl group $W=G(e,1,n)$ of the latter. It follows from the proof of
Proposition~\ref{prop:stab} that
$C_{W^{(r)}}(s)= G(e,r,u_0)\times\prod_j\fS_{u_j}$ if and only if
$C_W(s)= G(e,1,u_0)\times\prod_j\fS_{u_j}$. It hence suffices to discuss the
case $r=1$.

Set $X'= X(1,1,en)\simeq B\GL_{en}(\CC)_\ell^\wedge$. Let $\bT'$ be the subgroup
of diagonal matrices of $\bG:=\GL_{en}(\FF)$,
$W'=N_\bG(\bT')/\bT'\cong \fS_{en}$ and $L'$ the cocharacter lattice of $\bT'$.
Then as explained in Section~\ref{s:lie}, $(W',L')$ is the Weyl group of $X'$
and by Theorem~\ref{thm:homfixedpts}(a) the Weyl group $(W,L)$ of $X$ is
$(W'_\zeta,L'_\zeta)$ (see the discussion above the statement of the
theorem). Further,
$$W'_\zeta = C_{W'}(w_0) = \left(\big\langle(1,\ldots, e)\big\rangle \times
  \cdots \times \big\langle (e(n-1)+1,\ldots, en)\big\rangle\right) \rtimes H,$$
where 
$$w_0:=(1,\ldots,e)\cdots( e(n-1) + 1,\ldots,en)\in\fS_{en}$$
is a product of $n$ disjoint cycles of length $e$, $H\cong\fS_n$ acts by
wreathing the base subgroup and $ L'_\zeta$ is the $1$-eigenspace of
$w_0\zeta$ (see Section~\ref{subsec:sl}). Consequently the subgroup
$\brT$ of $\bT'$ of $\ell$-power order elements consists of diagonal matrices
of the form
$$ \diag( x_1,\ldots,x_{en})\quad\text{with }\ 
 x_{ue +i +1} = x^\zeta_{ue +i}\ \text{ for all }0\leq u \leq n-1,\, 1\leq i \leq e-1. $$
Here we are writing $\brT$ multiplicatively and for $x\in\ZZ/\ell^\infty$ by
$x^\zeta$ we mean the image of $x$ under the action of $\zeta\in\ZZ_\ell^\times
\cong \Aut(\ZZ/\ell^\infty)$. In particular if $x$ has order $\ell^d$, then
$x^\zeta$ is $x^n$, for any $n\in\ZZ$ congruent to $\zeta$ modulo $\ell^d$. 
 
Let $t = \diag(x_1,\ldots,x_{en})\in\brT$. Let $\{1,\ldots,en\} = J_0\cup \ldots\cup J_m$ be the partition induced by the relation $x_i = x_j$, where
$J_0$ consists of the indices $i$ with $x_i =1$ (by abuse of notation we allow
the possibility $J_0 =\emptyset$). Now
$J_u\mapsto J_u^\zeta:=\{x^\zeta\mid x\in J_u\}$ defines a free action of
$\langle\zeta\rangle$ on the set $\{J_u \mid 1\leq u\leq m\}$.
By suitable re-indexing we may assume that $\{J_1,\ldots,J_{m'}\}$ is a set of
orbit representatives for this action, where $m'=m/e$. For each $1\leq v\leq m'$,
let $\tilde J_i = \bigcup_{0\leq b \leq e-1} J_i^{\zeta ^b}$. Then 
$\{1,\ldots,en\} = J_0\cup\tilde J_1\cup\ldots\cup\tilde J_{m'}$ is a partition, 
$$ C_{W'}(t) \cong \fS_{|J_0|} \times \prod_{v=1} ^{m'} \fS_{|J_v|} ^{ e} , $$
and 
$$ C_W(t) = C_{C_{W'} (w_0) } ( t) \cong G(e,1, |J_0|/ e ) \times \prod_{v=1} ^{m'} \fS_{|J_v|} .$$
\par 

Suppose that $t\in \brT$ has characteristic polynomial 
$$(x-1)^{a_1}\prod_{i\ge 0}\prod_{f\in \fF_i}f^{a_f} \in\FF_q[x]\qquad
  \text{with } a_1+e\sum_{i\ge0}\ell^i\sum_{f\in \fF_i} a_f = en.$$
If $\la$ is an eigenvalue of $t$,
then $\la^\zeta = \la^{q^{\ell^i}}$ when $\la$ has degree $\ell^i$ over $\FF_q$.
Thus,
$$ C_W(t) =C_{C_{W'} (w_0)} (t)\cong G(e,1,a_1/e)\times\prod_{i\ge0}\prod_{f\in \fF_i} G(1,1,a_f)^{\ell^i}.$$

Let $s\in S$ and suppose that $t \in \brT$ is the image of $s$ under a discrete
approximation of a factorisation of the restriction of $\beta_1$ to
$B\langle s\rangle$. Then $t$ is also the image of $s$ under a discrete
approximation of $ \gamma \circ \iota_1 \circ \beta_1$ where $\gamma, \iota_1$
are as in the homotopy commutative diagram at the beginning of the section. But
$\gamma\circ\iota_1\circ \beta_1$ is also the composition of $\inc:BS\to X'(q)$
with $ X'(q) \to X'$. By Lemma~\ref{l:factor} applied to $\bG=\GL_{en}(\FF)$, we
have that $t$ and $s$ have the same characteristic polynomial. This proves the
first assertion.

Before moving on to the second assertion, let us record a fact which will be
used in the proof of Lemma~\ref{l:identifyws2}. Let $t\in\brT$ have
characteristic polynomial as above. Since
$\nu_\ell(q^{eb}-1) = \nu_\ell((q')^b-1)$ for all non-negative integers $b$ and
$e$ is the order of~$q$ modulo $\ell$, there
exists $w\in H$ which is a product of cycles of length $\ell^i$ for
each $f$ of degree $e\ell^i$ (counting multiplicities) such that $w.t^{q'} =t$,
with $H\cong\fS_n$ as above. In particular, note that $w\in G(e,e,n)$. 

For the second part, we identify the centraliser via its fusion system.
If $r=1$, by Lemma~\ref{lem:centrGL} and Proposition~\ref{prop:ruiz2} the fusion
system of $C_{\GL_{en}(q)}(s)$ is the fusion system on the right hand side, so
we conclude with Proposition~\ref{prop:centr}(b). For $r>1$ we argue by
induction on the order of $s$. If $s^{\ell^a}=1$ with $a=\nu_\ell(q^e-1)$,
the claim follows by \cite[Prop.~11.2]{BM07}.

Now let $s\in S$ be fully centralised and suppose that $s$ has order $\ell^d$,
$d>a$. Set $u = s^{\ell^{d-1}}$. We may replace $s$ by a suitable
$\cF^{(r)}$-conjugate such that both $s$ and $u$ are fully
$\cF^{(r)}$-centralised and hence by \cite[Prop.~3.8(c)]{BCGLO07} both $s$ and
$u$ are fully $\cF$-centralised. Write $\GL_{en}(q)=\GL_k(V)$ and let
$W_0\subseteq V$ be the $1$-eigenspace of $u$. Since $u$ is semisimple, there
is a decomposition $V= W_0\oplus W_1$ into $\langle u\rangle$-invariant
subspaces. By Lemma~\ref{lem:centrGL},
\[ C_{\GL_k(V) }(u) = \GL_k(W_0)\times C_{\GL_k(W_1)}(u).\]
%  \cong \GL_k(W_0)\times\prod_{i\ge0}\prod_{f\in \fF_i}\GL_{a_f}(q^{e\ell^i}).\]
 
Setting $S_j:= C_S(u)\cap\GL_k(W_j)$ we have $C_S(u) = S_0\times S_1$. Since
$u$ is fully $\cF$-centralised, $S_j$ is a Sylow $\ell$-subgroup of $\GL_k(W_j)$
and $S_1$ is a Sylow $\ell$-subgroup of $C_{\GL_k(W_1)}(u)$. 

Set $\cG_0=\cF_{S_0}(\GL_k(W_0))$ and $\cG_1=\cF_{S_1}(C_{\GL_k (W_1)}(u))$. 
As explained in the proof of \cite[Lemma 6.2]{Ru07}, the decomposition of
$C_{X(q')}(u)$ given in the first part of our statement translates to a
corresponding decomposition for fusion systems
\[ C_{\cF^{(r)} }(u) = \cG_0^{(r) } \times \cG_ 1. \]

Now $s$ preserves the decomposition $V =W_0 \oplus W_1$. Write $s=(s_0,s_1)$
where $s_j \in S_j$. Then the characteristic polynomial of $s$ on $V$ decomposes
as $h_0(x)h_1(x)$ where $h_j(x)$ is the characteristic polynomial of $s_j$ on
$W_j$, $j=0,1$. Moreover, $h_0(x)$ and $h_1(x)$ are relatively prime. In other
words, 
\[ h_0 (x) = (x-1) ^{a_1} \cdot \prod_{i\ge 0}\prod_{f\in \fF_i} f^{b_f}\ 
  \text{ and }\ h_1(x) = \prod_{i\ge 0}\prod_{f\in \fF_i} f^{c_f}, \]
where for each $i\geq 0$ and each $f\in \fF_i $, either $b_f=0$, $c_f=a_f$ or
vice versa.

We have
\[C_{\cF^{(r)}}(s) = C_{C_{\cF^{(r)}}(u)}(s)
  = C_{\cG_0^{(r) }}(s_0)\times C_{\cG_ 1}(s_1). \]
Since $C_{\cG_1}(s_1) = \cF_{C_{S_1}(s_1)}(C_{\GL_k(W_1)}(s_1))$, and since
up to $\ell$-completion $X(1,1,c_f)(q'^{\ell^i})$ is homotopy equivalent to
$B\GL_{ec_f}( q^{\ell^i})$ we obtain from Lemma~\ref{lem:centrGL} that
$C_{\cG_1}(s_1) $ is the fusion system associated to 
\[ \prod_{i\ge 0}\prod_{f\in \fF_i} X(1,1,c_f)(q'^{\ell^i}) .\]
 
By Proposition~\ref{prop:centr}, $C_{\cG_0^{(r)}}(s_0)$ is the fusion system
associated to $C_{X(e,r,a_1/e)(q')}(s_0) $. Since $s_0^{\ell^{d-1}}=u|_{W_0}=1$,
the order of $s_0 $ is strictly less than the order of $s$. Hence, 
\[ C_{X(e,r,a_1/e)(q')}(s_0)\simeq
  X(e,r,b_1) \times \prod_{i\ge 0}\prod_{f\in \fF_i} X(1,1,b_f)(q'^{\ell^i})\]
by induction.
The result follows as by Proposition~\ref{prop:centr}(b), $C_{\cF^{(r)}}(s)$ is
the fusion system associated to $C_{X(q')}(s)$.
\end{proof}

\begin{lem}   \label{l:identifyws2}
 In the notation of Theorem~\ref{thm:centr GGR}, let $s\in S$ have characteristic polynomial
 $(x-1)^{a_1}\prod_{i\ge 0}\prod_{f\in \fF_i} f^{a_f}. $
 Let $(W(s)\phi_s, q')$ be as in Conjecture~\ref{c:spetsandcomp}. Then we may
 choose
 \[ \phi_s=\prod_{i} \prod_{f \in \fF_i} \sigma_{i, f}\, \in\, \fS_n
   \leq G(e,r,n)\]  % = A(e,r,n)\rtimes\fS_n\]
 such that for each $i$ and $f$, $\sigma_{i,f}$ is a product of $a_f$ disjoint
 cycles of length $\ell^i$, and 
 \[ W(s)_{\phi(s)}
 \cong G(e,r,a_1/e)\times \prod_{i\geq 0}\prod_{f\in \fF_i} G(1,1,a_f). \]
 In particular, Conjecture~\ref{c:spetsandcomp} holds in this case.
\end{lem}

\begin{proof}
Let $(W,L)$ be identified with $(C_{W'}(w_0),L_\zeta)$ as in the proof
of Theorem~\ref{thm:centr GGR} and let $t\in\brT$ with $\breve{j}(s) = t$. Then
by the recipe for $\phi_s$ we may take for $\phi_s$ the element $w\in H$ as
described in the proof of Theorem~\ref{thm:centr GGR}. One checks that $w$ has
maximal possible $\zeta$-eigenspace on $L$ amongst the elements of $C_W(t)w$
and further that $C_W(t)$  acts faithfully on this $\zeta$-eigenspace. 
The first assertion follows. The description of $\phi_s$ given above and of
$C_{X(q')}(s)$ in Theorem~\ref{thm:centr GGR} yields
Conjecture~\ref{c:spetsandcomp}.
\end{proof}

%%%%%%%%%%%%%%%%%%%%%%%%%%%%%%%%%%%%
\subsection{Proof of Conjecture~\ref{conj:nr oweights} for $G(e,r,\ell)$ and $q\equiv 1 \pmod\ell $}   \label{sec:GGr ord}
Assume $r|e|(\ell-1)$ and let $X$ be the connected $\ell$-compact group with
Weyl group $(W,L)$ where $W=G(e,r,\ell)$, $e>1$. Let $\ell$ be very good for $W$
and $\tau$ be an automorphism of $X$ of finite order prime to $\ell$. Since
$\ell$ is odd, the element $\phi\in N_{\GL(L)}(W)$
to which $\tau$ corresponds via Theorem \ref{thm:class} is a scalar (see
\cite[Prop.~3.13]{BMM99}). It therefore suffices to consider the case $\tau=1$.
Let $q\equiv 1\pmod\ell$ and recall that by Proposition~\ref{prop:ruiz2} the
fusion system for $X(q)$ may be realised as a fusion system inside a general
linear group. Precisely, we choose $q_0$ to have order $e$ modulo $\ell$ and set
$a:=\nu_\ell(q_0^e-1)$ so that $\cF=\cF(X(q))$ is a subsystem of the
$\ell$-fusion system of $\GL_{e\ell}(q_0)$ on a Sylow $\ell$-subgroup $S$.

Recall the sets $\fF_i\subset\FF_{q_0}[x]$ defined before Lemma~\ref{lem:centrGL}.
An easy counting argument shows that
$$r_i:=|\fF_i|=\begin{cases} (\ell^a-1)/e& \text{if $i=0$},\\
          (\ell^a-\ell^{a-1})/e& \text{otherwise}.\end{cases}$$

We begin by identifying the $\cF$-centric radical subgroups.

\begin{lem}   \label{l:n=lcentrads}
 Let $W=G(e,r,\ell)$. Then the $\cF$-centric radical subgroups $R$ are:
\begin{itemize}
  \item[(1)] $R=S\cong C_{\ell^a}\wr C_\ell$ with
   $\Out_\cF(R)\cong C_{\ell-1} \times C_{e/r}$;
  \item[(2)] $R=D\cong (C_{\ell^a})^\ell$ with $\Out_\cF(R)\cong W$; and
  \item[(3)] $R=E:=Z_2(S)\langle \sigma\rangle \cong C_{\ell^a}* \ell_+^{1+2}$
   (a central product) where $\sigma$ is any element satisfying
   $S=D\langle\sigma\rangle$, and $\Out_\cF(E)\cong\SL_2(\ell)\rtimes C_{e/r}$.
 \end{itemize}
 For $d\ge 0$ and $R$ an $\cF$-centric radical subgroup, $\bw_R(\cF,d)$ is
 given in Table \ref{tab:n=ell}, where ``$-$'' indicates that there are no
 characters of that defect. Moreover
 $$\bm(\cF,a\ell)=\alpha+\beta=\sum_{\chi\in\Irr^{a\ell}(D)/W} z(kI_W(\chi)).$$
\end{lem}

\begin{table}[htb]
\caption{$\bw_R(\cF,d)$ for $R\in\cF^\CR$}   \label{tab:n=ell}
$\begin{array}{c|cccc}
  d/R & S& D& E\\
\hline
 a\ell& \alpha & \beta & -\\
 a\ell+1 & \ell(r_0r+e/r)& - & -\\
 a+1 & -& -& r_1r \\
 a+2 & -& -& 0\\
\end{array}$
\end{table}

\begin{proof}
For $W=G(e,r,\ell)$ the set $\fF$ defined after Lemma~\ref{l:weightdescent}
consists of the two functions $f_S$ and $f_D$ defined before
Proposition~\ref{p:out} and one other function given by
$$f(\gamma,\fc)=\begin{cases} 1& \mbox{ if $\gamma = 1, \fc=()$,}\\
                  0 & \mbox{ otherwise.} \end{cases}$$

Set $E:=R_{0,1,1,()}$. Then $E$ is a basic subgroup in the sense of
\cite[Sect.~4]{AF90}. From that reference, and using the notation of
Section~ \ref{subsec:GGr weights}, we see that $E=ZE_0$, where $Z$ is the Sylow
$\ell$-subgroup of $Z(\GL_K(V))$ and $E_0$ is extra-special of order $\ell^3$
and exponent $\ell$. Moreover,
$N_{\GL_K(V)}(E) /EC_{\GL_K(V)}(E)\cong\SL_2(\ell)$ and $N_{\GL_k(V)}(E)$ is an
extension of $N_{\GL_K(V)}(E)/EC_{\GL_K(V)}(E)$ by a cyclic group of order
$e/r$ (here $C_{\GL_k(V)}(E) =C_{\GL_K(V)}(E) = Z(\GL_K(V))\cong K^\times)).$
We deduce that (3) holds. The structure of the $\cF$-outer automorphism groups
in cases (1) and (2) can be read off from Proposition~\ref{p:out}. 
\par
To compute the weights, suppose first that $R=E$. Here we see that every
character of $E$ is either of defect $a+1$ or $a+2$. An argument to show that
$\bw_E(\cF,a+2)=0$ appears in \cite[Lemma~8.7]{KLLS19}. Now every non-linear
irreducible character of $E$ is induced by a character of $Z_2(S)$ which does
not contain $Z(E)\cong C_{\ell^a}$ in its kernel. Counting, we obtain
$(\ell^{a+1}-\ell^a)/\ell$ characters this way. 
The action of $\Out_\cF(E)$ on $\Irr^{a+1}(E)$ partitions this set into
$\ell^{a-1}(\ell-1)r/e$ equally sized orbits (with trivial stabiliser). We have
$|\cN_E|=2$ where $\cN_E$ is as defined in Section \ref{s:ordspets}, and
contributions to $\bw_E(\cF,a+1)$ from the trivial and non-trivial chains are
easily calculated to be $\ell^{a-1}(\ell-1)r/e=r_1r$ and $0$ respectively.

Next suppose that $R=S$ and $d=a\ell+1$.
We identify $\Out_\cF(S) \cong C_{\ell-1} \times C_{e/r}$ with the subgroup of
$\GL_2(\ell)$ generated by the matrices $\diag(1,\omega)$ and
$\diag(\omega^{\frac{(\ell-1)r}{e}},1)$, where $\omega$ is a generator of
$\FF_\ell^\times$. Under this identification, the action of $\Out_\cF(S)$ on
$\Irr^{a\ell+1}(S)$ is the action on column vectors
$(u,v)\in \ZZ/\ell^a\oplus \ZZ/\ell$. A complete set of orbit representatives
is given by $\{(0,0),(u,0),(0,1),(u,1)\}$ where there are exactly $r_0r$ choices
for the non-zero element $u$. The stabilisers of these representatives have
orders $(\ell-1)e/r$, $\ell-1$, $e/r$ and $1$ respectively, so
$$\bw_S(\cF,a\ell+1)=(\ell-1)e/r+(\ell-1)r_0r+e/r+r_0r=\ell(r_0r+e/r).$$

It remains to prove the last statement. Set $d:=a\ell$ and under the
identification of $W$ with $\Aut_\cF(D)$, set $N:=N_W(\Aut_S(D))$. By Clifford
theory, $\Ind_D^S$ is a surjective map from
$\{\chi\in\Irr(D) \mid Z(S) \nleq \ker(\chi)\}$ onto $\Irr^d(S)$. Since $\cF$
is saturated, the natural map $\Out_\cF(S) \rightarrow N$ is surjective, and if
$\widehat\vhi$ denotes the preimage of $\vhi\in N$ under this map then
calculating we get
$$\Ind_D^S(\chi^\vhi)(g)=\sum_{i=0}^{\ell-1} \chi^\vhi(x^{-i}gx^i)
  =\chi(x^{-i}g\vhi x^i)=\Ind_D^S(\chi)^{\widehat\vhi}(g)
  \quad\text{for $g\in S$}.$$
This shows that induction induces a stabiliser-preserving bijection between
$N$-orbits on $\{\chi\in\Irr(D)\mid Z(S)\nleq\ker(\chi)\}$ and
$\Out_\cF(S)$-orbits on $\Irr^d(S)$. Since $z(k(I_N(\chi)))=0$
whenever $Z(S) \le \ker(\chi)$, for $\chi\in \Irr(D)$, we obtain
$$\begin{aligned}
 \alpha+\beta &= \bw_D(\cF,d)+\bw_S(\cF,d)\\
  &= \!\!\sum_{\chi \in \Irr^d(D)/W}\!\!\! z(k(I_{W}(\chi)))
   -\!\!\!\sum_{\chi \in \Irr^d(D)/N}\!\!\!\! z(k(I_{N}(\chi)))
   + \!\!\!\!\sum_{\chi \in \Irr^d(S)/\Out_\cF(S)}\!\!\!\!\!\! z(k(I_{\Out_\cF(S)}(\chi))) \\
  &= \sum_{\chi \in \Irr^d(D)/W} z(k(I_{W}(\chi))),
\end{aligned}$$
as needed.
\end{proof}

We need the following combinatorial facts: 

\begin{lem}   \label{lem:comb}
 The following hold:
 \begin{itemize}
  \item[(a)] $|\Irr(\fS_\ell)|-z(k\fS_\ell)=\ell$.
  \item[(b)] $|\Irr(W)|-z(kW)$ = $e\ell/r$ for $W=G(e,r,\ell)$ with
   $r|e|(\ell-1)$. 
 \end{itemize}
\end{lem}

\begin{proof}
Part~(a) counts the partitions of $\ell$ with an $\ell$-hook, and these are
precisely the hooks. For (b) we use the parametrisation of $\Irr(G(e,1,\ell))$
by $e$-tuples of partitions $\la$ of $\ell$. The characters not of
$\ell$-defect zero are then precisely those for which one of the parts $\la_i$
of $\la$ has an $\ell$-hook. This can only happen when $|\la_i|=\ell$, and then
by~(a) there are exactly $\ell$ such. This shows (b) in the case $r=1$. For
general $r$ the linear characters of $G(e,1,\ell)/G(e,r,\ell)\cong C_r$ act on
the parametrising $e$-tuples of partitions by cyclic shift. A character with a
non-trivial stabiliser thus corresponds to an $e$-tuple invariant under a
non-trivial cyclic shift, of order $1<d|e$. But then $d$ divides $\ell$, which
is not possible as $e|(\ell-1)$. So by Clifford theory, we have
$$|\Irr(G(e,r,\ell))|=\frac{1}{r}|\Irr(G(e,1,\ell))|\quad\text{and}\quad
  z(kG(e,r,\ell))=\frac{1}{r}z(kG(e,r,\ell)),$$
from which the result follows. 
\end{proof}

\begin{prop}   \label{p:gerl}
 Conjecture~\ref{conj:nr oweights} holds for $W=G(e,r,\ell)$, with $r|e|(\ell-1)$
 in the case $q\equiv 1\pmod\ell$.
\end{prop}

\begin{proof}
We discuss the occurring defects $d$ in turn, and repeatedly use
Lemma~\ref{l:identifyws2} to identify the groups $W(s)_{\phi_s}$ in
Conjecture~\ref{conj:nr oweights}.
Suppose first that $d=a\ell+1$. If $s=1$ then $W(s)_{\phi_s}=W$ and
$|\Irr^1(W)|=\ell e/r$ by Lemma~\ref{lem:comb}(b). Next assume $1\neq s\in Z(S)$
so that $W(s)_{\phi_s}=\fS_\ell$. When $r=1$ there are
exactly $r_0$ classes of such elements by Lemma~\ref{lem:centrGL}, and in
general there are $r_0r$ classes by Theorem~\ref{thm:centr GGR}. Now
$|\Irr^1(W(s)_{\phi_s})|=\ell$ by Lemma~\ref{lem:comb}(a), and we obtain
$$\sum_{s\in Z(S)/\cF} |\Irr^1(W(s)_{\phi_s})|=\ell e/r+ \ell r_0r,$$ 
in accordance with Table~\ref{tab:n=ell}.

If $d=a+1$ then by Lemma \ref{lem:centrGL} it suffices to consider classes of
elements $s$ for which $a_f=1$ for some $f \in \fF_1$ (so that
$\nu_\ell(|C_S(s)|)=a+1$). By Theorem~\ref{thm:centr GGR} such classes split
into $r$ distinct $\cF^{(r)}$-classes and since here $W(s)_{\phi_s}=W(s)=1$, we
obtain $r_1r$ characters altogether, as needed by Table~\ref{tab:n=ell}.

The only occurring defects are $a+1,a\ell+1$ and $a\ell$, so to complete the
proof it suffices to show that
\begin{equation}   \label{e:m=irr}
  \bm(\cF) = \sum_{s\in S/\cF} |\Irr(W(s)_{\phi_s})|.
\end{equation}
Applying Lemmas \ref{l:n=lcentrads} and \ref{lem:comb}, and using the fact that
$\ell \nmid |W(s)_{\phi_s}|$ whenever $s\in D\backslash Z(S)$, the right hand side of
(\ref{e:m=irr}) is equal to
$$\begin{array}{rcl}
 && \displaystyle |\Irr(W)|+ \sum_{1 \neq s \in Z(S)/\cF} |\Irr(W(s)_{\phi_s})| + \sum_{s \in D \backslash Z(S)/\cF} |\Irr(W(s)_{\phi_s})|+ \sum_{s\in S \backslash D / \cF} |\Irr(W(s)_{\phi_s})| \\
 &=& |\Irr(W)|+r_0r|\Irr(\fS_\ell)| + \bm(\cF,a\ell)-\displaystyle\sum_{s\in Z(S)/\cF} z(kI_W(s)) + r_1r\\
 &=& |\Irr(W)|-z(kW)+r_0r|\Irr(\fS_\ell)|-r_0rz(k\fS_\ell) + \bm(\cF,a\ell) + r_1r\\\\
 &=& e\ell/r+r_0r\ell+\bm(\cF,a\ell) + r_1r =\bm(\cF),
\end{array}$$
as needed.
\end{proof}

Since Conjecture \ref{c:spetsandcomp} holds for $W=G(e,r,\ell)$ by
Lemma~\ref{l:identifyws2}, Proposition~\ref{p:gerl} implies
Conjecture~\ref{c:spetsowc} holds for the $\ZZ_\ell$-spets associated to $W$ in
the case $q\equiv 1\pmod\ell$ (see the remark immediately after the statement of
Conjecture~\ref{conj:nr oweights}).

%%%%%%%%%%%%%%%%%%%%%%%%%%%%%%%%%%%%%%%%%%%%%%%%%%%%%%%%%%%%%%%%%%%%%%%%%
\section{Exceptional and Aguad\'{e} groups}   \label{sec:except}
In this section we prove our results from Sections~\ref{sec:awc}
and~\ref{sec:owc} for the exceptional Weyl groups and the four Aguad\'e
exotic $\ZZ_\ell$-reflection groups.

%%%%%%%%%%%%%%%%%%%%%%%%%%%%%%%%%%%%
\subsection {Proof of Theorem~\ref{thm:nr weights} for exceptional and Aguad\'e groups}   \label{subsec:awc exc}
Let $(W,L)$ be a $\ZZ_\ell$-reflection group of type $E_n$, $n\in\{6,7,8\}$
with associated $\ell$-compact group $X$, $\ell$ a good prime dividing $|W|$,
and $q$ a prime power. By the reductions made after its statement, in order to
prove Theorem~\ref{thm:nr weights} for $X(q)$ we may assume $q\equiv1\pmod\ell$
and $\tau=1$.

\begin{lem}
 Let $(W,L)$, $X$, $\ell$ and $q$ be as above.
 Then the associated fusion system $\cF:=\cF(X(q))$ satisfies:
 \begin{itemize}
  \item  [(a)] $\cF=\cF_\ell(E_n(q))$ is the $\ell$-fusion system of $E_n(q)$
   on one of its Sylow $\ell$-subgroups $S$.
  \item [(b)] For some $\sigma\in S$ we have $S=T\langle\sigma\rangle$ for
   $T\le S$ a homocyclic subgroup of index $\ell$ and rank~$n$.
  \item [(c)] $\cF^\CR=\{S,T,E\}$ where $E=Z_2(S)\langle\sigma\rangle$.
   The $\cF$-automisers and number of weights are as given in
   Table~\ref{tab:except}.
  \item[\rm(d)] Theorem~\ref{thm:nr weights} holds for $\cF$.
 \end{itemize}
\end{lem}
  
\begin{table}[ht]
\caption{$\Out_\cF(P)$ for $P\in\cF^\CR$, and $z(k\Out_\cF(P))$ (in bold)}  \label{tab:except}
$\begin{array}{c|cccccccc|c}
 (n,\ell) & S& z&& T& z&& E& z& |\Irr(W)|\\ \hline
 (6,5) & C_{4} \times C_2& \bf{8}&& W& \bf{15}&& \SL_2(5).2& \bf{2}& 25\\
 (7,5) &  C_4 \times C_2 \times \fS_3& \bf{24}&& W& \bf{30}&& \SL_2(5).2 \times \fS_3& \bf{6}& 60\\
 (7,7) &  C_6 \times C_2& \bf{12}&& W& \bf{46}&& \SL_2(7).2 & \bf{2}& 60\\
 (8,7) &  C_6 \times (C_2)^2 & \bf{24}&& W& \bf{84}&& \SL_2(7).2 \times C_2& \bf{4}& 112\\
\end{array}$
\end{table}

\begin{proof}
Part (a) may be deduced from the description of the ``sporadic cases'' described
in \cite[p.~1821]{BM07}. The $\cF$-centric radical subgroups are described in
\cite[Table 2.2]{COS17}, and their structure and $\cF$-automisers can be
determined from the results in \cite[Sec.~2]{COS17}.
Theorem~\ref{thm:nr weights} for $\cF$ now follows from Table~\ref{tab:except}.
\end{proof}

Now, let $(W,L)$ be an exotic $\ZZ_\ell$-reflection group of order divisible
by $\ell>2$. Then $W$ is one of $G_{12}$, $G_{29}$, $G_{31}$ or $G_{34}$, with
$\ell=3,5,5$ or~$7$ respectively. The associated $\ell$-compact groups were
constructed by
Aguad\'{e} \cite{A89}. In the first and third cases, these arise as
$\ell$-completed classifying spaces of compact Lie groups of type $\tw2F_4$ and
$E_8$, while the other cases are exotic (see \cite[Sec.~1]{BM07}). Let
$q \equiv 1 \pmod \ell$ be a prime power; we also assume that
$a:=\nu_\ell(q-1)>1$ when $n=12$. For $n\in\{12,29,31,34\}$ let $X_n$ denote
the $\ell$-compact group associated to $G_n$ and $\cG_n=\cG_n(q)=\cF(X_n(q))$
denote the corresponding fusion system with $\ell$ as indicated.

Let $\cG_n'$ be the fusion system on a Sylow $\ell$-subgroup~$S_n$ of
$\SL_\ell(q)$. So $|\widehat{S}_n:S_n|=\ell^a$, where
$\widehat{S}_n=\widehat{T}_n\langle\sigma\rangle\cong C_{\ell^a}\wr C_\ell$
is a Sylow $\ell$-subgroup of $\GL_\ell(q)$ and
$\widehat{T}_n=\widehat{S}_n\cap\GL_1(q)^\ell$. Note, in particular, that $S_n$
is a group of maximal class. 

\begin{lem}   \label{lem:essential}
 Let $(W,L)$, $\ell$ and $q$ be as above. Then for $n\in\{12,29,31,34\}$, the
 fusion system $\cG_n$ associated to $X_n$ satisfies:
 \begin{itemize}
  \item[(a)] $\cG_n$ is a simple fusion system containing $\cG_n'$ as a
   subsystem (in fact the centraliser of the centre of $S_n$).
  \item[(b)] We have $S_n=T_n\langle\sigma\rangle$ where
   $T_n:=S_n\cap\hat{T}_n \cong C_{\ell^a}^{\ell-1}$.
  \item[(c)] $\cG_n^\CR/\cG_n=\{S_n,T_n,D_n,E_n\}$, where
   $D_n\cong E_n=Z_2(S_n)\langle\sigma\rangle\cong \ell_+^{1+2}$. The
   $\cG_n$-automisers and number of weights are as given in
   Table~\ref{tab:Aguade}.
  \item[\rm(d)] Theorem~\ref{thm:nr weights} holds for $\cG_n$.
 \end{itemize}
\end{lem}

\begin{table}[h]
\caption{$\Out_{\cG_n}(R)$ and $z(k\Out_{\cG_n}(R))$ for $R\in\cG_n^\CR$}   \label{tab:Aguade}
$\begin{array}{c|cccc|c}
   & S_n& T_n& D_n& E_n& |\Irr(G_n)|\\
\hline
 \Out_{\cG_n}(R)& C_{\ell-1} \times C_{\ell-1}& G_n& \SL_2(\ell)&
 \GL_2(\ell)\\
 z(k\Out_{\cG_n}(R))\ (n=12)&  4&   2& 1& 2& 8\\
 z(k\Out_{\cG_n}(R))\ (n=29)& 16&  17& 1& 4& 37\\
 z(k\Out_{\cG_n}(R))\ (n=31)& 16&  39& 1& 4& 59\\
 z(k\Out_{\cG_n}(R))\ (n=34)& 36& 127& 1& 6& 169\\
\end{array}$
\end{table}

\begin{proof}
See \cite[Sec.~10]{BM07} for (a)--(c). Part~(d) then follows from
Table~\ref{tab:Aguade}.
\end{proof}

In the case $a=1$ for $W=G_{12}$, apart from $\{S\cong 3^{1+2}_+\}$ there are
two $\cG_{12}$-classes of centric radical subgroups represented by $V_1,V_2$
with $V_i \cong C_3^2$ and $\Out_{\cG_{12}}(V_i)\cong\GL_3(2)$. Since
$\Out_{\cG_{12}}(S)\cong D_8$, we find
$$\bw(\cG_{12})
  =z(k\Out_{\cG_{12}}(S))+z(k\Out_{\cG_{12}}(V_1))+z(k\Out_{\cG_{12}}(V_2))
  =5+2+2=|\Irr(W)|+1.$$

%%%%%%%%%%%%%%%%%%%%%%%%%%%%%%%%%%%%
\subsection {The ordinary weight conjecture for Aguad\'e groups}   \label{subsec:8.2}
Here we show Proposition~\ref{prop:owc Aguade} for the Aguad\'{e} groups.
For $n\in\{12,29,31,32\}$ we set $\cG_n:=\cG_n(q)=\cF(G_n(q))$, where
$q\equiv1\pmod\ell$.

\begin{lem}   \label{lem:aguowc}
 For $d\ge 0$ and $R$ a $\cG_n$-centric radical subgroup, $\bw_R(\cG_n,d)$ is
 given in Table \ref{tab:ordaguade}, where ``$-$'' indicates that there are
 no characters of that defect. Moreover
 $$\bm(\cG_n,a(\ell-1))=\alpha+\beta
   =\sum_{\chi \in \Irr(T_n)/G_n} z(kI_{G_n}(\chi)).$$
%% where $\alpha:=\bw_S(\cF,(l-1)t)$ and $\beta:=\bw_T(\cF,(l-1)t)$.
\end{lem}

\begin{table}[h]
\caption{$\bw_R(\cG_n,d)$ for $R\in\cG_n^\CR$}   \label{tab:ordaguade}
$\begin{array}{c|ccccc}
  d& S_n& T_n& D_n & E_n\\
\hline
 (\ell-1)a& \alpha & \beta & -& -\\
 (\ell-1)a+1 & \ell^2 & - & - & - \\
 2 & -& -& \ell-1 & 1\\
 3 & -& -& 0 & 0\\
 \end{array}$
\end{table}

\begin{proof}
This is analogous to the argument given in the proof of
Lemma~\ref{l:n=lcentrads}. In particular if $R\in \{D_n,E_n\}$ we apply that
argument in the case $R=E$ and $a=1$ to conclude that
$$\bw_R(\cG_n,3)=0,\ \mbox{ and }\ \bw_R(\cG_n,2)=\begin{cases} \ell-1&
   \mbox{ if $R=D_n$,}\\ 1& \mbox{ if $R=E_n.$ } \end{cases}
$$
If $R=S_n$ then since $R$ has maximal class, it has $\ell^2$ linear characters
and the action of $\Out_{\cG_n}(S_n)$ on $\Irr^{a(\ell-1)+1}(S_n)$ is the
natural action of the subgroup of $\GL_2(\ell)$ generated by the matrices
$\diag(1,\omega)$ and $\diag(\omega,1)$ where $\omega$ is a generator of
$\FF_\ell^\times$. The column vectors $\{(0,0),(1,0),(0,1),(1,1)\}$ form a
complete set of orbit representatives for this action, so
$$\bm(\cG_n,a(\ell-1)+1)=(\ell-1)^2+2(\ell-1)+1 = \ell^2.$$
Finally, the argument to prove the last statement is analogous to that given in
Lemma \ref{l:n=lcentrads}, so we omit this.
\end{proof}

\begin{proof}[Proof of Proposition~\ref{prop:owc Aguade}]
For a $\cG_n$-class represented by $1\neq s \in S_n$, let $G_{n,s}:=(G_n)(s)$
be the centraliser Weyl group from Proposition~\ref{prop:centr}. These
are described in \cite[Prop.~10.1]{BM07}: if $s\notin Z(S_n)$, $|G_{n,s}|$ is
prime to $\ell$, and otherwise $G_{n,s}=\fS_\ell$. If $d=(a-1)\ell+1$ then by
Lemmas~\ref{lem:comb} and~\ref{lem:aguowc},
$$\sum_{s\in Z(S)/\cF} |\Irr^1(G_{n,s})|
  =|\Irr^1(G_n)|+|\Irr^1(\fS_\ell)|=\ell(\ell-1)+\ell= \bm(\cG_n,d),$$
as expected. Now, again by Lemmas~\ref{lem:comb} and~\ref{lem:aguowc},
$$\begin{array}{rcl}
\displaystyle\sum_{s \in S/\cG_n} |\Irr(G_{n,s})|&=& \displaystyle |\Irr(G_n)|+ \sum_{1 \neq s \in Z(S_n)/\cG_n} |\Irr(G_{n,s})| + \sum_{s \notin Z(S_n)/\cG_n} |\Irr(G_{n,s})| \\
&=& |\Irr(G_n)|+|\Irr(\fS_\ell)|+\bm(\cG_n,a(\ell-1))
 -\displaystyle\sum_{s \in Z(S_n)/\cG_n} z(kG_{n,s}) \\
&=& |\Irr(G_n)|-z(kG_n)+ |\Irr(\fS_\ell)|-z(k\fS_\ell)+\bm(\cG_n,a(\ell-1)) \\\\
&=& \ell(\ell-1)+\ell+\bm(\cG_n,a(\ell-1)) = \bm(\cG_n)-\bm(\cG_n,2).
\end{array}$$
Finally, note that when $n=12$ and $a=1$, we have $S_{12} \cong 3^{1+2}_+$ and
$\bm(\cG_{12},3)=9=|\Irr^1(G_{12})|+|\Irr^1(\fS_3)$ (see \cite[Tab.~2]{KLLS19}.)
\end{proof}

%%%%%%%%%%%%%%%%%%%%%%%%%%%%%%%%%%%%%%%%%%%%%%%%%%%%%%%%%%%%%%%%%%%%%%%%%
\section{Appendix: Weights for wreath products}

Here we prove a result which was used in Section~\ref{subsec:GGr weights}.

\begin{lem}
 Let $G= N\wr \fS_n$ be a finite group. There is a natural bijection
 \[\Psi:\Irr^0(G)\to\cW\]
 between the set $\Irr^0(G)$ of $\ell$-defect zero characters of $G$ and the
 set of functions 
 \[ \cW:=\Big\{w :\Irr^0(N)\to\{\ell\text{-cores}\}\quad\text{with}\quad
   \sum_{\vhi\in\Irr^0(N)} | w (\vhi) | =n\Big\}.\]
\end{lem}

See \cite[Sec.~4]{AF90}. The map $\Psi$ is defined as follows: Let
$\chi\in\Irr^0(G)$, $\tau\in\Irr(B)$ with $\chi\in\Irr(G|\tau)$, where
$B = N^n$ is the base subgroup of $G$. Let $H$ be the stabiliser of $\tau$ in
$G$ and let $\chi'\in\Irr(H|\tau)$ such that $\chi =\Ind_H^G(\chi')$. 
Then $\tau\in\Irr^0(B)$ and $\chi'\in\Irr^0(H)$. Let
$\tau=\tau_1\otimes\cdots\otimes\tau_n$ with $\tau_j\in\Irr^0 (N)$ and for each
$\vhi\in\Irr^0(N)$, denote by $n_{\chi,\vhi}= n_\vhi $ the number
of $j$ such that $\vhi\cong\tau_j$. Then
\[ H\cong\prod_{\vhi\in\Irr^0(N)} N\wr\fS_{n_\vhi}, \]
and we have an identification
\begin{equation}   \label{e:iota}
  \Irr^0(H|\tau) = \prod_{\vhi\in\Irr^0(N)}\Irr^0(N\wr\fS_{n_\vhi}|\vhi \otimes\cdots\otimes\vhi).
\end{equation}
For $\vhi\in\Irr^0(N)$, let $\hat\vhi$ be the canonical extension to
$N\wr\fS_{n_\vhi}$ of $\vhi^{\otimes n}$, i.e., $\hat\vhi$ is the
tensor induced character $\vhi^{\otimes^{n_\vhi}}$. Tensoring with $\hat\vhi$
induces a bijection
\[\iota_\vhi:\Irr^0(\fS_{n_\vhi})\to\Irr^0(N\wr\fS_{n_\vhi}|\vhi\otimes\cdots\otimes\vhi) .\]
Suppose that $\chi'$ corresponds to $\prod_\vhi\chi'_\vhi$ via (\ref{e:iota}).
Then $\Psi(\chi)(\vhi) = \kappa_\vhi $, where $\kappa_\vhi $ is the $\ell$-core
partition of $n_\vhi $ labelling $\iota_\vhi^{-1} (\chi'_\vhi)$. 
\medskip

Now let $\la$ be a linear character of $G':=G/([G,G]\fS_n)$ and let $\chi$,
$\tau$, $H$ and $\chi'$ as above. Since $\fS_n$ is in the kernel of $\la $,
$\la |_B = \la_0^{\otimes n}$ for some $\la_0\in\Irr(N)$. The character
$\lambda.\chi$ covers $\la|_{B_0}.\tau$, $H =\Stab_G(\la|_{B_0}.\tau)$ and for
all $\vhi\in\Irr^0(N)$, we have
$n_{\la.\chi,\la_0.\vhi} = n_{\chi,\vhi}=n_\vhi$.
Further, $\widehat{\la_0\vhi} =\la |_{N\wr\fS_{n_\vhi}}.\hat\vhi$, hence
for any $\kappa\in\Irr(\fS_{n_\vhi})$, 
\[ i_{\la_0.\vhi} (\kappa) = \widehat {\la_0\vhi}\otimes\kappa
  = \la |_{N\wr\fS_{n_\vhi}} . i_\vhi (\kappa) .\]
Since
\[\la\cdot\chi = \la\cdot\Ind_H^G(\chi') = \Ind_H^G(\la |_H\cdot\chi')
  = \prod_{\vhi\in\Irr^0(N)}\la |_{N\wr\fS_{n_\vhi}}\cdot\chi'_\vhi\]
we obtain
\begin{equation}   \label{eq:bij}
  \Psi (\la\cdot\chi) (\la_0\cdot\vhi) = \Psi(\chi).
\end{equation}
Let $\Irr (G')$ act on $\Irr^0 (G)$ via $\chi\mapsto\la .\chi$
(where $\chi\in\Irr^0(G)$, $\la\in\Irr(G')$), and on $\cW$ via
\[ \la. w (\vhi):= w (\la_0^{-1}\cdot\vhi),\qquad\text{where}\quad w\in\cW,
   \ \lambda\in\Irr(G'),\ \vhi\in\Irr^0(N). \]
By Equation~(\ref{eq:bij}), $\Psi$ is an isomorphism of
$\Irr(G')$-sets. 

The above discussion carries over to direct products
$G = \prod_{i\in I} N_i\wr\fS_{n_i}$. Set
$$\cW:=\Big\{w :\bigcup_{i\in I}\Irr^0(N_i)\to\{\ell\text{-cores}\}\quad
  \text{with}\quad \sum_{\vhi\in\Irr^0(N_i)} | w(\vhi) | =n_i\Big\}$$
and let $M$ be a subgroup of $G$ containing $[G,G]$ and all $\fS_{n_i}$,
$i\in I$. If $\la\in\Irr(G/M)$, the restriction of $\la$ to
$\prod_{i\in I} N_i^{n_i}$ is of the form $\prod_{i\in I}\la_i^{\otimes n_i}$
for some $\la_i\in\Irr (N_i)$. The group $\Irr(G/M)$ acts on $\cW$ via
\[\la. w_i(\vhi) = w(\la_i^{-1}\cdot\vhi)\quad
  \text{for }\lambda\in\Irr(G/M),\ w\in\cW,\ i\in I,\ \vhi\in\Irr^0(N_i). \] 
Then the $\Irr (G/M)$-sets $\Irr^0(G)$ and $\cW$ are isomorphic where as before
the $\Irr(G/M)$-action on $\Irr^0(G)$ is via multiplication. The Clifford
theory of cyclic extensions gives the following immediate consequence of the
above:

\begin{prop}   \label{p:wreath}
 With the above notation, suppose that $G/M$ is a cyclic $\ell'$-group. Then
 $\Irr^0(M)$ is indexed by the set of $D:=\Irr(G/M)$-orbits of $\cW$ in such a
 way that the $D$-orbit of $w\in\cW$ corresponds to $|\Stab_D(w)|$ elements of
 $\Irr^0(M)$.
\end{prop}

%%%%%%%%%%%%%%%%%%%%%%%%%%%%%%%%%%%%%%%%%%%%%%%%%%%%%%%%%%%%%%%%%%%%%%%%%


\begin{thebibliography}{99}

\bibitem{A89}
{\sc J. Aguad\'{e}}, Constructing modular classifying spaces. \emph{Israel J.
  Math. \bf66} (1989), 23--40.

\bibitem{A87}
{\sc J. Alperin}, Weights for finite groups. \emph{The Arcata Conference on
  Representations of Finite Groups (Arcata, Calif., 1986)}, 369--379, Proc.
  Sympos. Pure Math., 47, Amer. Math. Soc., Providence, RI, 1987.

\bibitem{AF90}
{\sc J. Alperin and P. Fong}, Weights for symmetric and general linear groups.
  \emph{J. Algebra \bf131} (1990), 2--22.

\bibitem{An94}
{\sc J. An}, Weights for the Chevalley groups $G_2(q)$. \emph{Proc. Lond. Math.
  Soc. \bf69} (1994) 2--46.

\bibitem{AG09}
{\sc K. Andersen and J. Grodal}, The classification of $2$-compact groups.
  \emph{J. Amer. Math. Soc. \bf22} (2009), 387--436.

\bibitem{AGMV}
{\sc K.K.S. Andersen, G. Grodal, J.M. M\o ller, and A. Viruel}, The
  classification of $p$-compact groups for $p$ odd. \emph{Ann. of Math. (2)
  \bf 167} (2008), 95--210.

\bibitem{AOV12}
{\sc K. Andersen, B. Oliver, and J. Ventura}, Reduced, tame and exotic fusion
  systems. \emph{Proc. Lond. Math. Soc. \bf105} (2012), 87--152.

\bibitem{AKO11}
{\sc M. Aschbacher, R. Kessar, and B. Oliver}, \emph{Fusion Systems in Algebra
  and Topology}. London Mathematical Society Lecture Note Series, 391, Cambridge
  University Press, Cambridge, 2011.

\bibitem{Bo05}
{\sc C. Bonnaf\'e}, Quasi-isolated elements in reductive groups. \emph{Comm.
  Algebra \bf33} (2005), 2315--2337.

\bibitem{BCGLO07}
{\sc C. Broto, N. Castellana, J. Grodal, R. Levi, and B. Oliver}, Extensions of
  $p$-local finite groups. \emph{Trans. Amer. Math. Soc. \bf359} (2007),
  3791--3858.

\bibitem{BLO03}
{\sc C. Broto, R. Levi, and B. Oliver}, The homotopy theory of fusion systems.
  \emph{J. Amer. Math. Soc. \bf16} (2003), 779--856.

\bibitem{BM07}
{\sc C. Broto and J. M\o ller}, Chevalley $p$-local finite groups. \emph{Algebr.
  Geom. Topol. \bf7} (2007), 1809--1919.

\bibitem{BMO}
{\sc C. Broto, J. M\o ller, and B. Oliver}, Equivalences between fusion systems
  of finite groups of Lie type. \emph{J. Amer. Math. Soc. \bf25} (2012), 1--20.

\bibitem{BM02}
{\sc C. Broto, J. M\o ller, and B. Oliver}, Automorphisms of fusion systems of
  finite simple groups of Lie type. \emph{Memoirs Amer. Math. Soc. \bf1267}
  (2019), 1--117.

\bibitem{BCM}
{\sc M. Brou\'e, R. Corran, and J. Michel}, Cyclotomic root systems and bad
  primes. \emph{Adv. Math. \bf325} (2018), 375--458.

\bibitem{BM92}
{\sc M. Brou\'e and G. Malle}, Th\'eor\`emes de Sylow g\'en\'eriques pour les
  groupes r\'eductifs sur les corps finis. \emph{Math. Ann. \bf292} (1992),
  241--262.

\bibitem{BMM93}
{\sc M. Brou\'e, G. Malle, and J. Michel}, Generic blocks of finite reductive
  groups. \emph{Ast\'erisque} No.~212 (1993), 7--92.

\bibitem{BMM99}
{\sc M. Brou\'e, G. Malle, and J. Michel}, Towards spetses {I}. \emph{Transform.
  Groups \bf4} (1999), 157--218.

\bibitem{BMM14}
{\sc M. Brou\'e, G. Malle, and J. Michel}, Split spetses for primitive
  reflection groups. \emph{Ast\'erisque} No.~359 (2014).

\bibitem{BM15}
{\sc O. Brunat and G. Malle}, Characters of positive height in blocks of finite
  quasi-simple groups. \emph{Int. Math. Res. Not. IMRN \bf17} (2015), 7763--7786.

\bibitem{CE94}
{\sc M. Cabanes and M. Enguehard}, On unipotent blocks and their ordinary
  characters. \emph{Invent. Math. \bf117} (1994), 149--164. 

\bibitem{COS17}
{\sc D. Craven, B. Oliver, and J. Semeraro}, Reduced fusion systems over
  $p$-groups with abelian subgroup of index $p$: II. \emph{Adv. Math. \bf322}
  (2017), 201--268. 

\bibitem{Ch13}
{\sc A. Chermak}, Fusion systems and localities. \emph{Acta Math. \bf211}
  (2013), 47--139.

\bibitem{DW94}
{\sc W. G. Dwyer and C. W. Wilkerson}, Homotopy fixed-point methods for Lie
  groups and finite loop spaces. \emph{Ann. of Math. (2) \bf139} (1994),
  395--442.

\bibitem{DW95}
{\sc W. G. Dwyer and C. W. Wilkerson}, The center of a $p$-compact group.
  \emph {The \v Cech Centennial} (Boston, MA, 1993), Contemp. Math. {\bf 181},
  119--157, Amer. Math. Soc., Providence, RI, 1995.

\bibitem{En00}
{\sc M. Enguehard}, Sur les $l$-blocs unipotents des groupes r\'eductifs finis
  quand $l$ est mauvais. \emph{J. Algebra \bf230} (2000), 334--377. 

\bibitem{Fe19}
{\sc Z. Feng}, The blocks and weights of finite special linear and unitary
  groups. \emph{J. Algebra \bf523} (2019), 53--92. 

\bibitem{FM84}
{\sc E. M. Friedlander and G. Mislin}, Cohomology of classifying spaces of
  complex Lie groups and related discrete groups. \emph{Comment. Math. Helv.
  \bf59} (1984), 347--361.

\bibitem{Ge93}
{\sc M. Geck}, Basic sets of Brauer characters of finite groups of Lie type II.
  \emph{J. Lond. Math.Soc. \bf47} (1993), 225--268. 

\bibitem{GH91}
{\sc M. Geck and G. Hiss}, Basic sets of Brauer characters of finite groups of
  Lie type. \emph{J. Reine Angew. Math. \bf418} (1991), 173--188. 

\bibitem{GM06}
{\sc M. Geck and G. Malle}, Reflection groups. \emph{Handbook of Algebra.
  Vol. 4}, 337--383. Elsevier/North-Holland, Amsterdam, 2006. 

\bibitem{GM20}
{\sc M. Geck and G. Malle}, \emph{The Character Theory of Finite Groups of Lie
  Type: A Guided Tour}. Cambridge Studies in Advanced Mathematics, 187,
  Cambridge University Press, Cambridge, 2020.

\bibitem{GL20}
{\sc J. Grodal and A. Lahtinen}, String topology of finite groups of Lie type.
  Preprint. arXiv:2003.07852, 2020.

\bibitem{Hi11}
{\sc F. Himstedt}, On the decomposition numbers of the Ree groups $\tw2F_4(q^2)$
  in non-defining characteristic. \emph{J. Algebra \bf325} (2011), 364--403. 

\bibitem{KLLS19}
{\sc R. Kessar, M. Linckelmann, J. Lynd, and J. Semeraro}, Weight conjectures
  for fusion systems. \emph{Adv. Math. \bf357} (2019), 106825.

\bibitem{LSp99}
{\sc G. I. Lehrer and T. A. Springer}, Reflection subquotients of unitary
  reflection groups. \emph{Canad. J. Math. \bf51} (1999), 1175--1193.

\bibitem{LS17}
{\sc J. Lynd and J. Semeraro}, Weights in a Benson--Solomon block. Preprint.
  arXiv:1712.02826, 2017.

\bibitem{Ma90}
{\sc G. Malle}, Die unipotenten Charaktere von $^2F_4(q^2)$. \emph{Comm.
  Algebra \bf18} (1990), 2361--2381.

\bibitem{Ma95}
{\sc G. Malle}, Unipotente Grade imprimitiver komplexer Spiegelungsgruppen.
  \emph{J. Algebra \bf177} (1995), 768--826.

\bibitem{MaICM}
{\sc G. Malle}, Spetses. \emph{Doc. Math. Extra Vol. ICM II} (1998), 87--96. 

\bibitem{Ma00}
{\sc G. Malle}, On the generic degrees of cyclotomic algebras. \emph{Represent.
  Theory \bf4} (2000), 342--369.

\bibitem{MT}
{\sc G. Malle and D. Testerman}, \emph{Linear Algebraic Groups and Finite Groups
  of Lie Type}. Cambridge Studies in Advanced Mathematics, 133. Cambridge
  University Press, Cambridge, 2011.

\bibitem{Mo96}
{\sc J. M\o ller}, Rational isomorphisms of $p$-compact groups. \emph{Topology
  \bf35} (1996), 201--225.

\bibitem{Mo02}
{\sc J. M\o ller}, $N$-determined $p$-compact groups. \emph{Fund. Math. \bf173}
  (2002), 201--300.

\bibitem{MN}
{\sc J. M\o ller and D. Notbohm}, Connected finite loop spaces with maximal
  tori. \emph{Trans. Amer. Math. Soc. \bf350} (1998), 3483--3504.

\bibitem{N99}
{\sc D. Notbohm}, For which pseudo-reflection groups are the $p$-adic
  polynomial invariants again a polynomial algebra? \emph{J. Algebra \bf214}
  (1999), 553--570.

\bibitem{OR}
{\sc B. Oliver and A. Ruiz}, Simplicity of fusion systems of finite simple
  groups. \emph{Trans. Amer. Math. Soc. \bf374} (2021), 7743--7777.

\bibitem{OV07} 
{\sc R. Oliver and J. Ventura}, Extensions of linking systems with $p$-group
  kernel. \emph{Math. Ann. \bf 338} (2007), 983--1043.

\bibitem{Qui72}
{\sc D. Quillen}, On the cohomology and $K$-theory of the general linear groups
  over a finite field. \emph{Ann. of Math. (2) \bf96} (1972), 552--586.

\bibitem{Rob96}
{\sc G. Robinson}, Local structure, vertices and Alperin's conjecture.
  \emph{Proc. London Math. Soc. \bf 72} (1996), 312--330.

\bibitem{Ru07}
{\sc A. Ruiz}, Exotic normal fusion subsystems of general linear groups.
  \emph{J. London Math. Soc. (2) \bf 76} (2007), 181--196.

\bibitem{S19}
{\sc J. Semeraro}, A $2$-compact group as a spets. \emph{Exp. Math. \bf32}
  (2023), 140--155.

\bibitem{StEnd}
{\sc R. Steinberg}, \emph{Endomorphisms of Linear Algebraic Groups.} Memoirs of
  the American Mathematical Society, No. 80, American Mathematical Society,
  Providence, R.I., 1968.
 
\bibitem{St75}
{\sc R. Steinberg}, Torsion in reductive groups. \emph{Adv. Math. \bf15}
  (1975), 63--92.

\end{thebibliography}
\end{document}